\documentclass[author-year, review]{elsarticle}
\usepackage[margin=2.3cm]{geometry}
\usepackage[utf8]{inputenc}
\journal{Transportation Research Part C: Emerging Technologies}

\usepackage[dvipsnames]{xcolor}                        
\usepackage[colorlinks]{hyperref}
\hypersetup{colorlinks=true, linkcolor=black, citecolor=blue, urlcolor=cyan}

\usepackage{multirow}
\usepackage{setspace}

\usepackage{appendix}
\usepackage{hyperref}
\usepackage{float}
\usepackage{subfig}
\usepackage{amsmath}
\usepackage{amsfonts}
\usepackage{amsthm}
\usepackage{mathrsfs}
\usepackage{eurosym}
\usepackage{algorithm}
\usepackage{soul}
\soulregister\cite7
\soulregister\ref7
\soulregister\eqref7
\soulregister\pageref7
\usepackage{cancel}
\usepackage[noend]{algpseudocode}
\usepackage{multirow}
\usepackage{comment}
\usepackage[font={footnotesize}]{caption}

\makeatletter
\def\ps@pprintTitle{%
 \let\@oddhead\@empty
 \let\@evenhead\@empty
 \def\@oddfoot{}%
 \let\@evenfoot\@oddfoot}
\makeatother

\usepackage{csvsimple}
\usepackage{longtable}

\usepackage{booktabs} 
\usepackage{pdflscape}
\newtheorem{theorem}{Theorem}

\newtheorem{proposition}{Proposition}

\usepackage{tikz}
\usetikzlibrary{arrows}
\usetikzlibrary{arrows.meta,bending}
\usetikzlibrary{decorations.markings}
\usetikzlibrary{shapes.geometric}
\usepackage{tikzscale}

\newcommand{\lb}{\textsc{LB-SFRVCP-T}}
\newcommand{\exact}{\textsc{E-SFRVCP-T}}

\newcommand{\norm}[1]{\left\lVert#1\right\rVert}
\DeclareMathOperator*{\argmin}{arg\,min}

\begin{document}

\begin{frontmatter}
	\title{A Stochastic Electric Vehicle Routing Problem under Uncertain Energy Consumption}

\author[1]{Andrea Spinelli}
\author[2]{Dario Bezzi}
\author[3,4]{Ola Jabali}
\author[1]{Francesca Maggioni\corref{mycorrespondingauthor}}
		\ead{francesca.maggioni@unibg.it}
		\cortext[mycorrespondingauthor]{Corresponding author}
\address[1]{Dipartimento di Ingegneria Gestionale, dell'Informazione e della Produzione, Universit\`{a} di Bergamo, Viale G. Marconi 5, Dalmine 24044, Italy}
\address[2]{Dipartimento di Ingegneria dell'Energia Elettrica e dell'Informazione ``G. Marconi'', Universit\`{a} di Bologna, Viale del Risorgimento 2, Bologna 40136, Italy}
\address[3]{Dipartimento di Elettronica, Informazione e Bioingegneria, Politecnico di Milano, Via G. Ponzio 34, Milano 20133, Italy}
\address[4]{HEC Montréal, Montréal, Canada}

\begin{abstract}
The increasing adoption of Electric Vehicles (EVs) for service and goods distribution operations has led to the emergence of Electric Vehicle Routing Problems (EVRPs), a class of vehicle routing problems addressing the unique challenges posed by the limited driving range and recharging needs of EVs. While the majority of EVRP variants have considered deterministic energy consumption, this paper focuses on the Stochastic Electric Vehicle Routing Problem with a Threshold recourse policy (SEVRP-T), where the uncertainty in energy consumption is considered, and a recourse policy is employed to ensure that EVs recharge at Charging Stations (CSs) whenever their State of Charge (SoC) falls below a specified threshold. We formulate the SEVRP-T as a two-stage stochastic mixed-integer second-order cone model, where the first stage determines the sequences of customers to be visited, and the second stage incorporates charging activities. The objective is to minimize the expected total duration of the routes, composed by travel times and recharging operations. To cope with the computational complexity of the model, we propose a heuristic based on an Iterated Local Search (ILS) procedure coupled with a Set Partitioning problem. To further speed up the heuristic, we develop two lower bounds on the corresponding first-stage customer sequences. Furthermore, to handle a large number of energy consumption scenarios, we employ a scenario reduction technique. Extensive computational experiments are conducted to validate the effectiveness of the proposed solution strategy and to assess the importance of considering the stochastic nature of the energy consumption. The research presented in this paper contributes to the growing body of literature on EVRP and provides insights into managing the operational deployment of EVs in logistics activities under uncertainty.
\end{abstract}

		
	\end{frontmatter}

\section{Introduction}
Several nations and local governments have established measures aimed at increasing the commercial use of Electric Vehicles (EVs) and progressively phasing out conventional vehicles. Indeed, sales of electric light commercial vehicle worldwide have almost doubled in 2022 relative to 2021 to more than 310,000 vehicles, making their market share around 3.6\%. Moreover, electric medium- and heavy-duty truck sales totalled nearly 60,000 in 2022, which constitutes 1.2\% of the total number of registrations for that category worldwide (see \cite{iea2023}). 

Logistics companies using EVs must plan the operation of such vehicles considering limited driving ranges and slow charging times. Specifically, planning recharging operations along the EV routes is particularly relevant in the context of the mid- and long-haul logistics, since the driving range of medium-duty EVs remains limited (see~\cite{schiffer2018electric}, \cite{schiffer2021integrated}). The reader is referred to~\cite{juan2016electric} for a survey on various research challenges related to the introduction of EVs in logistics and transportation.

A number of scientific contributions have focused on the operational challenges of using EVs in a logistics context. In this framework, a commonly studied problem is the Electric Vehicle Routing Problem (EVRP, see~\cite{kucukoglu2021electric} for a survey), where a set of customers must be visited by a homogeneous fleet of EVs that start and end their routes at a given depot.
In EVRPs, the limited autonomy of EVs is explicitly modeled and Charging Stations (CSs) are often assumed to be privately owned by the EV operator (see~\cite{froger2019improved}). This latter assumption entails that the charging demand from EVs not owned by the operator may be ignored. Each CS may host a different charging technology (e.g., fast, medium or slow). In the most general EVRP case, charging times are assumed to follow nonlinear charging functions, and partial charging is allowed. This entails that establishing which CS to visit and how much to charge at it are part of the decision variables.

The vast majority of the literature on EVRPs assumes the EV energy consumption to be deterministic and proportional to the traveled distance (see~\cite{lam2022branch}). In reality, energy consumption is not fully predictable, as it may heavily depend on uncertain factors like travel speed, driving style, traffic conditions, temperature and wind speed (see~\cite{asamer2016sensitivity}). As ignoring such uncertainties may lead an EV to get stranded without battery, it is pivotal to account for them when planning EV routes. To this end, in this paper we consider an EVRP with stochastic energy consumption. In order to cope with the uncertainty, we adopt a \textit{threshold recourse policy}. Given a fixed sequence of customers to visit, this policy entails that when the State of Charge (SoC) of the EV falls below a given threshold, the vehicle will immediately detour towards a CS and charge a sufficient amount of energy to bring its SoC at the subsequent customer to a given fixed level. The resulting problem is called the Stochastic Electric Vehicle Routing Problem with a Threshold recourse policy (SEVRP-T), and it has been introduced by  \cite{bezzi2023threshold}. In this problem, CSs with multiple technologies are considered, and nonlinear charging functions are approximated via piecewise linear functions. Notably,  while  detour decisions in the classical EVRP are taken at the nodes, in the SEVRP-T such decisions may be taken at any point while traversing an arc. Indeed, an EV may reach the threshold SoC level while traversing an arc, at which point it will immediately head towards a suitable CS, rather than fully traversing the arc. As the energy consumption is stochastic, the departing  SoC from a node is uncertain as well. Thus, the length of a detour cannot be calculated a priori as part of the input data.

In this paper, we formulate SEVRP-T as a two-stage stochastic mixed-integer second-order cone model, where the first-stage decisions determine the sequence of customers to visit by each EV, while the second-stage ones incorporate charging activities when needed. To take into account uncertain energy consumption, we consider a discrete set of scenarios. Our formulation is an improved version of the one presented in \cite{bezzi2023threshold}. The objective is to minimize the expected total duration of the routes, composed by travel times and recharging times. The resulting model is very challenging as it combines stochasticity and discrete decisions. Therefore, to cope with realistic instances, we propose a heuristic for the SEVRP-T. The heuristic is made up of two components: a \textit{route generator} and a \textit{solution assembler}. The route generator builds a pool of high-quality routes, which are sent to a Set Partitioning formulation at the end. This formulation selects (i.e., assembles) the most promising combination of routes from the pool, such that each customer is served by a single EV. The route generator relies on an Iterated Local Search procedure  (ILS), which uses a Variable Neighborhood Descent metaheuristic (VND). 
Additionally, we apply a scenario reduction technique, which allows us to handle instances with  large sets of energy consumption  scenarios. To evaluate route durations within the VND, we formulate and exactly solve a Stochastic Fixed Route Vehicle Charging Problem with a Threshold policy (SFRVCP-T). Furthermore, we derive two lower bounds on the duration of fixed routes. We use these bounds to filter unpromising moves. To validate the performance of our proposed heuristic, we have carried out extensive computational experiments on sets of instances adapted from available benchmarks. For small-sized instances which can be solved by an off-the-shelf solver, we further benchmark its performance. 

Summarizing, our contributions are as follows:
\begin{itemize}
\item Propose a two-stage stochastic programming formulation for the Stochastic Electric Vehicle Routing Problem with a Threshold recourse policy and analyze its key characteristics;
\item Design a heuristic algorithm based on an ILS with Set Partitioning to handle large problem instances;
\item Derive lower bounds on the duration of fixed routes to filter unpromising moves during the VND phase of the ILS heuristic;
\item Apply a scenario reduction technique (see \cite{HeiRom2003}) to maintain low the number of considered energy consumption scenarios;
\item Provide extensive numerical experiments with the aim of testing the proposed heuristic on instances derived from the EVRP literature and compare its performance with state-of-the-art solvers.

\end{itemize}

The remainder of the paper is organized as follows. In Section~\ref{sec:litrev} we present an overview of the EVRP literature. In Section~\ref{sec:problem} we formally introduce the SEVRP-T and present its mathematical formulation.  We describe our heuristic algorithm in Section~\ref{sec:solmethod}, and discuss the scenario reduction technique in Section~\ref{sec_scen_red_theory}. In Section~\ref{sec:results} we present our computational results. Finally, we summarise our conclusions and outline future research perspectives in Section~\ref{sec:conclusions}.

\section{Literature review} \label{sec:litrev}

The EVRP is an extension of the traditional Vehicle Routing Problem (VRP, see \cite{TothVigo2014}) whose  objective is to serve a given set of customers by establishing least-cost routes for a fleet of homogeneous vehicles. These are EVs in the case of the EVRP, and thus battery capacity and charging facilities are considered.

When dealing with EVRPs, various modeling elements have to be taken into account, including the charging process of battery and the energy consumption between customers'  visits. For recent surveys on the EVRP and its variants, the reader is referred to~\cite{macrina2020green}, \cite{kucukoglu2021electric} and \cite{xiao2021electric}.

In the extant literature on EVRPs, different charging policies are explored. In the Green Vehicle Routing Problem (GVRP, see \cite{erdougan2012green}) EVs fully recharge their batteries upon visiting a CS. The corresponding charging time is either assumed to be constant (e.g., \cite{andelmin2017exact}) or linearly dependent on the SoC of the vehicle when arriving at the CS (e.g., \cite{schneider2014electric}). Whenever a partial charging policy is adopted, the charging time is included as a decision variable of the model. In contrast to a full charging policy, a partial one results in savings in both energy consumption and time (see \cite{froger2019improved}). Different works assume that the charging process is linear (see \cite{felipe2014heuristic}, \cite{desaulniers2016exact} and \cite{WanZha2023}). However, in practice the charge retrieved at a CS is a nonlinear concave function of the charging time (see \cite{pelletier2017battery}). To account for this, several authors modeled the charging process using piecewise linear functions (e.g., \cite{montoya2017electric} and \cite{KleSch2023}). With this choice, the computational complexity is increased by the effect of breakpoints introduced to model the varying charging rates of the battery (see \cite{DonKocAlt2022}). Alternative approximations of the charging function are proposed in \cite{Lee2021} and \cite{Sch2024}.

With respect to the charging configuration, different technologies exist, i.e., normal, fast, and super-fast chargers (see \cite{KesCat2018}). This aspect, in conjunction with a partial recharging policy, has been explored in \cite{felipe2014heuristic}. In~\cite{froger2019improved} two compact mixed-integer linear programming formulations have been designed for the EVRP with nonlinear charging functions. The authors propose an arc-based tracking of the time and the SoC rather than the traditional node-based tracking. In addition, they develop an exact labeling algorithm to find the optimal charging decisions on the given routes. Recently, \cite{bezzi2023route} have presented a branch-and-price algorithm for EVRP  with partial recharges assuming linear charging processes pertaining to multiple  technologies.

The overwhelming majority of EVRP studies assume that all parameters are known with certainty when taking decisions. In such a context, the EV energy consumption is one of the most relevant parameters. Typically, it is assumed to be a linear function of traveled distance (see \cite{felipe2014heuristic}, \cite{BreBalHarVid2019} and \cite{AlmQuiAngBanDj2020}). However, to achieve more accurate and realistic results, its computation should account for multiple factors, both exogenous and endogenous, including road quality, vehicle characteristics, cargo weight, and environmental conditions (see \cite{kucukoglu2021electric}). In \cite{GoeSch2015} a nonlinear energy consumption model incorporating vehicle speed, road gradient and cargo load is proposed. In \cite{bruglieri2023matheuristic} an EVRP with time windows is addressed considering different factors for the energy consumption, i.e., payload and vehicle speed. Possible stops at the CSs are permitted en route, allowing for partial recharge. Recently, \cite{XiXuYaGuoZha2024} have introduced drivetrain losses and traffic congestion as a nonlinear function of vehicle speed, incorporating them into a comprehensive energy consumption model. Nonetheless, each additional factor complicates the computation of the energy consumption rate. Thus, there is a  crucial trade-off between solution times and the number, as well as the nature, of the considered factors in modeling energy consumption (see \cite{kucukoglu2021electric}). \cite{asamer2016sensitivity} have conducted a sensitivity analysis assessing the impact of various factors on energy consumption, such as total mass, efficiency of driving and rolling friction coefficient. The authors show that several factors, even if  difficult to be precisely measured, significantly affect the estimation. Therefore, it is reasonable to treat such factors as uncertain, opening the door to stochastic and robust optimization approaches.

\cite{bruni2020electric} have introduced an EVRP with stochastic energy consumption for a single vehicle and formulated it as a two-stage stochastic programming model. Tactical routes are operated on a daily basis, being energy feasible on expectation. It is assumed that the energy consumption over the entire network is revealed at the beginning of each operational period. At which point, en route recharging activities are possibly added, allowing the EV to travel from a customer node (or from the depot) to a suitable CS. In our problem, we assume that the uncertain energy consumption is only revealed when an arc starts being traversed by an EV. A probabilistic Bayesian machine learning approach is explored by \cite{basso2021electric} to predict the expected energy consumption and its variance on a road network. Specifically, an EV routing model composed of two stages is presented. In the first stage, minimum energy paths connecting the nodes are devised, while in the second stage the best order to visit the customers is determined. If necessary, partial charging activities are planned. The information provided by the machine learning techniques are cast into the model through a set of chance-constraints, assuring that the SoC of the vehicles will never go below a fixed threshold within a confidence interval. \cite{BassoKulSanQu2022} have considered stochasticity in terms of both energy consumption and customer demands. The resulting stochastic dynamic EVRP is modelled through a Markov decision process and solved through reinforcement learning techniques.

To the best of our knowledge, the first robust EVRP optimization model considering energy consumption uncertainty has been proposed by \cite{pelletier2019electric}. The aim is to plan minimum cost routes with pre-specified guarantees that they will not fail (i.e., run out of battery). Randomness is directly integrated in the model with respect to various uncertainty sets (e.g., boxed, budgeted and ellipsoidal uncertainty sets). \cite{JeoGhaZufNat2024} have proposed an adaptive robust optimization model for EVs with time windows and uncertain energy consumption rate. A partial recharging policy is adopted with a linear recharging function. Adaptive decisions are made after the uncertainty realization on each arc and are related to the SoC,  the service time of the EV at each node, and the battery recharging amount. The aim is to minimize the worst-case energy consumption while ensuring deliveries at the assigned time windows without running out of charge. The model is solved through a column-and-constraint generation based heuristic, which is coupled with a variable neighborhood search and alternating direction algorithm.

EVRPs may be subject to  several sources of uncertainty, e.g., customer demands, waiting times at CSs and their availability. \cite{LiuWanYinChen2023} have considered a pickup and delivery problem using EVs under demand uncertainty. The model is formulated as a two-stage adaptive robust model where the uncertain demand falls within a budget-type uncertainty set. The accumulated load variables are treated as adaptive with respect to the demand scenario, while the routing,  the service start times and remaining battery capacities along a route are fixed before the realization of uncertain demands. \cite{DasErrJabMalu2024} have optimized an EV route from an origin to a destination considering that public CSs occupancy indicators are known in real-time. Whenever new information on the status of a CS is received, the route is dynamically reoptimized.

In this paper, we present a two-stage stochastic programming model for an EVRP under uncertain energy consumption. In the first stage we determine customer sequences, each of which is to be served by a single EV that starts and ends its route at the depot. In the second stage we include partial recharging operations following the threshold recourse policy. Furthermore, we account for different charging technologies, each modeled by a piecewise linear concave function. To the best of our knowledge, the present work, together with \cite{bezzi2023threshold}, represent the only contributions in the literature in which detours to CSs may be performed from any arbitrary intermediate point of an arc, and not only from customers or depot locations. However, with respect to \cite{bezzi2023threshold} where only small instances with 10 customers were considered and no solution algorithm was provided, in this paper we propose an enhanced model and a heuristic algorithm. As a result, we are able to tackle instances with up to 80 customers within reasonable computational times.

\section{Problem description and formulation} \label{sec:problem}
Let $\mathcal{I}$ be the set of customers and $\mathcal{C}$ be the set of CSs at which the EVs may fully or partially recharge their batteries. The customers are served by a homogeneous fleet of EVs located at a single depot denoted by node 0. To allow multiple visits to a CS, as common in flow-based EVRP formulations, we define a set $\mathcal{K}$ which includes $n$ copies of each CS in $\mathcal{C}$, i.e., $|\mathcal{K}| =n |\mathcal{C}|$. Each CS $k\in\mathcal{K}$ is equipped with a charging technology described by a piecewise linear concave charging function with set of breakpoints $\mathcal{B}_k = \{0,\ldots,b_k\}$. We denote by $c_{kb}$ and $a_{kb}$ the charging time and the SoC of breakpoint $b \in \mathcal{B}_k$ at CS $k\in\mathcal{K}$, respectively. Let $\mathcal{I}^+ := \mathcal{I} \cup \{0\}$ and $\mathcal{V}:=\mathcal{I}^+ \cup \mathcal{K}$. Our problem is defined on the complete graph $\mathcal{G}:=(\mathcal{V},\mathcal{A})$. Let $(X_i, Y_i) \in \mathbb{R}^+ \times \mathbb{R}^+$ be the Cartesian coordinates of node $i \in \mathcal{V}$. Traveling from node $i \in \mathcal{V}$ to node $j \in \mathcal{V}$ incurs a deterministic driving time $t_{ij} \in \mathbb{R}^+$, assumed to be linearly correlated with Euclidean distance $d_{ij} \in \mathbb{R}^+$ between nodes $i$ and $j$. The EVs leave the depot fully charged, with a SoC equal to $Q^{max}$.

To account for EV energy consumption  uncertainty, we consider a set $\mathcal{S}$ of possible energy consumption scenarios. For every scenario $s \in \mathcal{S}$, we define $e_{ijs} \in \mathbb{R}^+$ as the energy consumption  along arc $(i,j)\in\mathcal{A}$ under scenario $s$ (i.e., in a given scenario we assume that the energy consumption remains constant along the arc). We denote by $p_s\in[0,1]$ the probability of scenario $s\in\mathcal{S}$, such that $\sum_{s \in \mathcal{S}}{p_s} = 1$. 

We model the SEVRP-T as a two-stage stochastic program with the  objective of  minimizing the expected total EV travel times and charging times. The first-stage decisions consist of routes satisfying the following conditions: 1) each customer is served exactly once by a single vehicle; 2) each route starts and ends at the depot. First-stage routes are described by binary decision variables $x_{ij}$, taking value 1 if arc $(i,j)$ with $i,j\in\mathcal{I}^+,i\neq j$ is traversed and 0 otherwise. No en route charging operations are considered in the first stage.  Furthermore, we assume that the visiting sequence determined in the first stage must be respected in the second stage. This assumption aligns with the notion of \textit{a priori} routes, commonly considered in the stochastic VRP literature (e.g., \cite{OyoArnWood2018}).  Indeed, a priori routes are useful when delivery times are communicated to customers prior to starting the route. Maintaining the sequence indirectly helps in preserving such communicated arrival times. The actual energy consumption $e_{ijs}$ in scenario $s\in\mathcal{S}$ of arc $(i,j)\in\mathcal{A}$ is revealed when it starts being traversed by an EV at the second stage, and the energy consumption per unit of distance is presumed to be equal along the arc in scenario $s$ (i.e., $e_{ijs}/d_{ij}$). 
As the realized energy consumption may lead the EV to run out of battery, we propose the following threshold recourse policy to govern the second stage. Whenever the SoC of an EV drops to $Q^T <Q^{max}$, while driving towards a node or at a node, the EV  determines a CS to perform en route charging. Intuitively, $Q^T$ may be viewed as the tolerance up to which an EV driver feels the necessity to recharge. This mimics the behavior of conventional vehicle drivers who tend to look for refueling stations after a certain level of fuel is reached. The EV will head to its chosen CS, recharge and visit the subsequent customer on its a priori route. We assume that, when performing a detour along arc $(i,j)$ with $i\in\mathcal{I}^+,j\in\mathcal{I},i\neq j$, the energy consumption from the detour point (i.e., once the SoC of an EV drops to $Q^T$) to the chosen CS and from the CS to customer $j\in\mathcal{I}$ follows the same energy consumption of arc $(i,j)$ under scenario $s\in\mathcal{S}$, which equals $e_{ijs}/d_{ij}$. The amount of energy to charge is such that the EVs' SoC upon arriving at the subsequent customer is exactly $Q^{G}>Q^{T}$. Additionally, we allow the SoC of each EV to be lower than $Q^{T}$ when returning to the depot in the last arc of the planned route. Lastly, we impose that at most a single visit to a CS between two consecutive nodes should be performed. In Figure \ref{fig_example_thresholdrecoursepolicy} we provide an example of the proposed threshold recourse policy applied to an instance with three customers and two CSs.

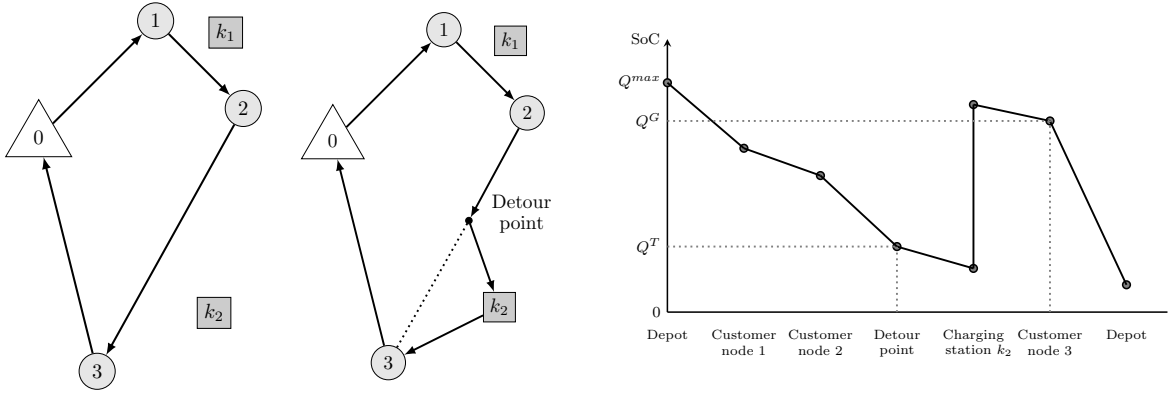
\begin{figure}[ht!]
\begin{minipage}{0.15\textwidth}
	\begin{center}
		\resizebox{3.5cm}{!}{
			\begin{tikzpicture}
				[node distance={15mm},
                detour/.style  = {circle, draw, thin, fill = black, scale = 0.2},
                depot/.style = {draw, regular polygon, regular polygon sides=3},
                cs/.style = {rectangle, draw, thin, fill = gray!40, scale = 1},
                customer/.style  = {circle, draw, thin, fill = gray!20, scale = 1}] 
				\node[depot] (0) at (0,0) {0}; 
				\node[customer] (1) at (2,2) {1}; 
				\node[customer] (2) at (3.5,0.5) {2};
                \node[cs] (3) at (3,-3) {$k_2$};
                \node[cs] (4) at (3.2,1.8) {$k_1$};
                \node[customer] (5) at (1,-4) {3};

				\draw[-latex,line width=1pt] (0) -- (1);
                \draw[-latex,line width=1pt] (1) -- (2);
                \draw[-latex,line width=1pt] (2) -- (5);
                \draw[-latex,line width=1pt] (5) -- (0);

		\end{tikzpicture}}
	\end{center}
\end{minipage} 
\hfill
\begin{minipage}{0.15\textwidth}
	\begin{center}
		\resizebox{3.5cm}{!}{
			\begin{tikzpicture}
				[node distance={15mm},
                detour/.style  = {circle, draw, thin, fill = black, scale = 0.2},
                depot/.style = {draw, regular polygon, regular polygon sides=3},
                cs/.style = {rectangle, draw, thin, fill = gray!40, scale = 1},
                customer/.style  = {circle, draw, thin, fill = gray!20, scale = 1}] 
				\node[depot] (0) at (0,0) {0}; 
				\node[customer] (1) at (2,2) {1}; 
				\node[customer] (2) at (3.5,0.5) {2};
                \node[cs] (3) at (3,-3) {$k_2$};
                \node[cs] (4) at (3.2,1.8) {$k_1$};
                \node[customer] (5) at (1,-4) {3};
                \node[detour] (6) at (2.45,-1.44) {f};
                \node at (3.4,-1.1) {Detour};
                \node at (3.4,-1.5) {point};

				\draw[-latex,line width=1pt] (0) -- (1);
                \draw[-latex,line width=1pt] (1) -- (2);
                \draw[-latex,line width=1pt] (2) -- (6);
                \draw[dotted,line width=1pt] (6) -- (5);
                \draw[-latex,line width=1pt] (6) -- (3);
                \draw[-latex,line width=1pt] (3) -- (5);
                \draw[-latex,line width=1pt] (5) -- (0);

		\end{tikzpicture}}
	\end{center}
\end{minipage} 
\hfill
\begin{minipage}{0.5\textwidth}
	\begin{center}
		\centering
		\resizebox{0.95\textwidth}{!}{
			\begin{tikzpicture}
\draw [thick] (0, 0) -- (9.15, 0);
\draw [-stealth, thick] (0, 0) -- (0, 5) node[left]{\footnotesize SoC};   
\node [left] (tick1) at (0.0, 0.0) {\footnotesize 0};
\node [left] (tick2) at (0.0, 1.2) {\footnotesize $Q^{T}$};
\node [left] (tick3) at (0.0, 3.5) {\footnotesize $Q^{G}$};
\node [left] (tick4) at (0.0, 4.2) {\footnotesize $Q^{max}$};

\node [below] (depotin) at (0, -0.2) {\scriptsize Depot};
\node [below] (cust1) at (1.4, -0.2) {\scriptsize Customer};
\node at (1.4,-0.7) {\scriptsize node 1};
\node [below] (cust2) at (2.8, -0.2) {\scriptsize Customer};
\node at (2.8,-0.7) {\scriptsize node 2};
\node [below] (detour) at (4.2, -0.2) {\scriptsize Detour};
\node at (4.2,-0.7) {\scriptsize point};
\node [below] (cs) at (5.6, -0.2) {\scriptsize Charging};
\node at (5.7,-0.7) {\scriptsize station $k_2$};
\node [below] (cust3) at (7.0, -0.2) {\scriptsize Customer};
\node at (7.0,-0.7) {\scriptsize node 3};
\node [below] (depotfin) at (8.4, -0.2) {\scriptsize Depot};

\fill [fill = gray, draw = black, thick] (0.0, 4.2) circle (2pt) node {};
\fill [fill = gray, draw = black, thick] (4.2, 1.2) circle (2pt) node {};
\fill [fill = gray, draw = black, thick] (7.0, 3.5) circle (2pt) node {};

\fill [fill = gray, draw = black, thick] (1.4, 3) circle (2pt) node {};
\fill [fill = gray, draw = black, thick] (2.8, 2.5) circle (2pt) node {};
\fill [fill = gray, draw = black, thick] (5.6, 0.8) circle (2pt) node {};
\fill [fill = gray, draw = black, thick] (5.6, 3.8) circle (2pt) node {};
\fill [fill = gray, draw = black, thick] (8.4, 0.5) circle (2pt) node {};

\draw[line width=1pt] (0, 4.2) -- (1.4,3);
\draw[line width=1pt] (1.4, 3) -- (2.8,2.5);
\draw[line width=1pt] (2.8, 2.5) -- (4.2,1.2);
\draw[line width=1pt] (4.2, 1.2) -- (5.6,0.8);
\draw[line width=1pt] (5.6, 0.8) -- (5.6,3.8);
\draw[line width=1pt] (5.6, 3.8) -- (7.0,3.5);
\draw[line width=1pt] (7.0, 3.5) -- (8.4,0.5);

\draw [dotted, gray, line width=1pt] (0.0,3.5) -- (7.0,3.5);
\draw [dotted, gray, line width=1pt] (0.0,1.2) -- (4.2,1.2);

\draw [dotted, gray, line width=1pt] (4.2, 0) -- (4.2,1.2);
\draw [dotted, gray, line width=1pt] (7.0, 0) -- (7.0,3.5);
\end{tikzpicture}}
	\end{center}
\end{minipage} 
\caption{An example with three customers $(1,2,3)$ and two CSs $(k_1,k_2)$. Left panel: the first-stage route $\{0,1,2,3,0\}$ is determined. Middle panel: after the realization of the uncertainty under a certain scenario, the threshold recourse policy is implemented, prescribing a stop at CS $k_2$ between customer nodes 2 and 3 to recharge the battery. Right panel: SoC when implementing the threshold recourse policy in the second stage.}
\label{fig_example_thresholdrecoursepolicy}
\end{figure}

In order to model the threshold recourse policy, we introduce binary variables $\phi_{ijks}$ and continuous variables $z_{ijs}$ for each arc $(i,j)$ with $i,j\in\mathcal{I}^+,i\neq j$ in each scenario $s\in\mathcal{S}$: variables $\phi_{ijks}$ take value 1 if the vehicle traversing arc $(i,j)$ detours towards CS $k\in\mathcal{K}$ under scenario $s\in\mathcal{S}$ and 0 otherwise, while $z_{ijs}\in[0,1]$ represent the portion of arc $(i,j)$ actually traversed from node $i$ to the detour point i.e., the point when the SoC reaches $Q^T$. Note that variables $z_{ijs}$ do not depend on the CS. Auxiliary variables $(X_{ijs}, Y_{ijs}) \in \mathbb{R}^+ \times \mathbb{R}^+$ keep track of the Euclidean coordinates of the detour point in arc $(i,j)$ with $i,j\in\mathcal{I}^+,i\neq j$ under scenario $s\in\mathcal{S}$, while $\ell_{ijks}$ represent the distance between the detour point and CS $k\in\mathcal{K}$. Finally, we make use of classical node-based tracking variables for piecewise charging function (see~\cite{froger2019improved}): $y_{is} \in \mathbb{R}^+$ represent the SoC of a vehicle when it departs from node $i \in \mathcal{V}$ under scenario $s \in \mathcal{S}$; $\underline{q}_{ks}, \overline{q}_{ks} \in \mathbb{R}^+$ represent the SoC of a vehicle when it enters/leaves CS $k \in \mathcal{K}$ under scenario $s \in \mathcal{S}$, respectively; $\underline{c}_{ks}, \overline{c}_{ks} \in \mathbb{R}^+$ represent scaled start/end time for charging a vehicle at CS $k \in \mathcal{K}$ under scenario $s \in \mathcal{S}$, according to the charging function of CS $k$, respectively; $\underline{w}_{kbs}, \overline{w}_{kbs} \in \{0,1\}$ take value 1 if the SoC of a vehicle is between $a_{k(b-1)}$ and $a_{kb}$, with $b\in\mathcal{B}_k\setminus \{0\}$, when entering/leaving CS $k \in \mathcal{K}$ under scenario $s \in \mathcal{S}$, respectively, and 0 otherwise; $\underline{\lambda}_{kbs}, \overline{\lambda}_{kbs} \in \mathbb{R}^+$ are continuous coefficients associated with breakpoint $b \in \mathcal{B}_k$ under scenario $s \in \mathcal{S}$ when the EV enters/leaves CS $k \in \mathcal{K}$, respectively.

The notation is summarized in Table~\ref{table:parameters}.

\begin{table}[H]
\caption{Sets, Parameters and Decision Variables.}
\label{table:parameters}
\resizebox{\textwidth}{!}{
\begin{tabular}{|l|l|}
\hline
\multicolumn{2}{|l|}{\textbf{Sets}} \\
\hline
$\mathcal{I}$ & Set of customers. \\
$\mathcal{I}^+ := \mathcal{I} \cup \{0\}$ & Set of customers and depot. \\
$\mathcal{C}$ & Set of Charging Stations (CS). \\
$\mathcal{K}$ & Set of duplicated charging stations.\\
$\mathcal{B}_k$ & Set of breakpoints for the piecewise linear charging function of CS $k \in \mathcal{K}$. \\
$\mathcal{V} := \mathcal{I}^+ \cup \mathcal{K}$ & Set of all nodes. \\
$\mathcal{A}$ & Set of all arcs. \\
$\mathcal{S}$ & Set of energy consumption scenarios. \\
\hline
\multicolumn{2}{|l|}{\textbf{Deterministic Parameters}} \\
\hline
$(X_i, Y_i) \in \mathbb{R}^+ \times \mathbb{R}^+$ & Euclidean coordinates of node $i \in \mathcal{V}$. \\
$d_{ij} \in \mathbb{R}^+$ & Distance between nodes $i \in \mathcal{V}$ and $j \in \mathcal{V}$. \\
$t_{ij} \in \mathbb{R}^+$ & Time consumption between nodes $i \in \mathcal{V}$ and $j \in \mathcal{V}$. \\
$Q^{max} \in \mathbb{R}^+$ & Maximum battery charge level. \\
$Q^{T} \in \mathbb{R}^+$ & Minimum threshold battery charge level, with $Q^{T} < Q^{max}$. \\
$Q^{G} \in \mathbb{R}^+$ & Goal level of energy after each detour, with $Q^{T} < Q^{G} < Q^{max}$. \\
$c_{kb} \in \mathbb{R}^+$ & Charging time of breakpoint $b \in \mathcal{B}_k$ of CS $k \in \mathcal{K}$. \\
$a_{kb} \in \mathbb{R}^+$ & State of Charge (SoC) of breakpoint $b \in \mathcal{B}_k$ of CS $k \in \mathcal{K}$. \\
$M \in \mathbb{R}^+$ & A Big-M number.\\
\hline
\multicolumn{2}{|l|}{\textbf{Stochastic Parameters}} \\
\hline
$p_s \in [0,1]$ & Probability of scenario $s \in \mathcal{S}$. \\
$e_{ijs} \in \mathbb{R}^+$ & Energy consumption between nodes $i \in \mathcal{V}$ and $j \in \mathcal{V}$ under scenario $s \in \mathcal{S}$. \\
\hline
\multicolumn{2}{|l|}{\textbf{Decision Variables}} \\
\hline
$x_{ij} \in \{0,1\}$ & 1 if the EV traverses arc $(i,j)$ with $i,j\in\mathcal{I}^{+}, i \neq j$, 0 otherwise. \\
$u_i \in \mathbb{R}^{+}$ & Dummy variable for subtour elimination constraints involving customer $i \in \mathcal{I}^+$. \\
$\phi_{ijks} \in \{0,1\}$ & 1 if the EV detours along arc $(i,j)$ with $i,j\in\mathcal{I}^{+}, i \neq j$ towards CS $k\in\mathcal{K}$\\ & under scenario $s \in \mathcal{S}$, 0 otherwise. \\
$z_{ijs} \in [0,1]$ & Percentage of arc $(i,j)$ with $i,j\in\mathcal{I}^{+}, i \neq j$ covered from node $i$ to detour point \\ & under scenario $s \in \mathcal{S}$. \\
$y_{is} \in \mathbb{R}^+$ & SoC of the EV when it departs from node $i \in \mathcal{V}$ under scenario $s \in \mathcal{S}$. \\
$\ell_{ijks} \in \mathbb{R}^+$ & Distance between the detour point along arc $(i,j)$ with $i,j\in\mathcal{I}^{+}, i \neq j$ to CS $k\in\mathcal{K}$\\ & under scenario $s \in \mathcal{S}$. \\
$(X_{ijs}, Y_{ijs}) \in \mathbb{R}^+ \times \mathbb{R}^+$ & Euclidean coordinates of the detour point in arc $(i,j)$ with $i,j\in\mathcal{I}^{+}, i \neq j$\\ & when $z_{ijs} > 0$ under scenario $s \in \mathcal{S}$. \\
$\underline{q}_{ks} \in \mathbb{R}^+$ & SoC of the EV when it enters CS $k \in \mathcal{K}$ under scenario $s \in \mathcal{S}$. \\
$\overline{q}_{ks} \in \mathbb{R}^+$ & SoC of the EV when it leaves CS $k \in \mathcal{K}$ under scenario $s \in \mathcal{S}$. \\
$\underline{c}_{ks} \in \mathbb{R}^+$ & Scaled start time for charging the EV at CS $k \in \mathcal{K}$ under scenario $s \in \mathcal{S}$. \\
$\overline{c}_{ks} \in \mathbb{R}^+$ & Scaled end time for charging the EV at CS $k \in \mathcal{K}$ under scenario $s \in \mathcal{S}$. \\
$\underline{w}_{kbs} \in \{0,1\}$ & 1 if the SoC of the EV is between $a_{k(b-1)}$ and $a_{kb}$, with $b\in\mathcal{B}_k\setminus \{0\}$, when entering CS $k \in \mathcal{K}$ \\& under scenario $s \in \mathcal{S}$, 0 otherwise. \\
$\overline{w}_{kbs} \in \{0,1\}$ & 1 if the SoC of the EV is between $a_{k(b-1)}$ and $a_{kb}$, with $b\in\mathcal{B}_k\setminus \{0\}$, when leaving CS $k \in \mathcal{K}$ \\& under scenario $s \in \mathcal{S}$, 0 otherwise. \\
$\underline{\lambda}_{kbs} \in \mathbb{R}^+$ & Continuous coefficient associated with breakpoint $b \in \mathcal{B}_k$ under scenario $s \in \mathcal{S}$ \\ & when the EV enters CS $k \in \mathcal{K}$. \\
$\overline{\lambda}_{kbs} \in \mathbb{R}^+$ & Continuous coefficient associated with breakpoint $b \in \mathcal{B}_k$ under scenario $s \in \mathcal{S}$ \\ & when the EV leaves CS $k \in \mathcal{K}$. \\
\hline
\end{tabular}}
\end{table}

\newpage
The extensive formulation of the SEVRP-T, which corresponds to  a two-stage stochastic mixed-integer second-order cone model,  is  as follows:
\small
\begin{align}
\text{minimize} \sum_{i,j \in \mathcal{I}^+,i \neq j} t_{ij} x_{ij} + \sum_{s \in \mathcal{S}}p_s \left( \sum_{k \in \mathcal{K}}(\overline{c}_{ks} - \underline{c}_{ks}) + \sum_{i,j \in \mathcal{I}^+, i \neq j} t_{ij} z_{ijs} + \sum_{i,j \in \mathcal{I}^+, i \neq j} \sum_{k \in \mathcal{K}} \bigg(\frac{t_{ij}}{d_{ij}} \ell_{ijks} + (t_{kj} - t_{ij})\phi_{ijks}\bigg) \right)  \label{eq:obj}
\end{align}
\vspace*{-0.8cm}
\begin{align}
\text{subject to:} \notag \\
&\sum_{j \in \mathcal{I}^+, j \neq i} x_{ij} = 1 & & \forall i \in \mathcal{I} \label{1_eq:copertura}\\
&\sum_{j \in \mathcal{I}^+, j \neq i} x_{ij} - \sum_{j \in \mathcal{I}^+, j \neq i} x_{ji} = 0 & & \forall i \in \mathcal{I}^+ \label{1_eq:singlevisit}\\
&u_i - u_j + |\mathcal{I}| x_{ij} \leq |\mathcal{I}| - 1 & & \forall i,j \in \mathcal{I}^+, i \neq j \label{1_eq:sec}\\
&\sum_{i,j \in \mathcal{I}^+, i \neq j} \phi_{ijks} \leq 1 & & \forall k \in \mathcal{K}, s \in \mathcal{S} \label{eq:station_single_visit}\\
&\sum_{k \in \mathcal{K}} \phi_{ijks} \leq x_{ij} & & \forall i,j \in \mathcal{I}^+, i \neq j, s \in \mathcal{S} \label{eq:detour_existent_arc}\\
&z_{ijs} \leq \sum_{k \in \mathcal{K}} \phi_{ijks} & & \forall i,j \in \mathcal{I}^+, i \neq j, s \in \mathcal{S} \label{eq:detour_percentage}\\
&\ell_{ijks} \leq M \phi_{ijks} & & \forall i,j \in \mathcal{I}^+, i \neq j, k \in \mathcal{K}, s \in \mathcal{S} \label{eq:ell_zero}\\
&y_{js} \geq Q^{T}\big(x_{ij} - \sum_{k \in \mathcal{K}} \phi_{ijks}\big) & & \forall i \in \mathcal{I}^+, j \in \mathcal{I}, i \neq j, s \in \mathcal{S} \label{eq:y_min_energy}\\
&y_{ks} \leq Q^{max}\sum_{i,j \in \mathcal{I}^+, i \neq j} \phi_{ijks} & &\forall k \in \mathcal{K}, s \in \mathcal{S} \label{eq:y_max_energy} \\
&{\underline{q}_{ks} = \sum_{i,j \in \mathcal{I}^+, i \neq j} \bigg(Q^{T}\phi_{ijks} - \frac{e_{ijs}}{d_{ij}} \ell_{ijks}\bigg)} & &\forall k \in \mathcal{K}, s \in \mathcal{S} \label{eq:battery_cs_level} \\
&y_{is} - y_{js} \geq e_{ijs} \big(x_{ij} - \sum_{k \in \mathcal{K}} \phi_{ijks}\big) - Q^{max}\big(1 - (x_{ij} - \sum_{k \in \mathcal{K}} \phi_{ijks})\big)  & &\forall i \in \mathcal{I}^+, j \in \mathcal{I}, i \neq j, s \in \mathcal{S} \label{eq:y_nodetour1} \\
&y_{is} - y_{js} \leq e_{ijs} \big(x_{ij} - \sum_{k \in \mathcal{K}} \phi_{ijks}\big) + Q^{max}\big(1 - (x_{ij} - \sum_{k \in \mathcal{K}} \phi_{ijks})\big) & &\forall i \in \mathcal{I}^+, j \in \mathcal{I}, i \neq j, s \in \mathcal{S} \label{eq:y_nodetour2} \\
&y_{is} \geq e_{i0s} \big(x_{i0} - \sum_{k \in \mathcal{K}} \phi_{i0ks}\big) - Q^{max}\big(1 - (x_{i0} - \sum_{k \in \mathcal{K}} \phi_{i0ks})\big) & &\forall i \in \mathcal{I}, s \in \mathcal{S} \label{eq:y_nodetour2depot} \\
&y_{ks} - y_{js} \! \geq e_{kjs} \! \! \! \sum_{i \in \mathcal{I}^+, i \neq j} \! \! \! \phi_{ijks} \! - Q^{max} \! \big(1 \! - \! \! \! \sum_{i \in \mathcal{I}^+, i \neq j} \! \! \phi_{ijks}\big) & &\forall j \in \mathcal{I}, k \in \mathcal{K}, s \in \mathcal{S} \label{eq:y_detour1} \\
&y_{ks} - y_{js} \! \leq e_{kjs} \! \! \! \sum_{i \in \mathcal{I}^+, i \neq j} \! \! \! \phi_{ijks} \! + Q^{max} \! \big(1 \! - \! \! \! \sum_{i \in \mathcal{I}^+, i \neq j} \! \! \phi_{ijks}\big) & &\forall j \in \mathcal{I}, k \in \mathcal{K}, s \in \mathcal{S} \label{eq:y_detour2} \\
&y_{ks} \! \geq e_{k0s} \! \! \! \sum_{i \in \mathcal{I}} \! \! \! \phi_{i0ks} \! - Q^{max} \! \big(1 \! - \! \! \! \sum_{i \in \mathcal{I}} \! \! \phi_{i0ks}\big) & &\forall k \in \mathcal{K}, s \in \mathcal{S} \label{eq:y_detourdepot} \\
&y_{js} \geq Q^{G} - Q^{max}\big(1 - \sum_{i \in \mathcal{I}^+, k \in \mathcal{K}, i \neq j} \phi_{ijks}\big) & &\forall j \in \mathcal{I}, s \in \mathcal{S} \label{eq:target1} \\
&y_{js} \leq Q^{G} + Q^{max}\big(1 - \sum_{i \in \mathcal{I}^+, k \in \mathcal{K}, i \neq j} \phi_{ijks}\big) & &\forall j \in \mathcal{I}, s \in \mathcal{S} \label{eq:target2} \\
&e_{ijs} z_{ijs} \geq y_{is} - Q^{T} - Q^{max}(1 - \phi_{ijks}) & & \forall i,j \in \mathcal{I}^+, i \neq j, k \in \mathcal{K}, s \in \mathcal{S} \label{eq:p_detour_point1}\\
&e_{ijs} z_{ijs} \leq y_{is} - Q^{T} + Q^{max}(1 - \phi_{ijks}) & & \forall i,j \in \mathcal{I}^+, i \neq j, k \in \mathcal{K}, s \in \mathcal{S} \label{eq:p_detour_point2}\\
&y_{ks} = \overline{q}_{ks} & &\forall k \in \mathcal{K}, s \in \mathcal{S} \label{eq:24}\\
&y_{0s} = Q^{max} & &\forall s \in \mathcal{S} \label{eq:Qmax_depot}
\end{align}

\begin{align}
&\underline{q}_{ks}\leq \overline{q}_{ks} & &\forall k \in \mathcal{K}, s \in \mathcal{S} \label{eq:26_InputLessOutput}\\
&\underline{q}_{ks} = \sum_{b \in \mathcal{B}_k} \underline{\lambda}_{kbs}a_{kb} & &\forall k \in \mathcal{K}, s \in \mathcal{S} \label{eq:27_LinearCombinationInputCharge} \\
&\underline{c}_{ks} = \sum_{b \in \mathcal{B}_k} \underline{\lambda}_{kbs}c_{kb} & &\forall k \in \mathcal{K}, s \in \mathcal{S} \label{eq:28_LinearCombinationInputCharge} \\
&\sum_{b \in \mathcal{B}_k} \underline{\lambda}_{kbs}=\sum_{b \in \mathcal{B}_k\setminus\lbrace 0\rbrace} \underline{w}_{kbs} & &\forall k \in \mathcal{K}, s \in \mathcal{S} \label{eq:29_WeigthsInputActivedStation} \\
&\sum_{b \in \mathcal{B}_k\setminus \{0\}} \underline{w}_{kbs} = \sum_{i,j \in \mathcal{I}^+, i \neq j} \phi_{ijks} & &\forall k \in \mathcal{K}, s \in \mathcal{S} \label{eq:30_RelationInputStationNode}\\
&\underline{\lambda}_{k0s} \leq \underline{w}_{k1s} & &\forall k \in \mathcal{K}, s \in \mathcal{S} \label{eq:31_Alpha} \\
&\underline{\lambda}_{kbs} \leq \underline{w}_{kbs}+\underline{w}_{k(b+1)s} & &\forall k \in \mathcal{K}, b \in \mathcal{B}_k\setminus \lbrace 0,b_k \rbrace,s \in \mathcal{S} \label{eq:32_Alpha}\\
&\underline{\lambda}_{kb_{k}s} \leq \underline{w}_{kb_{k}s} & &\forall k \in \mathcal{K}, s \in \mathcal{S} \label{eq:33_Alpha}\\
&\overline{q}_{ks} = \sum_{b \in \mathcal{B}_k} \overline{\lambda}_{kbs}a_{kb} & &\forall k \in \mathcal{K}, s \in \mathcal{S} \label{eq:34_LinearCombinationInputCharge} \\
&\overline{c}_{ks} = \sum_{b \in \mathcal{B}_k} \overline{\lambda}_{kbs}c_{kb} & &\forall k \in \mathcal{K}, s \in \mathcal{S} \label{eq:35_LinearCombinationInputCharge} \\
&\sum_{b \in \mathcal{B}_k} \overline{\lambda}_{kbs}=\sum_{b \in \mathcal{B}_k\setminus\lbrace 0\rbrace} \overline{w}_{kbs} & &\forall k \in \mathcal{K}, s \in \mathcal{S} \label{eq:36_WeigthsInputActivedStation} \\
&\sum_{b \in \mathcal{B}_k\setminus \{0\}} \overline{w}_{kbs} = \sum_{i,j \in \mathcal{I}^+, i \neq j} \phi_{ijks} & &\forall k \in \mathcal{K}, s \in \mathcal{S} \label{eq:37_RelationInputStationNode}\\
&\overline{\lambda}_{k0s} \leq \overline{w}_{k1s} & &\forall k \in \mathcal{K}, s \in \mathcal{S} \label{eq:38_Alpha} \\
&\overline{\lambda}_{kbs} \leq \overline{w}_{kbs}+\overline{w}_{k(b+1)s} & &\forall k \in \mathcal{K}, b \in \mathcal{B}_k\setminus \lbrace 0,b_k \rbrace,s \in \mathcal{S} \label{eq:39_Alpha}\\ 
&\overline{\lambda}_{kb_{k}s} \leq \overline{w}_{kb_{k}s} & &\forall k \in \mathcal{K}, s \in \mathcal{S} \label{eq:40_Alpha}\\
&X_{ijs} = X_j z_{ijs} + X_i(1 - z_{ijs}) & &\forall i,j \in \mathcal{I}^+, i \neq j, s \in \mathcal{S} \label{eq:coord_xp_nonlinear}\\
&Y_{ijs} = Y_j z_{ijs} + Y_i(1 - z_{ijs}) & &\forall i,j \in \mathcal{I}^+, i \neq j, s \in \mathcal{S} \label{eq:coord_yp_nonlinear}\\
&\ell_{ijks} + M(1 - \phi_{ijks})  \geq \sqrt{(X_{ijs} - X_k)^2 + (Y_{ijs} - Y_k)^2} & &\forall i,j \in \mathcal{I}^+, k \in \mathcal{K}, s \in \mathcal{S}\label{eq:nonlinear_distance}\\
&x_{ij} \in \{0,1\} & & \forall i,j \in \mathcal{I}^+, i \neq j \label{1_eq:xbin}\\
&u_i \geq 0 & & \forall i \in \mathcal{I}^+ \label{1_eq:u}\\
& \phi_{ijks} \in \{0,1\} & & \forall i,j \in \mathcal{I}^+, i \neq j, k \in \mathcal{K}, s \in \mathcal{S} \label{eq:StocDomainPhi} \\
& 0 \leq z_{ijs} \leq 1 & & \forall i,j \in \mathcal{I}^+, i \neq j, s \in \mathcal{S} \label{eq:StocDomainZ} \\
& \ell_{ijks} \geq 0 & & \forall i,j \in \mathcal{I}^+, i \neq j, k \in \mathcal{K}, s \in \mathcal{S} \label{eq:StocDomainPL} \\
& X_{ijs}, Y_{ijs} \geq 0 & & \forall i,j \in \mathcal{I}^+, i \neq j, s \in \mathcal{S} \label{eq:StocDomainX_pY_p}\\
& y_{is} \geq 0 & & \forall i \in \mathcal{V}, s \in \mathcal{S} \label{eq:StocDomainY} \\
&\underline{q}_{ks}, \overline{q}_{ks}, \underline{c}_{ks}, \overline{c}_{ks} \geq 0 & & \forall k \in \mathcal{K}, s \in \mathcal{S} \label{eq:StocDomainCQDelta}\\
&\underline{\lambda}_{kbs}, \overline{\lambda}_{kbs} \geq 0 & & \forall k \in \mathcal{K}, b \in \mathcal{B}_k, s \in \mathcal{S} \label{eq:StocDomainLambda} \\
&\underline{w}_{kbs}, \overline{w}_{kbs} \in \{0,1\}  & & \forall k \in \mathcal{K},  b \in \mathcal{B}_k \setminus\lbrace 0\rbrace, s \in \mathcal{S} \label{eq:StocDomainZW}
\end{align}

\normalsize
The objective function~\eqref{eq:obj} minimizes the expected total duration of the route, defined as the sum of travel times along arcs included in the first-stage route, plus any additional time due to the recourse actions taken in the second stage (i.e., the charging time, the time to reach the detour point, the time to arrive at the chosen CS, and the time to reach the next customer on the route). 

Constraints~\eqref{1_eq:copertura}--\eqref{1_eq:sec} involve only first-stage decision variables. In particular, constraints~\eqref{1_eq:copertura} ensure that each customer is visited once, while constraints~\eqref{1_eq:singlevisit} ensure connectivity. Constraints~\eqref{1_eq:sec} eliminate subtours. Constraints~\eqref{eq:station_single_visit} ensure that every CS copy is visited at most once, while constraints~\eqref{eq:detour_existent_arc} forbid to detour along arcs not included in the first-stage route and ensure that every first-stage arc can be subject to at most one detour. Constraints~\eqref{eq:detour_percentage} and~\eqref{eq:ell_zero} ensure that $z_{ijs} = 0$ and $\ell_{ijks} = 0$ if there is no detour along arc $(i,j)$ with $i,j\in\mathcal{I}^+,i\neq j$, where $M$ is a large number. Constraints~\eqref{eq:y_min_energy} ensure that when the EV traverses arc $(i,j)$ with $i\in\mathcal{I}^+,j\in\mathcal{I},i\neq j$ its SoC at node $j$ is always greater than $Q^{T}$ if no detour operations are performed. Constraints~\eqref{eq:y_max_energy} set the maximum battery level to $Q^{max}$ for each visited station and $0$ for each non-visited station. Constraints~\eqref{eq:battery_cs_level} establish the SoC of a vehicle when it enters station $k\in\mathcal{K}$. The SoC is equal to $0$ if the station is never visited. Constraints~\eqref{eq:y_nodetour1} and~\eqref{eq:y_nodetour2} ensure that, for every first-stage arc $(i,j)$ with $i\in\mathcal{I}^+,j\in\mathcal{I},i\neq j$, the difference between the SoC at nodes $i$ and $j$ is equal to the energy consumption along $(i,j)$, if there is no detour along the arc. Constraints~\eqref{eq:y_nodetour2depot} are the equivalent for the last arc of every route where we allow the EV to return to the depot with an empty battery. Constraints~\eqref{eq:y_detour1} and~\eqref{eq:y_detour2} determine the SoC when returning to customer $j\in\mathcal{I}$ after a detour. Constraints~\eqref{eq:y_detourdepot} are equivalent to the previous ones but hold only for the last arc of each route. Constraints~\eqref{eq:y_detour1}--\eqref{eq:y_detourdepot} use $Q^{max}$ as an upper bound for the SoC. Constraints~\eqref{eq:target1} and~\eqref{eq:target2} enforce the recharge policy for every arc except for the last one of each route. In particular, if a detour occurs when traversing arc $(i,j)$, with $i \in\mathcal{I}^+$ and $j\in\mathcal{I}$, then the SoC at node $j$ must be equal to $Q^G$. Constraints~\eqref{eq:p_detour_point1} and~\eqref{eq:p_detour_point2} fix the detour point, ensuring that every detour starts only when the vehicle reaches battery level $Q^{T}$. Constraints~\eqref{eq:24} update the SoC of the vehicle when leaving a CS. Constraints~\eqref{eq:Qmax_depot} set the battery charge level to $Q^{max}$ at the depot. As in \cite{froger2019improved}, constraints~\eqref{eq:26_InputLessOutput}--\eqref{eq:33_Alpha} and \eqref{eq:34_LinearCombinationInputCharge}--\eqref{eq:40_Alpha} establish the SoC and the charging time of the vehicle when arriving at a CS and leaving from a CS according to the nonlinear charging function of the CS, respectively. Constraints~\eqref{eq:coord_xp_nonlinear} and~\eqref{eq:coord_yp_nonlinear} define the Euclidean coordinates $(X_{ijs},Y_{ijs})$ of the detour point along arc $(i,j)$ with $i,j\in\mathcal{I}^+,i\neq j$ under scenario $s\in\mathcal{S}$. Such coordinates are derived from the Internal Section Formula. Constraints~\eqref{eq:nonlinear_distance} are second-order cone constraints that determine the value of $\ell_{ijks}$. Such constraints can be handled by commercial solvers despite being nonlinear (see~\cite{maggioni2009stochastic}). Finally, constraints~\eqref{1_eq:xbin}--\eqref{eq:StocDomainZW} set the domain of the variables. 

With respect to the formulation proposed in \cite{bezzi2023threshold}, we relax the assumption that the nominal energy consumption of the first-stage routes should not exceed $Q^{max}$. 
Additionally, we allow the SoC of each EV to fall below the threshold $Q^T$ upon returning to the depot in the final arc of the route. This choice enhances the model's flexibility, providing a more realistic representation of EV operations.

\paragraph{Observation}
In the SEVRP-T, the objective is the minimization of the expected total duration of the route, composed of travel times and charging times. Since we are considering a piecewise linear concave charging function, when the energy threshold is attained by the EV, the CSs chosen in the second stage are not necessarily the closest to the detour point. We formalize this observation in the following proposition.

\begin{proposition} \label{prop_distancevstime}
    Let $i\in\mathcal{I}^+$, $j\in\mathcal{I}$ and $k_1,k_2 \in \mathcal{K}$. Suppose that CSs $k_1,k_2$ have the same technology, i.e. $\mathcal{B}_{k_1}=\mathcal{B}_{k_2}=\mathcal{B}=\{0,1,2\}$ and $c_{k_1b}=c_{k_2b}=c_b$, $a_{k_1b}=a_{k_2b}=a_b$ for all $b\in\mathcal{B}$. The corresponding piecewise linear charging function is defined as follows (see Figure \ref{fig:counterexampledetour_new}):
    \begin{equation} 
  SoC(t) = \begin{cases}
    \beta^{1}t & if \text{ } 0\leq t \leq c_1 \\
    \beta^{2}(t-c_2)+Q^{max} & if \text{ } c_1 < t \leq c_2,
  \end{cases} \label{charging_function_prop}
\end{equation}
    with $\beta^{1}=a_1/c_1\geq 1$ and $\beta^{2}=(Q^{max}-a_1)/(c_2-c_1)<\beta^{1}$. Assume that when traversing arc $(i,j)$ under scenario $s\in\mathcal{S}$, the EV reaches a SoC equal to $Q^T$ at point $f$, such that the total distance to complete the path $f-k_1-j$ is shorter than that of path $f-k_2-j$, i.e., $d_{fk_1}+d_{k_1j}<d_{fk_2}+d_{k_2j}$. If the following conditions hold:
\begin{equation} \label{hyp1_prop}
0 \leq Q^T-e_{fks} < a_1 \qquad k\in\{k_1,k_2\}
\end{equation}
\begin{equation} \label{hyp2_prop}
Q^{max} -e_{kjs} \geq Q^G \qquad k\in\{k_1,k_2\}
\end{equation}
\begin{equation} \label{hyp3_prop}
    (t_{fk_1}-t_{fk_2})+\frac{e_{ijs}}{\beta^{1}d_{ij}}(d_{fk_1}-d_{fk_2})+\frac{1}{\beta^{2}}(e_{k_1js}-e_{k_2js})+(t_{k_1j}-t_{k_2j}) >0,
\end{equation}
then recharging at CS $k_1$ takes more time than recharging at CS $k_2$.
\end{proposition}

\begin{figure}[ht!] \centering
\begin{minipage}{0.46\textwidth}
	\begin{center}
		\resizebox{\textwidth}{!}{
			\begin{tikzpicture}
[auto, thick,
  deposito/.style = {rectangle, draw, thin, fill = black, scale = 2},
  stazione/.style = {rectangle, draw, thin, fill = gray!40, scale = 1},
  cliente/.style  = {circle, draw, thin, fill = gray!20, scale = 1},
  detour/.style  = {circle, draw, thin, fill = black, scale = 0.5},
  legenda-stazione/.style = {rectangle, draw, thin, fill = gray!40, scale = 1},       
  legenda-cliente/.style  = {circle, draw, thin, fill = gray!20, scale = 0.75},      
  legenda-detour/.style  = {circle, draw, thin, fill = black, scale = 0.35},
]
\node [cliente] (depot) at (0.5,0)  {$i$};
\node [cliente] (customer) at (7,0)  {$j$};
\node [stazione] (station1) at (4.5,1) {$k_1$}; 
\node [stazione] (station2) at (5.8,-1.2) {$k_2$}; 


\node [detour] (cross) at (4,0) {};
\node at (4,-0.35) {$f$};
\path[-latex] (depot) edge node [below] {$d_{if}$} (cross);
\draw[dotted] (cross) -- (customer);

\path[-latex] [dashed] (cross) edge node [above left] {$d_{fk_1}$} (station1);
\path[-latex] [dashed] (cross) edge node [below] {$d_{fk_2}$} (station2);

\path[-latex] [dashed] (station1) edge node [above] {$d_{k_1j}$} (customer);
\path[-latex] [dashed] (station2) edge node [below right] {$d_{k_2j}$} (customer);

\node (invisible) at (6, -2.8) {};   
\node (invisible) at (8, 0) {};   
\end {tikzpicture}}
	\end{center}
\end{minipage} 
\hfill
\begin{minipage}{0.51\textwidth}
	\begin{center}
		\centering
		\resizebox{\textwidth}{!}{
\begin{tikzpicture}
\draw [-stealth, thick] (-0.3, 0) -- (7.15, 0) node[below]{\footnotesize Time}; 
\draw [-stealth, thick] (0, -0.3) -- (0, 5) node[left]{\footnotesize SoC};   
\node [below left] at (0.0, 0.0) {\footnotesize 0};

\node [below] (c1) at (2, 0.0) {\footnotesize $c_1$};
\node [below] (c3) at (6.4, 0.0) {\footnotesize $c_2$};


\node [left] (a1) at (0.0, 2) {\footnotesize $a_1$};
\node [left] (a3) at (0.0, 4.2) {\footnotesize $a_2=Q^{max}$};



\draw [dotted, gray, line width=.8pt] (c1) -- (2, 2);
\draw [dotted, gray, line width=.8pt] (a1) -- (2, 2);

\draw [dotted, gray, line width=.8pt] (c3) -- (6.4, 4.2);
\draw [dotted, gray, line width=.8pt] (a3) -- (6.4, 4.2);

\draw [black, very thick] (0.0, 0.0) -- (2, 2);
\draw [black, very thick] (2, 2) -- (6.4, 4.2);

\node [left] (betafast) at (1.1, 1.1) {\footnotesize $\beta^1$};
\node [left] (betaslow) at (4.2, 3.2) {\footnotesize $\beta^2$};

\fill [fill = gray, draw = black, thick] (0.0, 0.0) circle (2pt) node {};
\fill [fill = gray, draw = black, thick] (2, 2) circle (2pt) node {};
\fill [fill = gray, draw = black, thick] (6.4, 4.2) circle (2pt) node {};
\end{tikzpicture}}
	\end{center}
\end{minipage} 
	\caption{{Example in which the chosen CS is not the closest one to the detour point $f$. Left panel: graph with relevant nodes and distances. Customers are represented as circles, while CSs are squares. Right panel: charging function of the CSs $k_1,k_2$.}}
	\label{fig:counterexampledetour_new}
\end{figure}
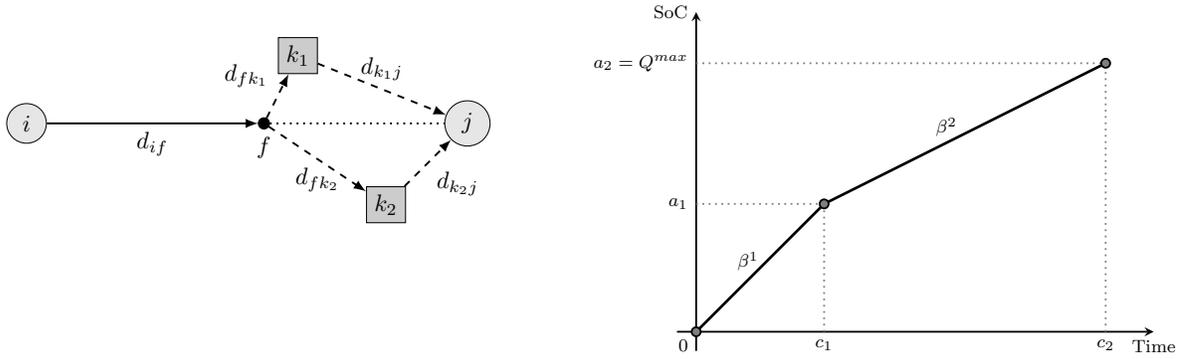

\begin{proof}
    Let $T^k$ be the total detour time, i.e.:
    \begin{equation} \label{total_time_prop}
        T^k=t_{if}+t_{fk}+t_{k}+t_{kj},
    \end{equation}
    where $t_k$ is the total recharging time at CS $k$. By hypothesis \eqref{hyp1_prop}, $t_k$ can be expressed as the sum of two components: $t_k^{1}$, representing the time required to reach a SoC of $a_1$, and $t_k^{2}$, representing the time required to reach a SoC of $Q^G+e_{kjs}$. Note that the charging operations are always feasible under hypotheses \eqref{hyp1_prop}--\eqref{hyp2_prop}. Therefore, using the charging function given in \eqref{charging_function_prop}, we obtain:
    \begin{equation*}
        t_k^{1}=c_1-\frac{1}{\beta^{1}}(Q^T-e_{fks})
    \end{equation*}
        \begin{equation*}
        t_k^{2}=\frac{1}{\beta^{2}}(Q^G+e_{kjs}-a_1),
    \end{equation*}
    with $k\in\{k_1,k_2\}$. Thus, the total time \eqref{total_time_prop} for each CS corresponds to:
    \begin{align*}
        {T^{k_1}} & {= t_{if}+t_{fk_1}+c_1-\frac{1}{\beta^{1}}Q^T+\frac{1}{\beta^{1}}e_{fk_1s}+\frac{1}{\beta^{2}}Q^G+\frac{1}{\beta^{2}}e_{k_1js}-\frac{1}{\beta^{2}}a_1+t_{k_1j}}\\
        {T^{k_2}} & = {t_{if}+t_{fk_2}+c_1-\frac{1}{\beta^{1}}Q^T+\frac{1}{\beta^{1}}e_{fk_2s}+\frac{1}{\beta^{2}}Q^G+\frac{1}{\beta^{2}}e_{k_2js}-\frac{1}{\beta^{2}}a_1+t_{k_2j}.}
    \end{align*}
    Subtracting the two previous expressions, yields:
    \begin{equation} \label{diff_times_prop}
        {T^{k_1}-T^{k_2} = (t_{fk_2}-t_{fk_1})+\frac{1}{\beta^{1}}(e_{fk_1s}-e_{fk_2s})+\frac{1}{\beta^{2}}(e_{k_1js}-e_{k_2js})+(t_{k_1j}-t_{k_2j}).}
    \end{equation}
    According to the model's assumption of uniform energy consumption from the detour point $f$ to each CS and from the CS to the subsequent customer, we have that:
    \begin{equation*}
        {e_{fks}=\frac{e_{ijs}}{d_{ij}}d_{fk} \qquad k \in \{k_1,k_2\}}.
    \end{equation*}
    {Hence, equation \eqref{diff_times_prop} can be reformulated as:}
    \begin{equation*}
        {T^{k_1}-T^{k_2}=(t_{fk_1}-t_{fk_2})+\frac{e_{ijs}}{\beta^{1}d_{ij}}(d_{fk_1}-d_{fk_2})+\frac{1}{\beta^{2}}(e_{k_1js}-e_{k_2js})+(t_{k_1j}-t_{k_2j}),}
    \end{equation*}
    which is strictly positive for hypothesis \eqref{hyp3_prop}.
\end{proof}

Proposition \ref{prop_distancevstime} can be naturally extended to cases where the CSs have different technologies. Furthermore, Proposition \ref{prop_distancevstime} may be generalized to charging functions with more than three breaking points.

\section{Solution strategy} \label{sec:solmethod}
The SEVRP-T is NP-hard since it is an extension of the well-known capacitated vehicle routing problem (see \cite{TothVigo2014}). Given that significant computational effort is required to find optimal solutions for medium- and large-scale EVRP  instances (see \cite{afroditi2014electric}), much of the EVRP literature proposes approximate solution techniques. In this paper, we propose a matheuristic for the SEVRP-T, which is based on an Iterated Local Search procedure with a Set  Partitioning phase (ILS-SP). This technique builds upon the work of~\cite{froger2022electric}, where a similar algorithmic structure showed good performance for deterministic versions of the EVRP.

Our heuristic is composed of two phases: a \textit{route generator} phase and a \textit{route assembly} phase (see Section \ref{sec_ILS}). The route generator phase builds a pool of high-quality routes via an ILS procedure. This phase is performed via two steps: \textit{diversification} and \textit{intensification}. Diversification is implemented as a perturbation procedure (see Section \ref{sec_ILS_perturbation_phase}), while intensification is carried out through a Variable Neighborhood Descent search strategy (VND) (see Section \ref{sec_ILS_VNDphase}). To speed up the evaluation of a given solution, we derive two lower bounds on the expected duration of fixed routes (see Section \ref{sec:fixroute}). Once the pool of routes has been populated, a solution to the SEVRP-T is determined by solving an SP over those routes (see Section \ref{sec_SP}).

The structure of our heuristic is outlined in Algorithm \ref{alg:ILS-P}. Lines 2--12 are related to the route generator phase, while line 13 is the route assembly phase. Firstly, the pool $\Omega$ of routes is empty, and the current best solution $S^*$ is set equal to an input initial solution $S_0$ (line 2). A solution $S$ is comprised by a number of  first-stage routes (i.e., sequences of customer visits). At the same time the number $I$ of iteration is initialized to zero. Next, the algorithm enters its main iterative process (line 3). Excluding the first iteration (lines 4-5), the current best solution $S^*$ is perturbed to move from the local optimum, investigating a different part of the search space (line 7). After such procedure, a new solution $S'$ is obtained and used as input for the intensification phase through the VND algorithm (line 8).  A possibly new solution $S$  is thus identified, and all its new routes are inserted in the pool $\Omega$ (line 9). If the objective function value of $S$ is less than that of the incumbent solution $S^*$, then the incumbent solution is updated to be $S$ (line 11). The main iterative process is repeated until the maximum number of iterations $I_{max}$ is reached. Lastly, a Set Partitioning (SP) problem is solved considering the pool $\Omega$ of generated routes (line 13). In what follows, we detail all algorithmic components of our  heuristic.

\begin{algorithm}[H]
    \caption{ILS-SP}
    \textbf{Input:} An instance of the SEVRP-T with feasible solution $S_0$, a maximum number of iterations $I_{max}$\newline
	\textbf{Output:} A solution for the SEVRP-T\newline
    Let $g(S)$ be the value of the objective function for a solution $S$ and assume $g(\emptyset)=+\infty$.
  \begin{algorithmic}[1]
	\Function{ILS-SP}{}
  \normalsize
		\State $\Omega \gets \emptyset$, $S^* \gets S_0$, $I \gets 0$
		\While{$I < I_{max}$}
			\If{$I = 0$}
				\State $S' \gets S_0$
			\Else
				\State $S' \gets \textsc{perturb}(S^*)$  \Comment{\footnotesize See Section \ref{sec_ILS_perturbation_phase}} \normalsize
			\EndIf
			\State $S \gets \textsc{vnd}(S')$ \Comment{\footnotesize See Section \ref{sec_ILS_VNDphase}} \normalsize
			\State Insert all the new routes of $S$ into $\Omega$
			\If{$g(S) < g(S^*)$}
				\State $S^* \gets S$
			\EndIf
			\State $I \gets I + 1$
		\EndWhile
        \State Solve a Set Partitioning problem on $\Omega$ \Comment{\footnotesize See Section \ref{sec_SP}} \normalsize
	\EndFunction
    \end{algorithmic}
    \label{alg:ILS-P}
\end{algorithm}

\subsection{Route generator and route assembler} \label{sec_ILS}
We use an ILS algorithm (see \cite{LouMarStu2003}) to populate the pool of routes used to assemble a solution for the SEVRP-T. An ILS algorithm escapes from a local optima by perturbing the incumbent solution. We find the local optima via a VND procedure. Further details are provided in the following sections (see Sections \ref{sec_ILS_perturbation_phase}-\ref{sec_ILS_VNDphase}). Finally, in Section \ref{sec:fixroute}, we introduce an exact algorithm for the fixed route vehicle charging problem, which is a subproblem in the VND phase.

\subsubsection{Perturbation phase} \label{sec_ILS_perturbation_phase}
In Algorithm \ref{alg:ILS-P} at line 7, whenever a local optimal solution $S^*$ is reached, it is perturbed. The perturbation phase, also known as \textit{shaking}, is  similar to the one used by \cite{froger2022electric},  which is performed as follows. Firstly, a random customer $i \in \mathcal{I}$ is selected. Then, the $\kappa$ geographically closest customers to $i$ are removed from their respective routes, with $\kappa$ a random integer in the interval $[\min\{|\mathcal{I}|,5\},\max\{\min\{|\mathcal{I}|,5\}$, \\$\lceil \sqrt{|\mathcal{I}|}\rceil\}]$. Finally, the removed customers are reinserted at different positions in the solution, one at a time and in a random order, according to the following rules:
\begin{itemize}
\item A customer cannot be placed back into the same route from which it was removed;
\item A customer is reinserted in the position yielding the minimum increase in the expected time with respect to the base solution, i.e., the one without the removed customers. This is carried out by evaluating the increase in the expected time of every feasible insertion of the customer in every other route, reoptimizing the charging decisions through Algorithm~\ref{alg:E-SFRVCP-T} (see Section \ref{sec:fixroute});
\item If no feasible insertion is possible, a new route with the removed customer is set up.
\end{itemize}

\subsubsection{Variable Neighborhood Descent search phase} \label{sec_ILS_VNDphase}
In the intensification phase of the ILS-SP algorithm (see line 8 in Algorithm \ref{alg:ILS-P}) we apply a VND procedure (see \cite{MlaHan1997} and \cite{golden2008vehicle}). This technique explores an ordered set of neighborhoods $\mathcal{D}$, with respect to the current solution $S$. As soon as an improved solution is found, the algorithm stops searching in the current neighborhood and restarts the search with the first operator of $\mathcal{D}$. Otherwise, the VND moves to the next neighborhood in $\mathcal{D}$. A local optimum is reached when the last operator fails to improve the solution $S$. Let $r_1$ and $r_2$ represent any two routes in the solution $S$. The neighborhoods in  $\mathcal{D}$ are generated based on the following operators: 

\begin{enumerate}
\item \textit{1-0 vertex exchange}: one customer is moved from route $r_1$ to route $r_2$;
\item \textit{1-1 vertex exchange}: interchange between customer $i$ from route $r_1$ and customer $j$ from route $r_2$;
\item \textit{2-0 vertex exchange}: two consecutive customers $i$ and $j$ are moved from route $r_1$ to route $r_2$. The opposite arc $(j,i)$ is also considered;
\item \textit{2-1 vertex exchange}: interchange between consecutive customers $i$ and $j$ from route $r_1$ and customer $\tilde{i}$ from route $r_2$. The opposite arc $(j,i)$ is also considered as well;
\item \textit{2-2 vertex exchange}: interchange between consecutive customers $i$ and $j$ from route $r_1$ and two other adjacent customers $\tilde{i}$ and $\tilde{j}$ from route $r_2$. The opposite arcs $(j,i)$ and $(\tilde{j},\tilde{i})$ are also considered, yielding four possible combinations;
\item \textit{2-opt}: two nonconsecutive edges from routes $r_1$ and $r_2$ are removed, and two new edges are formed to make a new route;
\item \textit{Separate:} (introduced by~\cite{froger2022electric}) this operator splits a single route into two routes by adding a return to the depot after a customer visit.
\end{enumerate}

In our implementation we use intra-route and inter-route versions of the first five operators of the list, resulting in a total of 12 different moves. The structure of the VND procedure is outlined in Algorithm \ref{alg:VND}. For every neighborhood and for every move, we apply a couple of filtering strategies to prune unpromising moves (lines 8-9). Details of these strategies are discussed in the next section.

\begin{algorithm}[H]
    \caption{VND($S$)}
    		\textbf{Input:} An instance of the SEVRP-T with a solution $S$\newline
		\textbf{Output:} A new solution $S^*$ of the SEVRP-T\newline
        For $d\in\mathcal{D}$, let $R_d$ be the set of first-stage routes $r$ impacted by move $d$, with corresponding total first-stage travel time $t_{R_d}$ and $R'_d$ the newly created routes after applying move $d$ to $R_d$.
    \begin{algorithmic}[1]
	\Function{VND}{}
  		\State $S^* \gets S, d \gets 0$
        \State $\mathcal{D} \gets$ list of neighborhoods
        \State $\delta^* \gets +\infty$            \Comment{\footnotesize Best improvement so far} \normalsize
		\While{$d < |\mathcal{D}|$} 
            \State $\delta \gets 0$
            \For{every move in neighborhood $d$}
            \normalsize
                \If{$\displaystyle \sum_{r \in R'_d}\sum_{(i,j) \in r}t_{ij} < \gamma t_{R_d}$}
                \Comment{\footnotesize If first-stage travel times are promising. See Section \ref{sec:fixroute}}
                \normalsize
                    \If{$\displaystyle \sum_{r \in R'_d} \Call{LB-SFRVCP-T}{r, \overline{s}} < t_{R_d}$}  \Comment{\footnotesize See Section \ref{sec:fixroute}} \normalsize
                        \State $\delta \gets t_{R_d} - \displaystyle\sum_{r \in R'_d}\Call{E-SFRVCP-T}{r}$ \Comment{\footnotesize See Section \ref{sec:fixroute}} \normalsize
                        \State $R^* \gets R'_d$
                        
                    \EndIf
                \EndIf

                \If{$\delta < \delta^*$}
                    \State{$\delta^* \gets \delta$}
                \EndIf
            \EndFor
            \If{$\delta^* < 0$}
                \State $S^* \gets S^* \cup R^* \setminus R_d$
            \Else
                \State $d \gets d+1$
            \EndIf
		\EndWhile
  		\State \Return $S^*$
	\EndFunction
    \end{algorithmic}
    \label{alg:VND}
\end{algorithm}

\subsubsection{The stochastic fixed route vehicle charging problem with a threshold policy} \label{sec:fixroute}

In the execution of Algorithm \ref{alg:VND}, the evaluation of the expected duration of a first-stage route $r$  is a subproblem arising several times throughout our algorithm. We call this problem the Stochastic Fixed Route Vehicle Charging Problem with a Threshold policy (SFRVCP-T). The objective of a fixed route vehicle charging problem (see \cite{froger2019improved}) is to identify optimal charging operations, related to which CSs to insert in the given route (comprised of a sequence of customer visits) and the amount of energy to recharge at them, while minimizing the total duration of the route. In the special case of the SFRVCP-T, the CSs to visit for each scenario $s\in\mathcal{S}$ for each route must be identified.

One possibility to solve the SFRVCP-T is to derive a second-order cone mixed-integer programming model from the one presented in Section \ref{sec:problem}. Specifically, the first-stage variables $x_{ij}$ are set to 1 if arc $(i,j)$ belongs to $r$ and $0$ otherwise. As a consequence, only the second-stage variables should be determined.  However, this approach is predictably slow from a computational perspective, as demonstrated empirically in Section \ref{sec_performance_heuristic_STOCH}. For this reason, we propose an exact algorithm called E-SFRVCP-T$(r)$ (see Algorithm \ref{alg:E-SFRVCP-T}).

Algorithm \ref{alg:E-SFRVCP-T} operates as follows. First, we initialize to zero the duration $t_r$ of the route $r$ (line 2). Then, the algorithm enters in two iterative processes, one over the set of scenarios and one over the arcs of the route. If no detour is needed under scenario $s\in\mathcal{S}$ when traversing arc $(i,j)\in r$ (line 6), then no recharging operations are needed. This happens when the EV has enough energy to completely traverse the arc with a SoC that is greater than $Q^{T}$. In this case, the EV's SoC is updated, as well as the duration of the route (lines 7--8). Otherwise, the best CS yielding the lowest travel detour and recharging time is chosen (lines 10--19). The selection is performed on a precomputed set $\mathcal{C}_{ij}$ of non-dominated CSs. The construction of set $\mathcal{C}_{ij}$ is outlined in Algorithm \ref{alg:NDCS} in \ref{sec_appendix_nondominatedCS_algo}. We denote by $\psi^{-1}_k(q_i,q_f)$ the inverse charging function of CS $k\in\mathcal{C}_{ij}$, where $q_i$ and $q_f$ are the initial and final SoC of the vehicle, respectively. In line 16, if no feasible CSs exist, i.e., the EV runs out of battery before reaching any CSs from the detour point or the amount to charge is greater than $Q^{max}$, the route $r$ is infeasible (lines 20-21). The total duration $\tilde{t}_{rs}$ under scenario $s$ is updated (line 23), as well as the expected recourse duration $t_r$ of route $r$, weighted by the probability $p_s$ of scenario $s\in\mathcal{S}$ (line 25).

\begin{algorithm}[H]
    \caption{E-SFRVCP-T($r$)}
		\textbf{Input:} A fixed route $r$\newline
		\textbf{Output:} Expected recourse duration $t_{r}$ for route $r$
      \begin{algorithmic}[1]
  \Function{\exact}{}
        \State{$t_{r} \gets 0$}
        \For{every scenario $s \in \mathcal{S}$}
            \State{$\tilde{t}_{rs} \gets 0, SoC \gets Q^{max}$}
            \For{every arc $(i,j) \in r$}
		          \If{$SoC - e_{ijs} > Q^{T}$} \Comment{\footnotesize No detour in arc $(i,j)$} \normalsize
				    \State{$SoC \gets SoC - e_{ijs}$}
				    \State{$\tilde{t}_{rs} \gets \tilde{t}_{rs} + t_{ij}$}
		          \Else 
                    \State{$t_r^* \gets \infty$}
            \normalsize
                    \For{every CS $k \in \mathcal{C}_{ij}$}
                        \State{$z_{ijs} \gets (SoC - Q^{T}) / e_{ijs} $}       \Comment{\footnotesize Percentage of arc $(i,j)$ covered until detour} \normalsize
                        \State{$X_{ijs} \gets X_j z_{ijs} + X_i(1 - z_{ijs})$}     \Comment{\footnotesize X-coordinate of the detour point} \normalsize
                        \State{$Y_{ijs} \gets Y_j z_{ijs} + Y_i(1 - z_{ijs})$}     \Comment{\footnotesize Y-coordinate of the detour point} \normalsize
                            
                        \State{$\ell_{ijks} \gets \sqrt{(X_{ijs} - X_k)^2 + (Y_{ijs} - Y_k)^2}$} \Comment{\footnotesize Distance detour point--CS} \normalsize
                    
                        \If{$Q^{T} - \frac{e_{ijs}}{d_{ij}} \ell_{ijks} \geq 0$ and $Q^{G} + e_{kjs} \leq Q^{max}$}   
                        \normalsize
                            \State{$t_r \gets t_{ij} z_{ijs} + \frac{t_{ij}}{d_{ij}} \ell_{ijks} + \psi^{-1}_k(Q^{T} - \frac{e_{ijs}}{d_{ij}} \ell_{ijks}, Q^{G} + e_{kjs}) +  t_{kj} - t_{ij}$}  
                            \If{$t_r < t_r^*$}
                                \State{$t_r^* \gets t_r$}
                            \EndIf
                        \EndIf
                    \EndFor
                    \If{$t_r^* = \infty$}
                        \State {\textbf{Return} $\infty$}      \Comment{\footnotesize Route $r$ is infeasible} \normalsize
                    \Else
                        \State{$\tilde{t}_{rs} \gets \tilde{t}_{rs} + t_r^*$}
                    \EndIf
				    \State{$SoC \gets Q^{G}$}
	            \EndIf
            \EndFor
            \State{$t_{r} \gets t_{r} + p_s \tilde{t}_{rs}$}
        \EndFor
        \State \Return $t_{r}$
	\EndFunction
    \end{algorithmic}
    \label{alg:E-SFRVCP-T}
\end{algorithm}

For a given route composed by $l$ arcs from an instance with $|\mathcal{C}|$ charging stations, Algorithm \ref{alg:E-SFRVCP-T} has a worst-case complexity equal to $\mathcal{O}(l|\mathcal{C}|)$. Indeed, the travelling time of the arcs without detours is evaluated in constant time, but every detour requires a loop over all the possible CSs.

Given the complexity of E-SFRVCP-T($r$), coupled with the with need to frequently call this algorithm within the VND, we propose two lower bounds that are used within the heuristic to accelerate it. The two lower bounds operate on a given first-stage route $r$ and are efficiently computed. We use these bounds in sequence to filter unpromising solutions. Let $R_d$ be the set of first-stage routes impacted by move $d\in\mathcal{D}$ with total duration $t_{R_d}$, and $R'_d$ the newly created routes after applying move $d$ to $R_d$.

The first lower bound is related only to the duration of the first-stage route, without considering any recharging operations (see line 8 in Algorithm \ref{alg:VND}). Specifically, if the total first-stage travelled time of the routes in $R'_d$ exceeds the total duration $t_{R_d}$ of the current set of routes $R_d$ multiplied by a factor $\gamma \geq 1$, then the corresponding move is discarded. Parameter $\gamma$ avoids evaluating excessively lengthily routes. If the condition holds, the algorithm proceeds to the second bound.

The second lower bound accounts for both the technology of the CSs and the duration of the detour. We refer to it as the Lower Bound for SFRVCP-T (LB-SFRVCP-T${(r,\mathcal{S}})$). Its computation is performed in Algorithm \ref{alg:LB-SFRVCP-T}, relying on a simplified version of Algorithm \ref{alg:E-SFRVCP-T}. Let $r$ be a given first-stage route. If no detour is needed under scenario $s\in\mathcal{S}$ while traversing arc $(i,j)\in r$ (line 6), recharging operations are unnecessary, and the SoC and the route duration are simply updated (lines 7--8). Conversely, when a detour is required, the total duration of the route is updated to include both travelling and recharging times (line 10). For the total duration of the detour, we use the lower bound \eqref{eq:lowerboundlengthdetour} derived in Proposition \ref{prop:lowerbound}. Further details on its computation are provided below. Regarding recharging times, the SoC of the EV is assumed to be $Q^T$ upon arrival at the CS and sufficient to reach customer $j$ with a SoC of $Q^G$ upon departure. The total recharging time is then calculated by converting the required charge amount into time using $a^{fast}_j$, i.e., the fastest charge rate of the CS with the fastest technology close to node $j$. All these assumptions ensure that the computation remains consistent with the goal of determining a lower bound.

{Algorithm \ref{alg:LB-SFRVCP-T} is designed to compute the lower bound for all possible scenarios in the scenario tree $\mathcal{S}$. However, to further speed up the procedure, in line 9 of Algorithm \ref{alg:VND} we compute such lower bound only for the average scenario $\overline{s}$.}

\begin{algorithm}[H]
    \caption{LB-SFRVCP-T($r,\mathcal{S}$)}
    \textbf{Input:} A fixed route $r$, a set of scenarios $\mathcal{S}$\newline
    \textbf{Output:} A lower bound on the expected recourse duration $t_{r}$ for route $r$
    \begin{algorithmic}[1]
	\Function{\lb}{} 
        \State{$t_{r} \gets 0$}
        \For{every scenario $s \in \mathcal{S}$}
            \State{$\tilde{t}_{rs} \gets 0, SoC \gets Q^{max}$}
            \For{every arc $(i,j) \in r$}
		          \If{$SoC - e_{ijs} > Q^{T}$} \Comment{\footnotesize No detour in arc $(i,j)$} \normalsize
				    \State{$SoC \gets SoC - e_{ijs}$}
				    \State{$\tilde{t}_{rs} \gets \tilde{t}_{rs} + t_{ij}$}
		          \Else 
                        \normalsize
				    \State{$\tilde{t}_{rs} \gets \tilde{t}_{rs} + \frac{t_{ij}}{d_{ij}} \max\{d_{ij}, \underline{d}_i + \underline{d}_j\} + a^{fast}_j (Q^{G} +\frac{e_{ijs}}{d_{ij}} \underline{d}_j - Q^{T})$}
				    \State{$SoC \gets Q^{G}$}
	            \EndIf
            \EndFor
            \State{$t_{r} \gets t_{r} + p_s \tilde{t}_{rs}$}
        \EndFor
        \State \Return $t_{r}$
	\EndFunction
    \end{algorithmic}
    \label{alg:LB-SFRVCP-T}
\end{algorithm}

\newpage
The correctness of the lower bound in line 10 of Algorithm \ref{alg:LB-SFRVCP-T} is proven in Proposition \ref{prop:lowerbound}. 

\begin{proposition} \label{prop:lowerbound} 
Let $(i,j)$ be an arc with $i,j \in \mathcal{I}$ and $\underline{k}^i := \arg\min_{k \in \mathcal{K}}{d_{ik}}$ be the CS yielding the minimum distance $\underline{d}_i$ from customer $i \in \mathcal{I}$. If the EV reaches the threshold $Q^{T}$ under scenario $s\in\mathcal{S}$ while traversing arc $(i,j)$, then the quantity:
\begin{equation} \label{eq:lowerboundlengthdetour}
    \max\{d_{ij}, \underline{d}_i + \underline{d}_j\}
\end{equation}
is a lower bound to the total length of any detour along arc $(i,j)$.
\end{proposition} 

\begin{proof} As distances are Euclidean, $d_{ij}$ is by definition the length of the shortest path from $i$ to $j$. It trivially follows that $d_{ij}$ is a lower bound to the total length of any detour along $(i,j)$; it may be a tight bound if the selected CS is on the arc $(i,j)$.

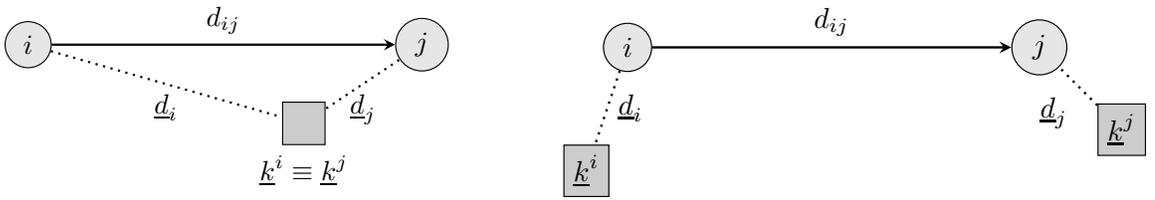
\begin{figure}[ht!]
  \centering
\begin{minipage}{0.44\textwidth}
	\begin{center}
		\resizebox{0.87\textwidth}{!}{
			\begin{tikzpicture}
[auto, thick,
  stazione/.style = {rectangle, draw, thin, fill = gray!40, scale = 2.3},
  cliente/.style  = {circle, draw, thin, fill = gray!20, scale = 1},
  detour/.style  = {circle, draw, thin, fill = black, scale = 0.5},
  legenda-stazione/.style = {rectangle, draw, thin, fill = gray!40, scale = 1},        
  legenda-cliente/.style  = {circle, draw, thin, fill = gray!20, scale = 0.75},         
  legenda-detour/.style  = {circle, draw, thin, fill = black, scale = 0.35},
]
\node [cliente] (customer-i) at (0,0)  {$i$};
\node [cliente] (customer-j) at (5,0)  {$j$};
\node [stazione] (station) at (3.5,-1) {};

\node at (3.5,-1.6) {$\underline{k}^i\equiv \underline{k}^j$};

\coordinate (intersection) at (3.5, 0);

\path [-stealth] (customer-i) edge node [above] {$d_{ij}$} (customer-j);

\path [dotted] (customer-i) edge node [below] {$\underline{d}_i$} (station);
\path [dotted] (station) edge node [below] {$\underline{d}_j$} (customer-j);

\end {tikzpicture}}

	\end{center}
\end{minipage} 
\hfill
\begin{minipage}{0.5\textwidth}
	\begin{center}
		\centering
		\resizebox{\textwidth}{!}{
\begin{tikzpicture}
[auto, thick,
  stazione/.style = {rectangle, draw, thin, fill = gray!40, scale = 1},
  cliente/.style  = {circle, draw, thin, fill = gray!20, scale = 1},
  detour/.style  = {circle, draw, thin, fill = black, scale = 0.5},
  legenda-stazione/.style = {rectangle, draw, thin, fill = gray!40, scale = 1},        
  legenda-cliente/.style  = {circle, draw, thin, fill = gray!20, scale = 0.75},         
  legenda-detour/.style  = {circle, draw, thin, fill = black, scale = 0.35},
]
\node [cliente] (customer-i) at (0,0)  {$i$};
\node [cliente] (customer-j) at (5,0)  {$j$};
\node [stazione] (station2) at (-0.5,-1.5) {$\underline{k}^i$}; 
\node [stazione] (station3) at (6,-1) {$\underline{k}^j$};

\coordinate (intersection) at (3.5, 0);

\path [-stealth] (customer-i) edge node [above] {$d_{ij}$} (customer-j);

\path [dotted] (customer-i) edge node [right] {$\underline{d}_i$} (station2);
\path [dotted] (station3) edge node [midway] {$\underline{d}_j$} (customer-j);
\end {tikzpicture}}
	\end{center}
\end{minipage} 
\caption{Relevant distances between nodes and CSs in the computation of the lower bound \eqref{eq:lowerboundlengthdetour}. Left panel: the case with $\underline{k}^i \equiv \underline{k}^j$. Right panel: the case with $\underline{k}^i \not\equiv \underline{k}^j$.}   \label{fig:LB}
\end{figure}

We first prove that $\underline{d}_i+ \underline{d}_j$ is also a valid lower bound. Recall that any detour from customer $i$ to customer $j$ must follow this order:
\begin{enumerate}
\item travel from node $i$ to the detour point;
\item travel from the detour point to a CS $k$;
\item travel from CS $k$ to node $j$.
\end{enumerate}
Since the detour must end in $j$ and we assume $\underline{k}^j$ to be the nearest CS to $j$, $\underline{d}_j$ is a lower bound on the length of the path from the chosen CS $k$ to $j$ (point 3). The same is not true for node $i$, since the vehicle will not travel directly from $i$ to the CS $k$, but the triangle inequality implies that $\underline{d}_i$ is a valid lower bound on the total distance covered from $i$ to the CS $k$ (points 1 and 2). This lower bound may be tight if $\underline{k}^i \equiv \underline{k}^j$ (see Figure \ref{fig:LB} left panel) and the detour happens at node $i$. Note that $\underline{d}_i + \underline{d}_j$ could be less than $d_{ij}$ if $\underline{k}^i$ and $\underline{k}^j$ are different (see Figure \ref{fig:LB} right panel), hence we use $\max\{d_{ij}, \underline{d}_i + \underline{d}_j\}$ as a stronger lower bound.
\end{proof}

\subsubsection{Set Partitioning} \label{sec_SP}
Once the pool of feasible routes $\Omega$ has been populated by the ILS procedure, the final solution of the SEVRP-T is obtained through the second component of the heuristic. As in \cite{montoya2017electric} and \cite{AndBar2019}, this second component consists of  solving a Set Partitioning problem considering the set of routes in  $\Omega$.

Let $x_r$ be a binary variable equal to 1 if route $r\in\Omega$ is selected, 0 otherwise. The expected recourse duration of route $r\in \Omega$ is denoted as $t_r$, while $\alpha_{ir}$ is a binary parameter equal to 1 if route $r\in\Omega$ visits customer $i\in\mathcal{I}$, 0 otherwise. The corresponding Set Partitioning problem on the set $\Omega$ is formulated as follows: 
\begin{eqnarray}
\text{minimize}  & & \sum_{r \in \Omega} t_r x_r   \label{obj_setcovering}  \\
\text{subject to:} \notag\\
& & \sum_{r \in \Omega} \alpha_{ir} x_r = 1  \qquad \forall i \in \mathcal{I}  \label{covering_setcovering}  \\
& & x_r \in \{0,1\} \qquad \forall r \in \Omega. \label{domain_setcovering}
\end{eqnarray} 

Objective~\eqref{obj_setcovering} selects a subset of routes from $\Omega$ that minimizes the total duration. Constraints~\eqref{covering_setcovering} ensure that each customer $i \in \mathcal{I}$ is visited exactly once. Finally, constraints~\eqref{domain_setcovering} set the domains of the decision variables $x_r$, with $r \in \Omega$. 

\section{A scenario reduction technique for the SEVRP-T} \label{sec_scen_red_theory}

The application of the ILS-SP heuristic outlined in Section \ref{sec:solmethod} heavily depends on the size of $\mathcal{S}$ (see line 3 in Algorithm \ref{alg:E-SFRVCP-T}). The larger $|\mathcal{S}|$ is, the greater the informative power of the scenario tree is, enabling to capture a broader range of energy consumption outcomes. However, this comes at the cost of increased computational effort. For this reason, in this section, we present the technique we have adopted for reducing the number of scenarios in the case of medium- and large-sized SEVRP-T instances. The scenario reduction technique is mainly adapted from \cite{DupGro-KusRom2003} and \cite{HeiRom2003}. Let $P$ be the original distribution defined over a finite set $\mathcal{S}$ of energy consumption scenarios. The technique determines a subset of scenarios $S^{red}\subset \mathcal{S}$ and the corresponding probability distribution $P^{red}$ with $|S^{red}|$ as an input. The set $S^{red}$ and the probability $P^{red}$ are determined with the objective of minimizing the probabilistic distance between $P$ and $P^{red}$, which is established by solving an appropriately defined optimal transportation problem. In what follows we detail the technique.

 Let $e_s:=[e_{ijs}]_{(i,j)\in \mathcal{A}}$ represent the vector of all realizations of the energy consumption across the entire network in scenario $s\in\mathcal{S}$. The initial probability distribution $P$ on $\mathcal{S}$ is written as $P:=\sum_{s \in \mathcal{S}}p_s \delta_{e_s}$, where $p_s$ is the probability of scenario $s\in\mathcal{S}$ and $\delta_{e_s}$ denotes the Dirac distribution at $e_s$. Given a subset $\mathcal{S}^{red}\subset \mathcal{S}$ of fixed cardinality $|\mathcal{S}^{red}|<|\mathcal{S}|$, we define $\mathcal{S}^{del}:=\mathcal{S}\setminus\mathcal{S}^{red}$. Specifically, $\mathcal{S}^{red}$ and $\mathcal{S}^{del}$ are the sets of the energy consumption scenarios which are preserved and deleted, respectively, after the application of the scenario reduction technique. In this regard, we construct the probability distribution $P^{red}$ based on scenarios $e_s$ with probabilities $p_{s}^{red}\in[0,1]$ for $s\in\mathcal{S}^{red}$, i.e., $P^{red}:=\sum_{s\in\mathcal{S}^{red}}p_s^{red}\delta_{e_s}$, with $\sum_{s\in\mathcal{S}^{red}} p_s^{red} = 1$. This is achieved by deleting all scenarios $s\in\mathcal{S}^{del}$ and by assigning new probabilities $p_s^{red}$ to each preserved scenario with $s\in\mathcal{S}^{red}$. For simplicity, let $p^{red}:=[p_s^{red}]_{s \in \mathcal{S}^{red}}$ be the vector of reassigned probabilities. Therefore, the objective of the scenario reduction technique is to optimally choose a subset $\mathcal{S}^{red}$ of energy consumption scenarios and a suitable probability $p^{red}$ from the feasible set $\{ \mathcal{S}^{red}\subset \mathcal{S} \text{ with } |\mathcal{S}^{red}| \text{ fixed}, p^{red}\geq 0, \sum_{s \in \mathcal{S}^{red}} p_s^{red}=1\}$ so that the distance $D(P,P^{red})$ between the probability distributions $P$ and $P^{red}$ is minimized. This is performed by solving the following discrete version of the \textit{optimal transportation problem} (see \cite{DupGro-KusRom2003}):
\begin{eqnarray*}
    \text{minimize}  & & D(P,P^{red})=\sum_{s\in\mathcal{S}} \sum_{\tilde{s} \in \mathcal{S}^{red}} \norm{e_s-e_{\tilde{s}}}_2 \eta_{s\tilde{s}}        \\
\text{subject to:}\\
& & \sum_{s \in \mathcal{S}}\eta_{s\tilde{s}}=p_{\tilde{s}}^{red}  \qquad \forall \tilde{s} \in \mathcal{S}^{red}    \\
& & \sum_{\tilde{s} \in \mathcal{S}^{red}}\eta_{s\tilde{s}}=p_s  \qquad  \forall s \in \mathcal{S}    \\
& & \eta_{s\tilde{s}} \geq 0,
\end{eqnarray*}

where $\eta_{s\tilde{s}}$ is the probability of assigning scenario $s\in\mathcal{S}$ to scenario $\tilde{s}\in\mathcal{S}^{red}$.
If $\mathcal{S}^{red}$ is predefined, the problem of determining optimal probabilities $p_s^{red}$ can be easily handled, as established in the following theorem (see Theorem 2.1 in \cite{HeiRom2003} or Theorem 2 in \cite{DupGro-KusRom2003}):
\begin{theorem}[\cite{HeiRom2003} and \cite{DupGro-KusRom2003}]
Let $\mathcal{S}^{red}$ be given. The minimum distance $D(P,P^{red})$ over the set $\{p^{red}\geq 0, \sum_{s \in \mathcal{S}^{red}} p_s^{red}=1\}$ is equal to:
\begin{equation}\label{mindist}
D(P,P^{red})=\sum_{\tilde{s} \notin \mathcal{S}^{red}}p_{\tilde{s}} \cdot \min_{s \in \mathcal{S}^{red}} \norm{e_s-e_{\tilde{s}}}_2.
\end{equation}
Additionally, the minimum is achieved at:
\begin{equation} \label{optimal_redistribution_rule}
p_s^{red}=p_s+\sum_{\tilde{s} \in \mathcal{S}_s^{del}}p_{\tilde{s}} \qquad \forall s \in \mathcal{S}^{red},
\end{equation}
where $\mathcal{S}_s^{del}:=\{\tilde{s} \in \mathcal{S}^{del}: s=s(\tilde{s})\}$ and $s(\tilde{s})\in \argmin_{s\in \mathcal{S}^{red}} \norm{e_s-e_{\tilde{s}}}_2$ for all $\tilde{s}\in\mathcal{S}^{del}$.
\end{theorem}

Equation \eqref{optimal_redistribution_rule} is known as \textit{optimal redistribution rule}. It states that the new probability of each retained scenario in $\mathcal{S}^{red}$ corresponds to the sum of its original probability $p_s$ and the probabilities of all deleted scenarios closest to it in terms of the $\ell_2$-distance.

On the other hand, the choice of optimally determining $\mathcal{S}^{red}$ with fixed cardinality $|\mathcal{S}^{red}|$ may be established by solving a \textit{set-covering} problem over $\mathcal{S}$, which is known to be NP-hard. However, efficient solution algorithms are available for the case of $|\mathcal{S}^{red}|=|\mathcal{S}|-1$ and $|\mathcal{S}^{red}|=1$.

When $|\mathcal{S}^{red}|=1$, the set covering problem simplifies to:
\begin{equation}\label{set_covering_forward_selection}
\min_{\tilde{s} \in \mathcal{S}} \sum_{s \in \mathcal{S}} p_s \norm{e_s-e_{\tilde{s}}}_2.
\end{equation}

If the minimum occurs at $\overline{s}$, only scenario $e_{\overline{s}}$ is retained and its probability will be $p_{\overline{s}}^{red}=1$, consistently with the redistribution rule \eqref{optimal_redistribution_rule}. We recursively solve problem \eqref{set_covering_forward_selection}, thus allowing us to consider cases in which $|\mathcal{S}^{red}|>1$, selecting more scenarios that will not be deleted. This procedure, called \textit{forward selection algorithm}, determines an index set $\{\overline{s}_1,\ldots,\overline{s}_{|\mathcal{S}^{red}|}\}$ such that:
\begin{equation}\label{formula_forward_selection}
\overline{s}_m \in \argmin_{\tilde{s} \in \mathcal{S}^{del}_{[m-1]}} \text{ } \sum_{o \in \mathcal{S}^{del}_{[m-1]}\setminus\{\tilde{s}\}} p_o \cdot \min_{s \notin \mathcal{S}^{del}_{[m-1]}\setminus\{\tilde{s}\}} \norm{e_o-e_s}_2 \qquad \forall m=1,\ldots,|\mathcal{S}^{red}|,
\end{equation}
where $\mathcal{S}^{del}_{[0]}:=\mathcal{S}$ and $\mathcal{S}^{del}_{[m-1]}:=\mathcal{S}\setminus\{\overline{s}_1,\ldots,\overline{s}_{m-1}\}$ for $m=2,\ldots,|\mathcal{S}^{red}|$. The procedure reflects the structure of $D(P,P^{red})$ in \eqref{mindist} and can be executed using the \textit{Fast Forward Selection} (FFS) algorithm (see Algorithm \ref{alg:FFS} in \ref{sec_appendix_FFS_algo}).

For the sake of illustration, Figure \ref{fig_scenario_reduction} presents an example where the FFS algorithm is applied to a one-dimensional scenario tree with $|\mathcal{S}|=50$ scenarios, each with the same probability $p_s=\frac{1}{50}$. The goal is to reduce the tree to a subset of $|\mathcal{S}^{red}|=10$ scenarios. The FFS algorithm selects the scenarios shown in black in Figure \ref{fig_scenario_reduction}a and assigns to them new probabilities (following equation \eqref{optimal_redistribution_rule}), which are  depicted in Figure \ref{fig_scenario_reduction}b.

\begin{figure}[ht!]
\centering
\subfloat[Original scenario tree (50 scenarios).] 
{\includegraphics[width=.5\textwidth]{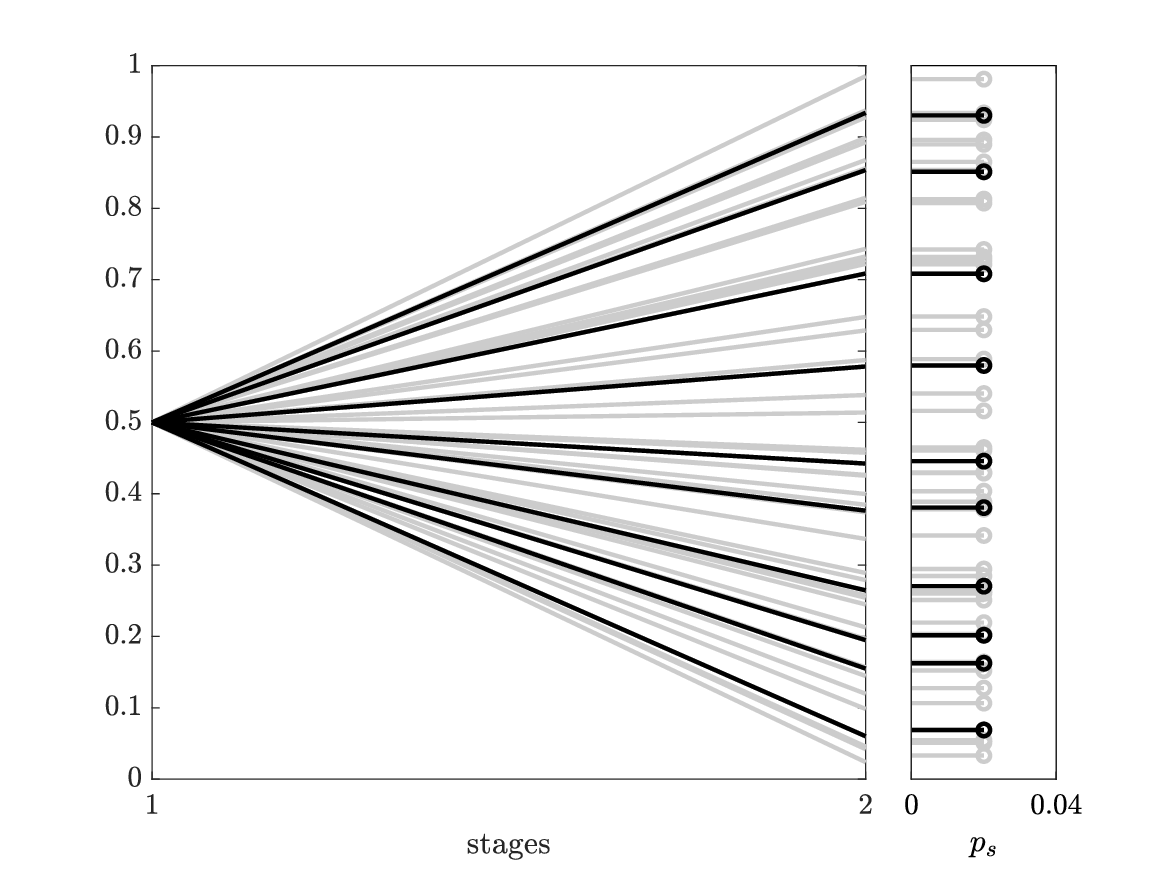}}
\subfloat[Reduced scenario tree (10 scenarios).]
{\includegraphics[width=.5\textwidth]{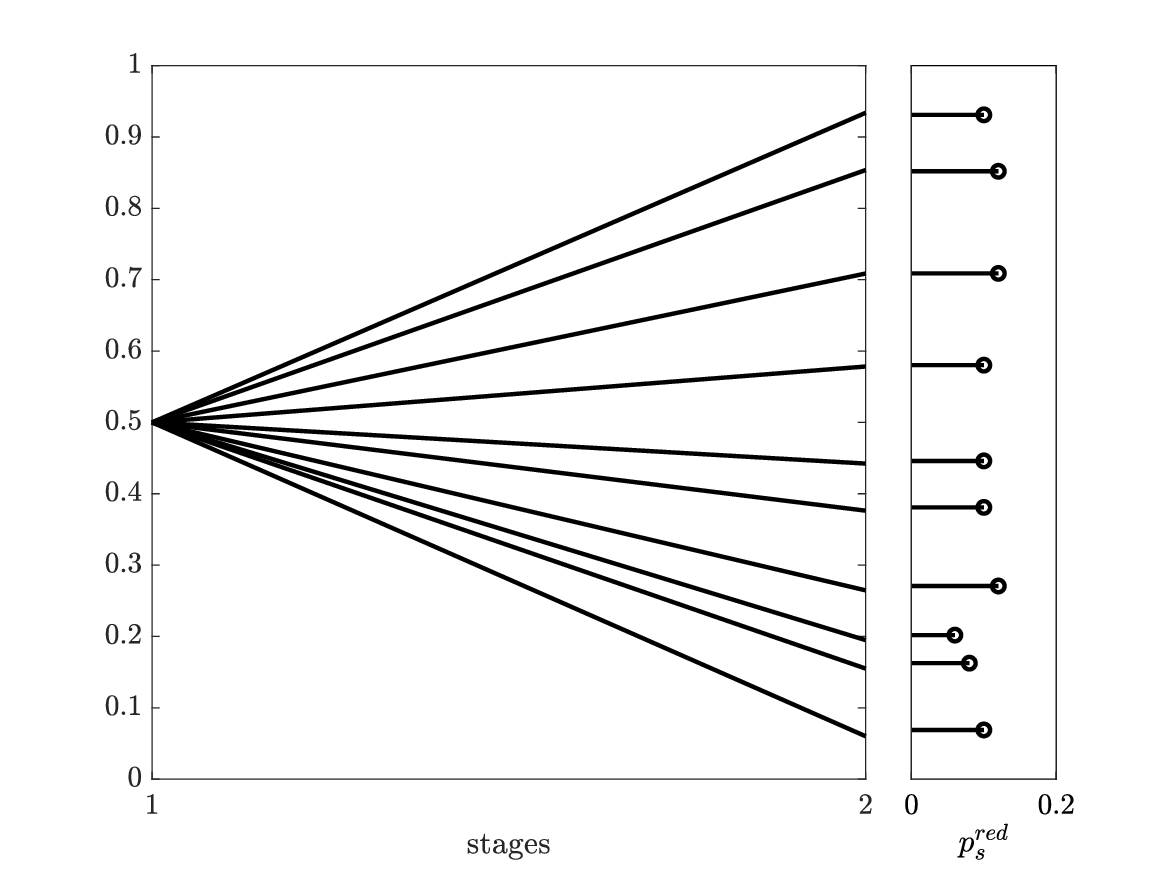}}\\
\caption{Example of the FFS algorithm application. Left panel: the original scenario tree with 50 scenarios and their corresponding probabilities. The 10 selected scenarios are highlighted in black. Right panel: the reduced scenario tree, consisting of the 10 scenarios selected by the FFS algorithm, along with their reassigned probabilities.}
\label{fig_scenario_reduction}
\end{figure}

\section{Computational results} \label{sec:results}
In this section, we present and discuss the results of the numerical experiments with the aim of:
\begin{itemize}
    \item Evaluating the performance of the proposed heuristic in both deterministic and stochastic settings;
    \item Validating the stochastic model in terms of in-sample stability (see \cite{KauWal2007});
    \item Measuring the impact of uncertainty and the quality of the deterministic solution in a stochastic setting (see \cite{birge2011introduction});
    \item Performing a sensitivity analysis on varying input energy values.
\end{itemize}

The model and the heuristic were implemented in Python (v. 3.12.2) and solved using Gurobi (v. 11.0). All computational experiments were run on a MacBookPro17.1 with an Apple M1 processor of 8 cores and 16 GB of RAM memory. Unless otherwise specified, a runtime limit of three hours (10800 seconds) is imposed on the solver and for the heuristic.

In the following, we start describing the instances and the parameters, together with the scenario generation procedure we opted for (see Section \ref{sec_instance_parameters}). Then, we assess the performance of the heuristic algorithm  both in deterministic and stochastic settings (see Section \ref{sec_performance_heuristic_DET} and Section \ref{sec_performance_heuristic_STOCH}, respectively). Finally, we discuss managerial insights and the effect of different energy thresholds on the SEVRP-T (see Section \ref{sec_managerial_insights}).

\subsection{Instances and parameters} \label{sec_instance_parameters}
We considered a selection of benchmark instances taken from~\cite{montoya2017electric} (available from \url{http://www.vrp-rep.org/}). We used instances composed by 10, 20, 40 and 80 customers. Additionally, we created a new set of twenty instances, each having  15 customers, by removing the last five customers from the 20-customer instances. Instances are named using the following coding scheme: tcAcBsCcDE. Specifically, A is the method used to locate customers with 0: random uniform, 1: clustered, 2: mixture of both; B is the number of customers; C is the number of CSs; D is equal to ``t'' if the CSs are located using a \textit{p-median} heuristic and ``f'' if the CSs are randomly located; E is the number of the instance for each combination of parameters. Three types of CSs are considered (slow, moderate, and fast). 

As in \cite{montoya2017electric} and \cite{froger2022electric}, we considered an EV with a consumption rate of 0.125 kWh per unit of distance. However, we assumed  $Q^{max} = 24$, as opposed to a battery capacity of  16 kWh considered in  \cite{montoya2017electric} and \cite{froger2022electric}. This choice was made to avoid infeasibility issues due to uncertain nature of the energy consumption. Other relevant parameters are fixed as follows: number $n$ of copies of each CS equal to 3; $Q^T = 30\%~Q^{max}=7.2$ kWh; $Q^G = 80\%~Q^{max}=19.2$ kWh; $M$ is set to greatest distance between two nodes. In Algorithm \ref{alg:ILS-P} we impose a maximum number $I_{max}$ of ILS iterations equal to 2000, and in Algorithm \ref{alg:VND} parameter$\gamma=1$. Finally, we have removed instances tc1c10s2cf3 and tc1c10s3cf3 due to infeasibility issues arising with a large number of scenarios.
\subsubsection{Scenario generation} \label{sec_scen_gen}
In this section, we discuss the procedure we adopted to generate scenario trees for the uncertain energy consumption. Considering the energy consumption $\hat{e}_{ij} \in \mathbb{R}^+$ from the \cite{montoya2017electric}, we have created $|\mathcal{S}|$ energy consumption scenarios for each arc. We recall that we denote by $e_{ijs}$ the energy consumption along arc $(i,j)\in\mathcal{A}$ under scenario $s$ and $p_s$ its probability of occurrence. If $|\mathcal{S}|=1$, then $e_{ij1} = \hat{e}_{ij}$ for all $ i,j \in \mathcal{V}$. On the other hand, if $|\mathcal{S}|>1$, scenarios were sampled according to a Monte Carlo procedure (see \cite{Shap20003}) from three possible probability distributions: Uniform ($U$), truncated Normal ($N$) and truncated Exponential ($E$). The energy consumption on each arc $(i,j)\in\mathcal{A}$ was sampled independently. To ensure results reproducibility, we considered the same initial random seed. We assumed that the three distributions have the same mean equal to the deterministic energy consumption, i.e., $\mu^U_{ij} = \mu^N_{ij} = \mu^E_{ij} = \hat{e}_{ij}$. For the Uniform distribution we randomly sampled from the interval $[0.75 \hat{e}_{ij}, 1.25 \hat{e}_{ij}]$, such that a maximum deviation of 25\% from the average consumption $\hat{e}_{ij}$ is allowed. For the Uniform distribution, this corresponds to a variance equal to $\sigma^U_{ij} = \frac{(0.5\hat{e}_{ij})^2}{12}$. For the Normal and Exponential distributions, we considered the same variance of the Uniform case, i.e., $\sigma^N_{ij} = \sigma^E_{ij} = \sigma^U_{ij}$. Lastly, we assumed $p_s=\frac{1}{|\mathcal{S}|}$.

\subsection{Performance of the ILS-SP heuristic in the deterministic setting} \label{sec_performance_heuristic_DET}
In this section, we assess the performance of our heuristic on a selection of benchmark instances, assuming $|\mathcal{S}|=1$ with $e_{ij1}=\hat{e}_{ij}$. We compare the performance of the mathematical model solved by Gurobi and the ILS-SP heuristic through different indicators in Table \ref{tab_sevrp-deterministic_10_15_20}: the number of instances solved to optimality out of  total number of instances (\# Opt / \# Inst), CPU time, number of routes, optimality gap provided by Gurobi and gap between the objective function values $z_{algo}$ and $z_{model}$ computed as follows:
\begin{equation}\label{gap_algo_model}
Gap :=    \frac{z_{algo}-z_{model}}{z_{model}}.
\end{equation}
If Gurobi is not able to solve the model within the runtime limit, we set  $z_{model}$ to the objective function value achieved by the incumbent solution found by the solver. Otherwise, $z_{model}$ corresponds to the optimal value.

In Table \ref{tab_sevrp-deterministic_10_15_20} we outline the average results of instances with 10, 15 and 20 customers. Detailed results are reported in tables \ref{tab_sevrp-deterministic_10}--\ref{tab_sevrp-deterministic_20} in \ref{sec_appendixA_deterministic}. We note that all the 10-customer instances are solved at optimality by the model as well as the heuristic. Conversely, in five out of twenty and eighteen out of twenty instances with 15 and 20 customers, respectively, the model is not able to find an optimal solution within the runtime limit, resulting in an average Gurobi gap of 6.95\% and 12.14\%, respectively. However, in all 15-customer and 20-customer instances solved to optimality by the solver, the heuristic found the optimal solutions as well. Moreover, the solutions provided by the heuristic are all either equal or better than the best feasible solution found by Gurobi.

In the three sets of instances, solving the model with the heuristic compared to Gurobi provides an average CPU time saving of 84.08\%, 99.14\% and 99.15\%, respectively. From this analysis, we conclude that the proposed ILS-SP heuristic provides high-quality solutions for small-sized instances for the deterministic version of  SEVRP-T, within reasonable computational times.

\begin{table}[ht!]
  \centering
  \caption{Comparison between Algorithm \ref{alg:ILS-P} and Gurobi solver for instances with $|\mathcal{I}|=\{10,15,20\}$ customers in the deterministic setting ($|\mathcal{S}| = 1$). Average results are reported.}
  \begin{tabular}{|c|c|c|c|c|c|c|c|}
    \hline
    \multirow{2}{*}{$|\mathcal{I}|$} & \multirow{2}{*}{\# Opt / \# Inst} & \multicolumn{2}{|c|}{Time (s)} &  \multicolumn{2}{|c|}{Number of routes} &  \multicolumn{2}{|c|}{Gap}
    \\
    \cline{3-8}
     & & Algorithm & Model & Algorithm & Model & Gurobi & Algorithm-Model
     \\
    \hline
10 & 18 / 18 & 11.40	& 71.59 & 1.56 & 1.56 & 0.00\% & 0.00\%
\\
15 & 15 / 20 & 39.06 & 4552.97 & 1.40	& 1.40 & 6.95\% & 0.00\%
\\
20 & 2 / 20 & 91.49 & 10709.27 & 1.65 & 1.65 &	12.14\% & $-0.29\%$
\\
\hline
  \end{tabular}
  \label{tab_sevrp-deterministic_10_15_20}
\end{table}

\subsection{Performance of the ILS-SP heuristic in the stochastic setting} \label{sec_performance_heuristic_STOCH}
In this section, we discuss the performance of the ILS-SP heuristic when solving the SEVRP-T under uncertain energy consumption. We start by comparing the two different approaches in solving the fixed route vehicle charging problem SFRVCP-T outlined in Section \ref{sec:fixroute}. Then, we assess the performance of our heuristic over different energy consumption probability distributions. Finally, we provide the results of an in-sample stability analysis (see Section \ref{sec_insample_stability}) and measure the impact of uncertainty on the problem under investigation (see Section \ref{sec_impact_of_uncertainty}).

In Table~\ref{table:comparison1} we report the computational time of the fixed route model outlined in Section~\ref{sec:fixroute} solved by Gurobi versus Algorithm \ref{alg:E-SFRVCP-T} (E-SFRVCP-T($r$)), over a set of 12 fixed routes with a number of arcs $l=\{3,4,\ldots,15\}$. These fixed routes are obtained by randomly sampling sequences of $l-1$ customers from SEVRP-T instances. We consider $|\mathcal{S}|= \{1,5,20,50,100\}$ for each fixed route.

We observe that Algorithm~\ref{alg:E-SFRVCP-T} is several orders of magnitude faster than the corresponding SFRVCP-T model solved by Gurobi, consistently across all sampled routes. Indeed, the average saving in terms of CPU time is of 99.87\%.

\begin{table}[ht!]
\centering
\caption{CPU time (in milliseconds) comparison between the exact Algorithm \ref{alg:E-SFRVCP-T} and Gurobi when solving the SFRVCP-T.}
\begin{tabular}{|c|ccccc|ccccc|}
\hline
 & \multicolumn{5}{c|}{\textbf{SFRVCP-T solved by Gurobi (ms)}} & \multicolumn{5}{c|}{\textbf{SFRVCP-T solved by E-SFRVCP-T (ms)}} \\
 & \multicolumn{5}{c|}{$|\mathcal{S}|$} & \multicolumn{5}{c|}{$|\mathcal{S}|$} \\
$l$ & 1 & 5 & 20 & 50 & 100 & 1 & 5 & 20 & 50 & 100 \\
\hline
3   & 84.8  & 195.7 & 613.2 & 1343.7 & 2690.5 & 0.0 & 0.3 & 1.0 & 2.4  & 4.9  \\
4   & 52.2  & 151.3 & 547.2 & 1232.0 & 2445.2 & 0.0 & 0.1 & 0.5 & 1.4  & 2.7  \\
5   & 39.1  & 121.2 & 389.5 & 897.6  & 1722.8 & 0.0 & 0.1 & 0.5 & 1.1  & 2.3  \\
6   & 68.7  & 230.6 & 918.5 & 2061.5 & 4068.5 & 0.1 & 0.2 & 1.0 & 2.7  & 5.5  \\
7   & 101.1 & 312.8 & 1219.2 & 2639.2 & 5331.2 & 0.1 & 0.5 & 1.8 & 4.5  & 9.0  \\
8   & 98.2  & 342.6 & 1295.0 & 2911.1 & 5906.4 & 0.1 & 0.5 & 1.8 & 4.2  & 8.7  \\
9   & 116.5 & 456.6 & 1728.7 & 3912.7 & 7843.8 & 0.2 & 0.7 & 2.7 & 6.7  & 13.0 \\
10  & 123.2 & 466.1 & 1724.5 & 3925.2 & 7962.7 & 0.1 & 0.6 & 2.1 & 5.2  & 10.3 \\
11  & 153.7 & 552.6 & 2004.6 & 4676.7 & 9340.6 & 0.2 & 0.7 & 2.7 & 6.8  & 14.0 \\
12  & 169.2 & 646.1 & 2599.4 & 5957.0 & 12160.2 & 0.2 & 1.0 & 3.7 & 9.2  & 18.4 \\
13  & 151.5 & 575.9 & 2242.4 & 5115.9 & 10368.8 & 0.2 & 0.8 & 3.1 & 7.0  & 14.0 \\
14  & 319.3 & 978.8 & 3447.1 & 8501.7 & 17336.0 & 0.3 & 1.6 & 6.1 & 14.9 & 31.5 \\
15  & 166.7 & 768.1 & 2916.4 & 6879.5 & 14122.0 & 0.3 & 1.1 & 4.0 & 9.9  & 19.5 \\
\hline
avg & 126.4 & 446.0 & 1665.0 & 3850.2 & 7792.2 & 0.1 & 0.6 & 2.3 & 5.8  & 11.8 \\
\hline
\end{tabular}
\label{table:comparison1}
\end{table}

An analysis of the quality of the solutions in the stochastic setting is presented in Table \ref{tab_summary_10c}. We consider the 10-customer instances and provide a comparison between the ILS-SP heuristic (Algorithm \ref{alg:ILS-P}) and the solver, in terms of the number of instances solved to optimality out of  total number of instances (\# Opt / \# Inst), CPU time, number of routes and Gap. We consider an increasing number of scenarios $|\mathcal{S}|$ = \{5, 10, 20, 50\}, and the three probability distributions of the energy consumption defined in Section \ref{sec_scen_gen}. Detailed results are reported in tables \ref{tab_sevrp-stoch_c10_S5_unif}--\ref{tab_sevrp-stoch_c10_50_exponential} in \ref{sec_appendixB_stochastic}. 

For simplicity, in the following we describe the results obtained in the case of uniform distribution. Similar findings can be drawn for the normal and exponential probability distributions. First, we notice that when $|\mathcal{S}|=\{5,10\}$, on average more than 60\% of the instances are solved to optimality by Gurobi. In the cases not solved to optimality, the Gurobi gap is relatively low, i.e., 0.68\% and 4.12\%, respectively. Conversely, for larger values of $|\mathcal{S}|$, Gurobi is not able to guarantee optimality within the runtime limit for all instances. On the other hand, the ILS-SP heuristic provides better solutions in less than four minutes. 

When we compare the outcomes of the three probability distributions, we note that all the indicators are similar. We only observe that, when the size of the scenario tree is relatively small (i.e., $|\mathcal{S}|=\{5,10\}$), the number of routes for the exponential case is the lowest one, followed by the normal and the uniform distributions. 

\begin{table}[ht!]
  \centering
  \caption{Aggregate results for ILS-SP heuristic (Algorithm \ref{alg:ILS-P}) and Gurobi for 10-customer instances and $|\mathcal{S}| = \{5,10,20,50\}$.}
      \resizebox{\textwidth}{!}{
  \begin{tabular}{|c|c|c|c|c|c|c|c|c|}
    \hline
    \multirow{2}{*}{Distribution} & \multirow{2}{*}{$|\mathcal{S}|$} & \multirow{2}{*}{\# Opt / \# Inst} & \multicolumn{2}{|c|}{Time (s)} &  \multicolumn{2}{|c|}{Number of routes} &  \multicolumn{2}{|c|}{Gap}
    \\
    \cline{4-9}
     & & & Algorithm & Model & Algorithm & Model & Gurobi & Algorithm-Model
     \\
    \hline
\multirow{4}{*}{Uniform} & 5 & 16 / 18	& 29.49 & 1936.69 & 1.78 & 1.78 & 0.68\% &	0.00\%
\\
& 10 & 8 / 18& 52.46 & 7168.62 & 1.89 & 1.89 &	4.12\% &	$-0.17\%$
\\
& 20 & - / 18& 90.38 & 10800.00	& 1.94 & 1.94 & 12.18\% & 	$-0.59\%$
\\
& 50 & - / 18& 208.43 & 10800.00 & 1.94 & 2.06 & 24.17\% & $-5.47\%$
\\
\hline
\multirow{4}{*}{Normal} & 5 & 17 / 18 & 29.63 & 1443.24 & 1.72 & 1.72 & 0.45\% &	0.00\%
\\
& 10 & 7 / 18 & 51.22 & 6934.15 &	1.72 & 1.72 & 4.53\% & $-0.24\%$
\\
& 20 & - / 18& 96.53 & 10800.00 & 2.00 & 2.00 &	11.69\% & $-0.67\%$
\\
& 50 & - / 18& 211.86	& 10800.00 & 2.00 & 1.94 & 23.06\% & $-5.27\%$
\\
\hline
\multirow{4}{*}{Exponential} & 5 & 17 / 18 & 29.09 &	1574.14 &	1.61 &	1.61 &	1.18\%	& 0.00\%
\\
& 10 & 8 / 18 & 49.51	& 6941.21	& 1.67	& 1.67& 	4.45\% &	$-0.25\%$
\\
& 20 & - / 18& 90.18 & 10800.00 & 1.78 & 1.78 & 11.88\% & $-0.79\%$
\\
& 50 & - / 18& 199.29 & 10800.00 & 2.00 & 1.94 & 24.02\% & $-6.07\%$
\\
\hline
  \end{tabular}}
  \label{tab_summary_10c}
\end{table}

\subsubsection{In-sample stability} \label{sec_insample_stability}
In this section, we discuss the minimum number of scenarios guaranteeing stability and the impact of adopting the scenario reduction technique, presented in Section \ref{sec_scen_red_theory}, when dealing with the SEVRP-T. The results of an in-sample stability analysis (see \cite{KauWal2007}) over the 10-customer instances for the three considered distributions are displayed in Figure \ref{fig_insample_stability}. The detailed results are reported in Table \ref{tab_insample_stability_analysis} in \ref{sec_appendixB_stochastic}. We consider the following sizes of the scenario tree: $|\mathcal{S}|=\{5, 10, 20, 50, 100\}$. As the objective function value, we consider the one provided by the solver whenever Gurobi finds an optimal solution within the runtime limit. Otherwise, we use the value obtained through the ILS-SP heuristic. As empirically shown in the previous sections, these two values coincide if the model is solved to optimality by the solver. 

Regardless of the probability distributions, three different sets of six instances each can be identified, represented in Figure \ref{fig_insample_stability} with solid, dash-dotted and dotted line, respectively. The first set (solid line) converges to an average objective function value of $12.14$. The second set (dash-dotted line) exhibits a higher variability, with three different subsets. Indeed, three instances (tc0c10s3ct1, tc0c10s2ct1, tc1c10s3ct3) converge to $10.11$, two instances (tc2c10s3cf0, tc2c10s2cf0) converge to $11.10$, while the remaining one (tc1c10s2ct3) reaches an average value of $11.61$. Finally, the third set (dotted line) shows a similar behavior to the first one, with an average value of $7.81$ with 100 scenarios. We identify two outliers, tc1c10s2ct3 and tc1c10s3ct3, in the case of uniform and exponential distribution, respectively. We notice that, in the 83\% of the instances (fifteen out of eighteen), after 20 scenarios the objective function value becomes fairly stable on average. Therefore, we decided to consider this number as benchmark for the scenario tree size in the following discussion.

\begin{figure}[ht!]
\centering
\subfloat[Uniform distribution.]
{\includegraphics[width=.5\textwidth]{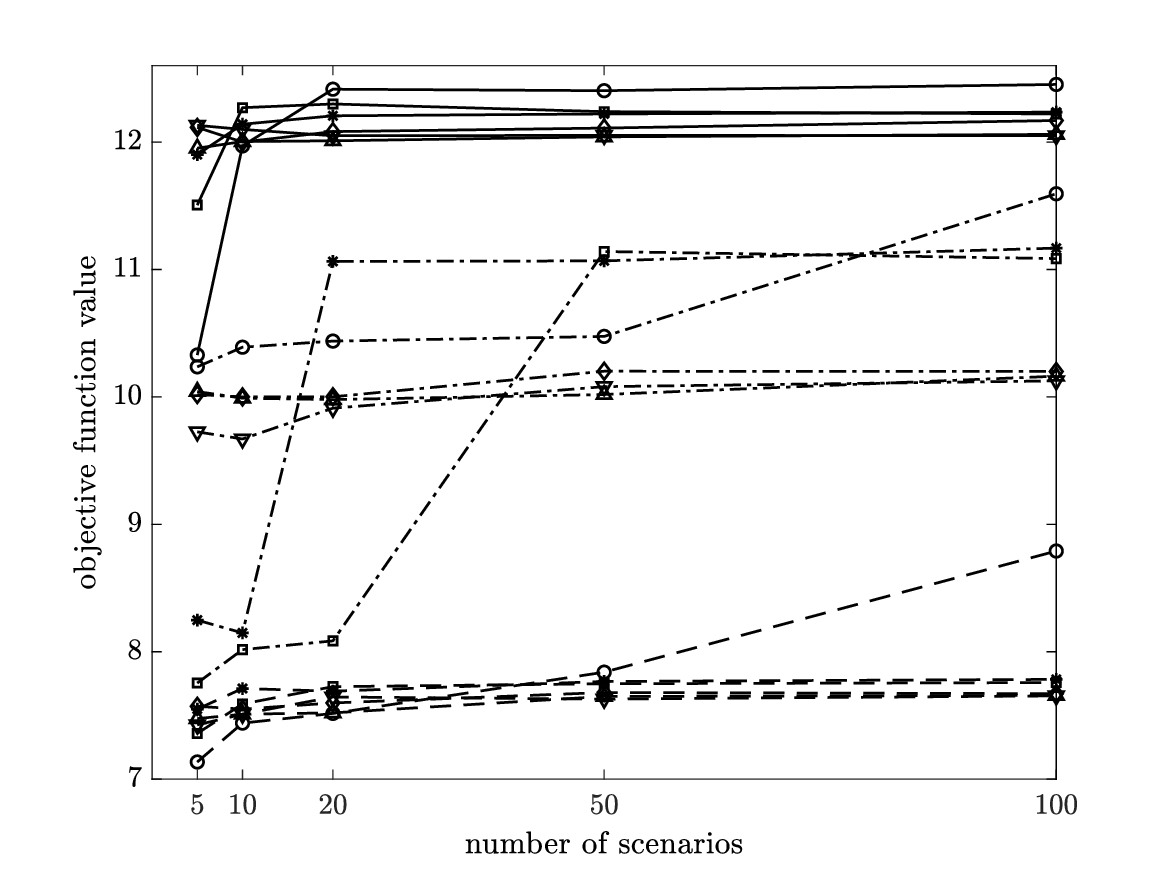}}
\subfloat[Normal distribution.]
{\includegraphics[width=.5\textwidth]{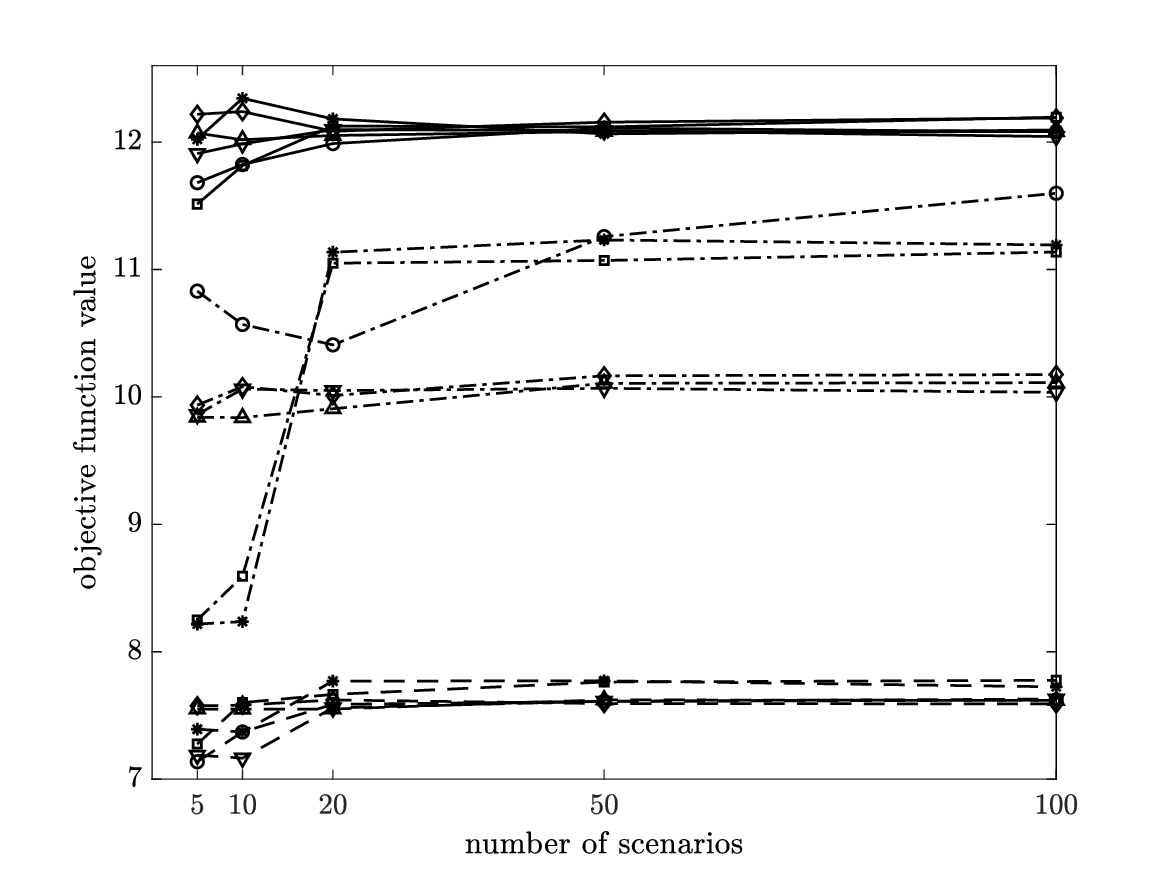}}\\
\subfloat[Exponential distribution.]
{\includegraphics[width=.5\textwidth]{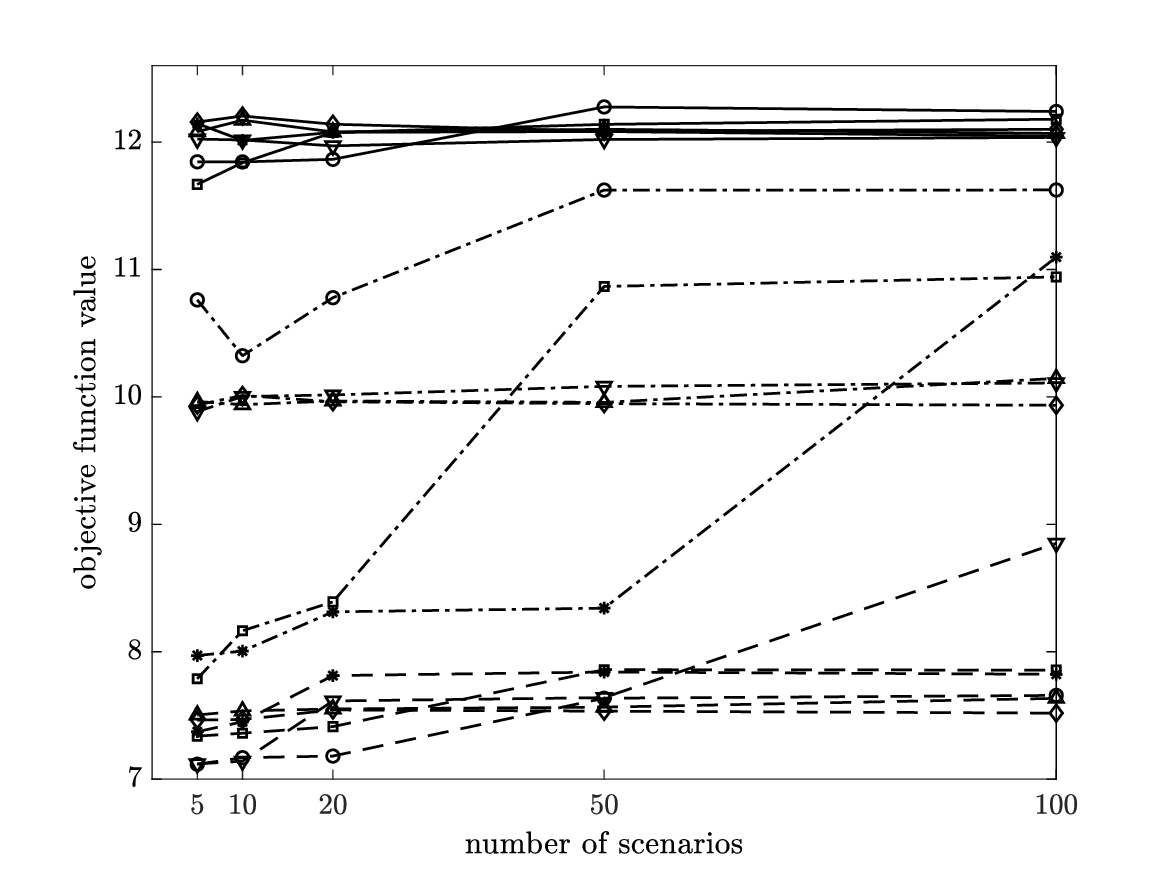}}
{\includegraphics[width=.5\textwidth]{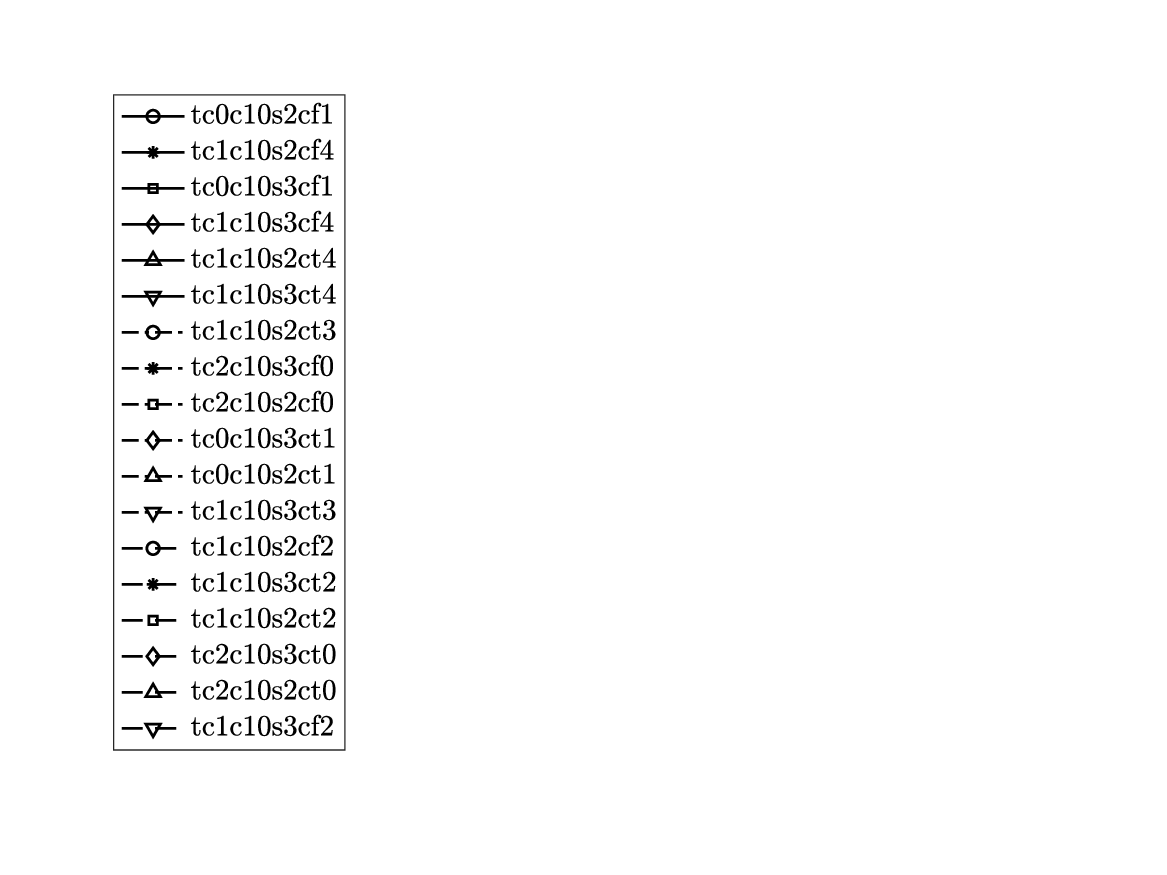}}
\caption{Approximate in-sample stability analysis for different energy consumption probability distributions.}
\label{fig_insample_stability}
\end{figure}

In Table \ref{tab_summary_40c_80c} we report the results on medium- and large-sized instances of 40 and 80 customers obtained through the ILS-SP heuristic and with a uniformly distributed energy consumption (see tables \ref{tab_sevrp-stochastic_40c_s1s10}-\ref{tab_sevrp-stochastic_80c_50sr20} in \ref{sec_appendixC} for details). We observe that increasing the size of the scenario tree generally leads to a higher average objective function value. However, for larger cases, the runtime limit is often reached, resulting in a limited number of iterations of the ILS-SP heuristic. This is particularly pronounced when the scenario tree size is $|\mathcal{S}|=50$. To address this issue, we apply the \textit{forward selection algorithm} as a scenario reduction technique described in Section \ref{sec_scen_red_theory} and implemented in Python via the \textit{ScenarioReducer} library (see \cite{Gioia2023}). Specifically, we reduce a scenario tree with 50 scenarios to a smaller one with 20 scenarios. These results  are shown in  Table \ref{tab_summary_40c_80c} and denoted with the symbol $\ddagger$. From a computational perspective, the CPU time,  the  number of iterations of the ILS-SP heuristic, and the number of routes are comparable with the base case of 20 scenarios on average for the cases with 40 and 80 customers. We note that the average objective function value obtained with 20 reduced scenarios is slightly higher than that of the base case of 20 scenarios and more inline with the  results obtained using the 50 scenarios. However, in the latter case, the CPU time is often close or even equal to the runtime limit of 10800s on average. For this reason, we conclude that the application of scenario reduction technique is effective in this context since it provides good solutions by enabling the heuristic to execute a larger number of iterations.

\begin{table}[ht!]
  \centering
  \caption{Solutions obtained through the ILS-SP heuristic (Algorithm \ref{alg:ILS-P}) for 40- and 80-customer instances, with a uniformly distributed energy consumption. Average results are reported. $\ddagger$ results for the scenario reduction case from 50 to 20.}
      \resizebox{0.9\textwidth}{!}{
  \begin{tabular}{|c|c|c|c|c|c|c|c|}
    \hline
    Number of customers & $|\mathcal{S}|$ &  Time (s) & Number of iteration &  Number of routes &  Objective
    \\
    \hline
\multirow{5}{*}{40} & 1	& 731.68 & 2000.00 & 1.90 & 17.89
\\
& 10	& 2593.75 &	2000.00 & 	2.10 & 18.29
\\
& 20	& 3972.28 & 1969.00	 & 2.30	& 18.57
\\
& 20{$\ddagger$}	& 4029.39 & 1969.00 &	2.30 & 18.87
\\
& 50	& 10613.88	& 1252.10	& 2.40 & 19.36
\\
\hline
\multirow{3}{*}{80} & 20 & 10800.00 &	343.00 &  2.75	& 26.85
\\
& 20{$\ddagger$} & 10800.00 & 360.85 & 2.75 & 27.16
\\
& {50} & {10800.00} & {174.35} & {2.65} & {28.19}
\\
\hline
  \end{tabular}}
  \label{tab_summary_40c_80c}
\end{table}

\subsubsection{The impact of uncertainty} \label{sec_impact_of_uncertainty}
To quantify the quality of the stochastic solution, we conducted an analysis with classical stochastic measures (see \cite{MagWal2012,CraMagPerRei2018}).  Based on the results from Section \ref{sec_insample_stability}, we perform the current analysis with $|\mathcal{S}|=20$. The average results are reported in Table \ref{tab_summary_stochmeas10c}, while details are in Table \ref{tab_stochastic_measures_10c_20scen} in \ref{sec_appendixB_stochastic}. The results on the so-called Recourse Problem (RP, see \eqref{eq:obj}--\eqref{eq:StocDomainZW}) are obtained through the ILS-SP heuristic, while all other results were obtained by the solver, since it converged to optimality in these cases.

Considering uniformly distributed energy consumption, on average the \%Expected Value of Perfect Information (\%EVPI) is 6.91\% of the optimal objective function value of the RP (see \eqref{eq:obj}--\eqref{eq:StocDomainZW}). This means that a decision maker would be ready to pay at most 6.91\% of the total cost for obtaining  perfect information on the energy consumption at the beginning of planning. Similar conclusions can be drawn for the other two probability distributions.

As a simpler approach, the decision maker may replace the energy consumption rate by its expected value as the average of the sampled values, and solve the  Expected Value Problem (EVP). The Value of Stochastic Solution (VSS) measures the expected cost from solving the stochastic model RP (see \eqref{eq:obj}--\eqref{eq:StocDomainZW}) rather than 
the EVP. We compute it as the difference between the EEV and the RP, where the EEV is the Expected result of using the EVP solution, denoting the objective function value of the RP model having the first-stage decision variables fixed at the optimal values obtained by solving the EVP (see \cite{MagAllBer2014}). In twelve out of eighteen instances the VSS is equal to infinity. This is due to the fact that, when performing the route prescribed by the EVP in the stochastic setting, we are not able to find a feasible solution for the threshold recourse policy.  Therefore, ignoring energy consumption uncertainty while planning routes may cause infeasibility.

\begin{table}[ht!]
  \centering
  \caption{Stochastic measures with 10-customer instances and $|\mathcal{S}| = 20$. Average results are reported.}
  \resizebox{0.6\textwidth}{!}{
  \begin{tabular}{|c|c|c|c|c|c|c|}
    \hline
    Distribution & RP &  WS & \%EVPI & EVP & EEV & \%VSS 
    \\
\hline
    Uniform & 9.90 & 9.18 & 6.91\% & 9.24 & inf & inf
\\
\hline
Normal & 10.05 & 9.22 & 7.83\% & 9.25 & inf & inf
\\
\hline
Exponential & 9.71 & 9.09 & 6.17\% & 9.25 & inf & inf
\\
\hline
  \end{tabular}}
  \label{tab_summary_stochmeas10c}
\end{table}

\subsection{Managerial insights} \label{sec_managerial_insights}
We conclude our analysis by providing some managerial insights. Specifically, we discuss the results of a sensitivity analysis on $Q^T$ and $Q^G$, which are two key SEVRP-T parameters. We consider the 40-customer instances with $|\mathcal{S}| = 20$, which are reduced from an initial set of 50 scenarios following a uniform energy consumption distribution. Average results are reported in Table \ref{tab_summary_maninsight_40c}, while extensive results are in tables \ref{tab_sevrp-stochastic_detailed_Qt}--\ref{tab_sevrp-stochastic_detailed_Qg} in \ref{sec_appendixD_managerial}.

\begin{table}[ht!]
  \centering
  \caption{Average results, considering uniform energy consumption distributions, obtained by varying the values of $Q^{T}$ and $Q^{G}$ through the ILS-SP heuristic (Algorithm \ref{alg:ILS-P}) for instances with 40 customers and $|\mathcal{S}| = 20$, reduced from an initial set of 50 scenarios through the FFS procedure (Algorithm \ref{alg:FFS}).}
      \resizebox{0.9\textwidth}{!}{
  \begin{tabular}{|c|c|c|c||c|c|c|}
\hline
& \multicolumn{3}{|c||}{$Q^T$} & \multicolumn{3}{|c|}{$Q^G$}
    \\
    & $20\%Q^{max}$ & $30\%Q^{max}$ & $40\%Q^{max}$ & $70\%Q^{max}$ & $80\%Q^{max}$ & $90\%Q^{max}$\\
\cline{2-7}
Number of routes & 2.45 & 2.30 & 2.10 & 2.05 & 2.30 & 3.16
\\
\hline
Objective & 18.83 & 18.87 & 20.27 & 18.66 & 18.87 & 20.40 
\\
\hline
  \end{tabular}}
  \label{tab_summary_maninsight_40c}
\end{table}

First, we account for three different possibilities of $Q^T$: $20\%Q^{max}$, $30\%Q^{max}$ and $40\%Q^{max}$. Increasing the value of $Q^T$ causes more frequent stops of the EVs, imposed by a stricter threshold recourse policy. This implies higher durations due to the charging operations, as confirmed by the results in Table \ref{tab_summary_maninsight_40c} (see the second, third and fourth columns).

We also consider $Q^G$ values of $70\%,80\%,90\%$ of $Q^{max}$. We recall that $Q^G$ is the SoC at the subsequent node after each detour. Therefore, when $Q^G$ is close to $Q^{max}$, the EVs require more time to be recharged, given that the slope of the nonlinear charging function is less steep in this final phase. Indeed, in our instances, the average percentage increase of the total time when $Q^G$ goes from $70\%$ to $90\%$ of $Q^{max}$ is 9.32\%. We note that an infeasibility issue occurs in one instance (tc1c40s5cf1, see Table \ref{tab_sevrp-stochastic_detailed_Qg} in \ref{sec_appendixD_managerial}) when $Q^G=90\%Q^{max}$. This arises because of the infeasibility of reaching nodes with a high  levels of $Q^G$.

Finally, we observe that increasing the value of $Q^T$ results in a decrease in the average number of routes. Conversely, increasing $Q^G$ leads to an increase in the average number of routes. This observation is not directly related to the specific application under investigation but rather mainly depends on the combinatorial nature of the problem. For the sake of illustration, we depict in Figure \ref{fig_managerialinsight} the first-stage routes obtained for the instance tc0c40s8cf0 when $Q^T=20\%Q^{max}$ (left panel, four routes) and $Q^T=40\%Q^{max}$ (right panel, two routes).

\begin{figure}[ht!]
\centering
\hspace*{-1cm}
\subfloat[{$Q^T=20\%Q^{max}$.}]
{\includegraphics[width=.6\textwidth]{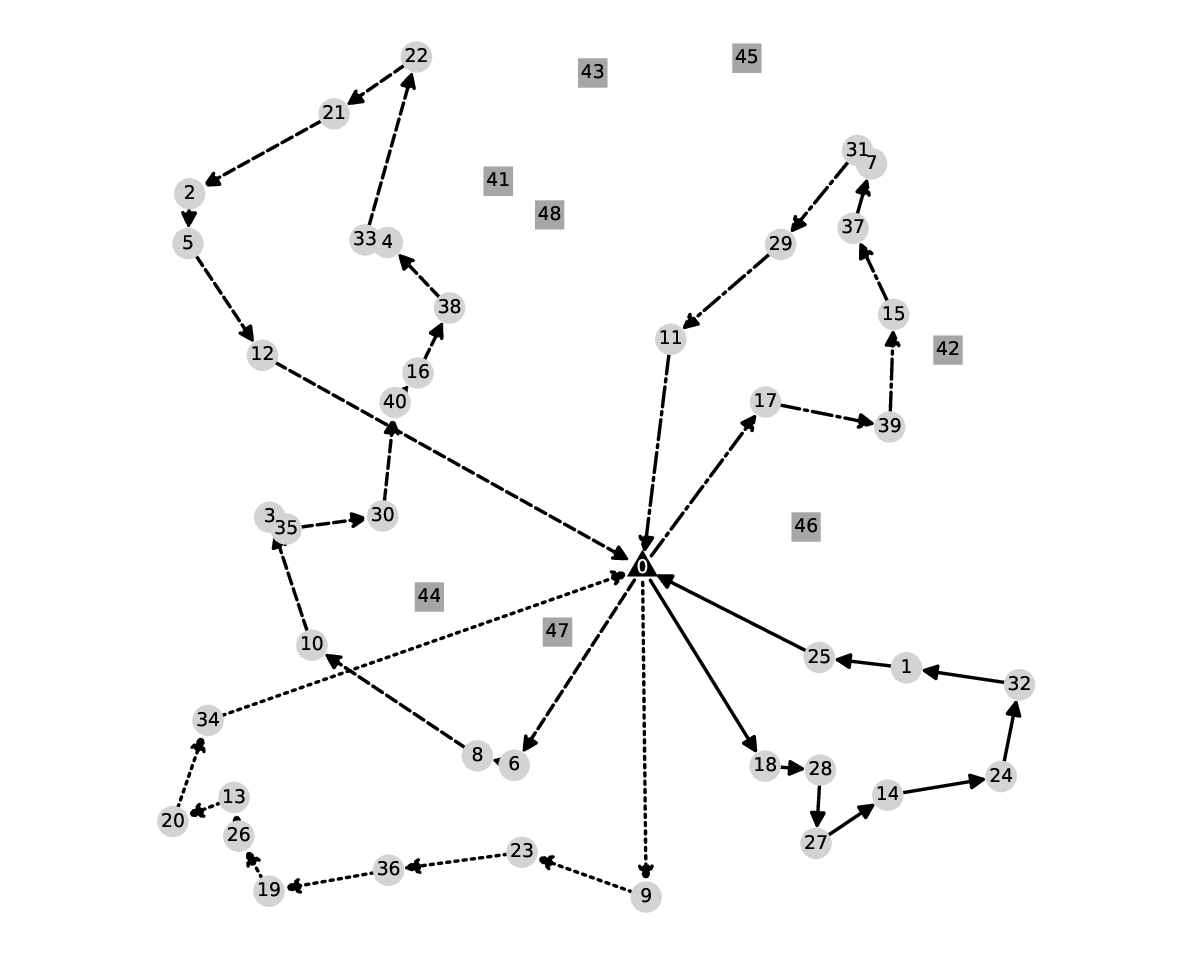}}
\hspace*{-1.43cm}
\subfloat[{$Q^T=40\%Q^{max}$.}]
{\includegraphics[width=.6\textwidth]{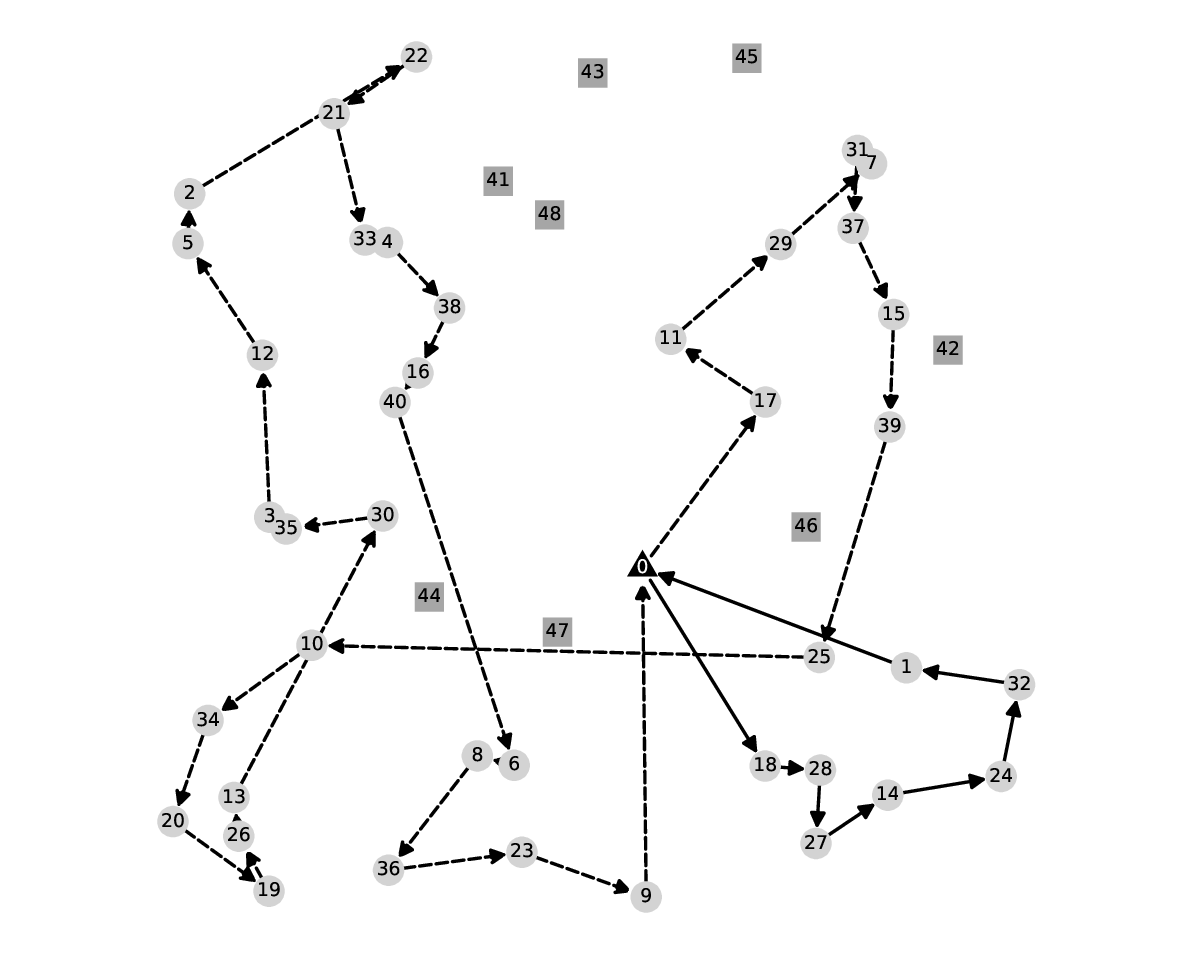}}
\caption{First-stage routes for instance tc0c40s8cf0 with $Q^T=20\%Q^{max}$ (left panel) and $Q^T=40\%Q^{max}$ (right panel). Depot, customers and charging stations are depicted as a triangle, circles and squares, respectively.}
\label{fig_managerialinsight}
\end{figure}

\section{Conclusions} \label{sec:conclusions}
In this paper, we considered the stochastic electric vehicle routing problem with a threshold recourse policy under uncertain energy consumption. We modeled the problem as a two-stage stochastic mixed-integer second-order cone model. In the first stage the sequence of customers visits is determined, while in the second stage a threshold recourse charging policy is implemented. According to this policy, if the state of charge of the battery reaches a predefined threshold level, the EV detours towards an appropriate CS. After a partial recharging operation, the EV completes  its first-stage route. As such, the threshold recourse policy allows detours towards CSs along arcs, rather than being confined to nodes.

Given the computational complexity of the model, we proposed a heuristic algorithm based on an Iterated Local Search procedure combined with a Set Partitioning phase. Specifically, we built a pool of high-quality routes through the ILS. To include charging operations in a given first-stage route, we proposed an exact algorithm to solve the stochastic fixed route vehicle charging problem with a threshold policy. Furthermore, to filter unpromising moves during the search phase, we derived two lower bounds on the duration of fixed routes. The most promising combination of routes is then selected by the Set Partitioning formulation. To handle cases involving large-sized instances and several scenarios, we applied a scenario reduction technique to the problem.

We conducted extensive computational experiments on instances adapted from the EVRP literature. We assessed the effectiveness of the heuristic in the deterministic and stochastic settings, showing good performances both in terms of solution quality and computational time. We analyzed the impact of different probability distributions. Moreover, we conducted an in-sample stability analysis to determine the number of scenarios yielding convergence. In addition, we evaluated the impact of uncertainty through the computation of stochastic measures, showing the importance of considering energy consumption uncertainty in this setting.

Regarding future developments, different streams of research can originate from this work. First of all, the generation of the scenario trees could be performed on the basis of real data, taking into account spatial and temporal correlations among travel speeds and energy consumptions. To this end, the increasingly available spatio-temporal data coming from different sources could be used to identify a suitable scenarios representation by applying Deep Learning techniques. Secondly, alternative recourse policies could be examined in order to relax some of the SEVRP-T assumptions. Thirdly, exact decomposition strategies, such as L-shaped methods, could be employed. Finally, a dynamic version of the stochastic problem monitoring energy consumption at an instantaneous level merits future investigation.

\section*{Acknowledgements}
This work has been supported by ``ULTRA OPTYMAL - Urban Logistics and sustainable TRAnsportation: OPtimization under uncertainTY and MAchine Learning'', a PRIN2020 project funded by the Italian University and Research Ministry (grant number 20207C8T9M).

This study was also carried out within the MOST - Sustainable Mobility National Research Center and received funding from the European Union Next-GenerationEU (PIANO NAZIONALE DI RIPRESA E RESILIENZA (PNRR) – MISSIONE 4 COMPONENTE 2, INVESTIMENTO 1.4 – D.D. 1033 17/06/2022, CN00000023), Spoke 5 ``Light Vehicle and Active Mobility'' and Spoke 10 “Logistics and Freight”. This manuscript reflects only the authors' views and opinions, neither the European Union nor the European Commission can be considered responsible for them.

Andrea Spinelli acknowledges the organizers and the experts of the EURO Summer Institute ESI2024 ``Decision-making under uncertainty for commodities and financial markets'', held in Ischia (Italy) from September 15 to 25, 2024. Finally, the authors  acknowledge the support received from Gruppo Nazionale per il Calcolo Scientifico (GNCS-INdAM).




\bibliography{Bibliography}

\newpage

\appendix

\section{Algorithm for the selection of non-dominated charging stations} \label{sec_appendix_nondominatedCS_algo}

\begin{algorithm}[H]
    \caption{NDCS($i,j)$}
    \textbf{Input:} Nodes $i,j\in\mathcal{I}^+$\newline
	\textbf{Output:} A set $\mathcal{C}_{ij}$ of non-dominated CSs
  \begin{algorithmic}[1]
	\Function{NDCS}{}
  \normalsize
  \State{$\mathcal{C}_{ij} \gets \mathcal{K}$} \Comment{\footnotesize All CSs are inserted} \normalsize
  \State{$\mathcal{\widetilde{K}} \gets \mathcal{K}$} \Comment{\footnotesize Auxiliary set to span over all CSs}\normalsize
  		\For{every CS $k_1 \in \mathcal{K}$}
            \State{$\mathcal{\widetilde{K}} \gets \mathcal{\widetilde{K}}\setminus \{k_1\}$}
            \State $d^{(k_1)}_{ij} \gets$ minimum distance between CS $k_1$ and arc $(i,j)$
            \If{$j=0$}
                \State{$\overline{q}_{k_1} \gets e_{k_1j}$} \Comment{\footnotesize SoC at CS $k_1$ after charging, before traversing the last arc} \normalsize
            \Else
                \State{$\overline{q}_{k_1} \gets Q^G+e_{k_1j}$}
            \EndIf
            \If{$Q^T<\frac{e_{ij}}{d_{ij}}d^{(k_1)}_{ij}$ or $\overline{q}_{k_1} > Q^{max}$}
                \State{$\mathcal{C}_{ij} \gets \mathcal{C}_{ij} \setminus \{k_1$\}} \Comment{\footnotesize Remove CS $k_1$ from the set $\mathcal{C}_{ij}$} \normalsize
            \Else
                \For{every CS $k_2 \in \mathcal{\widetilde{K}}$}
                    \State $d^{(k_2)}_{ij} \gets$ minimum distance between CS $k_2$ and arc $(i,j)$
                    \If{$k_1$ and $k_2$ have the same technology}
                        \If{$d_{k_1j}>d_{k_2j}$}
                            \If{$d^{(k_1)}_{ij}>\max\{d_{ik_2},d_{k_2j}\}$}
                                \State{$\mathcal{C}_{ij} \gets \mathcal{C}_{ij} \setminus \{k_1$\}}
                            \EndIf
                        \Else
                            \If{$d^{(k_2)}_{ij}>\max\{d_{ik_1},d_{k_1j}\}$}
                                \State{$\mathcal{C}_{ij} \gets \mathcal{C}_{ij} \setminus \{k_2$\}}
                            \EndIf
                        \EndIf
                    \ElsIf{$k_1$ has a slower technology than $k_2$}
                        \If{$d^{(k_1)}_{ij}+d_{k_1j}>\max\{d_{ik_2},d_{k_2j}\}+d_{k_2j}$}
                            \State{$\mathcal{C}_{ij} \gets \mathcal{C}_{ij} \setminus \{k_1$\}}
                        \EndIf
                    \ElsIf{$k_1$ has a faster technology than $k_2$}
                        \If{$d^{(k_2)}_{ij}+d_{k_2j}>\max\{d_{ik_1},d_{k_1j}\}+d_{k_1j}$}
                            \State{$\mathcal{C}_{ij} \gets \mathcal{C}_{ij} \setminus \{k_2$\}}
                        \EndIf
                    \EndIf
                \EndFor
            \EndIf
        \EndFor
        
\State \Return $\mathcal{C}_{ij}$

	\EndFunction
    \end{algorithmic}
    \label{alg:NDCS}
\end{algorithm}

\clearpage
\section{Fast forward selection algorithm for scenario reduction} \label{sec_appendix_FFS_algo}

\begin{algorithm}[H]
    \caption{FFS (see \cite{HeiRom2003} and \cite{Gioia2023})}
    \textbf{Input:} A set $\mathcal{S}$ of scenarios $e_s$ with probabilities $p_s$, the cardinality $|\mathcal{S}^{red}|$ of the set of reduced scenarios\newline
	\textbf{Output:} The set $\mathcal{S}^{red}$ of reduced scenarios with probabilities $p^{red}$
  \begin{algorithmic}[1]
	\Function{FFS}{}
  \normalsize
  		\State $z^{[1]} \gets \infty, \overline{s}_1 \gets 0$
		\For{every couple of scenarios $(o,\tilde{s}) \in \mathcal{S}$}
		\State $c_{o\tilde{s}}^{[1]} \gets \norm{e_o-e_{\tilde{s}}}_2$ \Comment{\footnotesize Distance between scenarios} \normalsize
		\EndFor
		\For{every scenario $\tilde{s}\in\mathcal{S}$} \Comment{\footnotesize See problem \eqref{set_covering_forward_selection}} \normalsize
		\State $z_{\tilde{s}} \gets \sum_{o \in \mathcal{S}\setminus\{\tilde{s}\}}p_o c_{o\tilde{s}}^{[1]}$
		\If{$z_{\tilde{s}}<z^{[1]}$}
		\State $z^{[1]} \gets z_{\tilde{s}}$, $\overline{s}_{1} \gets \tilde{s}$
		\EndIf
		\EndFor
		\State $\mathcal{S}_{[1]}^{del} \gets \mathcal{S}\setminus \{\overline{s}_1\}$ 
	
  \normalsize
  		\State $m\gets 2$
		\While{$m\leq |\mathcal{S}^{red}|$}
		  		\State $z^{[m]} \gets \infty, \overline{s}_m \gets 0$
  
		\For{every couple of scenarios $(o,\tilde{s}) \in \mathcal{S}^{del}_{[m-1]}$}
		\State $c_{o,\tilde{s}}^{[m]} \gets \min\{c_{o\tilde{s}}^{[m-1]},c_{o\overline{s}_{m-1}}^{[m-1]}\}$ \Comment{\footnotesize Update distance between scenarios} \footnotesize 
		\EndFor
		\For{every scenario $\tilde{s}\in\mathcal{S}^{del}_{[m-1]}$} \Comment{\footnotesize See \eqref{formula_forward_selection}} \normalsize
		\State $z_{\tilde{s}} \gets \sum_{o \in \mathcal{S}^{del}_{[m-1]}\setminus\{\tilde{s}\}}p_o c_{o\tilde{s}}^{[m]}$
		\If{$z_{\tilde{s}}<z^{[m]}$}
		\State $z^{[m]} \gets z_{\tilde{s}}$, $\overline{s}_{m} \gets \tilde{s}$
		\EndIf
		\EndFor
		\State $\mathcal{S}_{[m]}^{del} \gets \mathcal{S}_{[m-1]}^{del}\setminus \{\overline{s}_m\}$ 
				\EndWhile  
		\State $\mathcal{S}^{del} \gets \mathcal{S}^{del}_{[|\mathcal{S}^{red}|]}$, $\mathcal{S}^{red} \gets \mathcal{S}\setminus \mathcal{S}^{del}$ \Comment{\footnotesize Sets of deleted and reduced scenarios} \normalsize
		
  \normalsize
		\For{every scenario $s \in \mathcal{S}^{red}$} \Comment{\footnotesize See \eqref{optimal_redistribution_rule}} \normalsize
		\State $c^*_{s} \gets \infty$, $\mathcal{S}^{del}_s \gets \emptyset$
		\For{every scenario $\tilde{s} \in \mathcal{S}^{del}$}
		\State $c_{s\tilde{s}} \gets \norm{e_s-e_{\tilde{s}}}_2$
		\If{$c_{s\tilde{s}}< c^*_{s}$}
		\State $c^*_s \gets c_{s\tilde{s}}$, $\mathcal{S}^{del}_s \gets \tilde{s}$
		\ElsIf{$c_{s\tilde{s}} = c^*_{s}$}
		\State $\mathcal{S}^{del}_s \gets \mathcal{S}^{del}_s \cup \{\tilde{s}\}$
		\EndIf
		\EndFor
		\State $p_s^{red} \gets p_s + \sum_{\tilde{s} \in \mathcal{S}^{del}_s}p_{\tilde{s}}$ \Comment{\footnotesize Optimal redistribution rule} \normalsize
		\EndFor
		
		\State \Return $\mathcal{S}^{red}$, $p^{red}$
		
	\EndFunction
    \end{algorithmic}
    \label{alg:FFS}
\end{algorithm}

In the initial step of the procedure (lines 2-9), problem \eqref{set_covering_forward_selection} is solved, identifying scenario $\overline{s}_1$ and set $\mathcal{S}^{del}_{[1]}$. Next, the recursive phase (lines 10-19) selects the remaining scenarios $\overline{s}_2, \ldots,\overline{s}_{|\mathcal{S}^{red}|}$ to retain according to \eqref{formula_forward_selection}. Upon completion, sets $\mathcal{S}^{del}$ and $\mathcal{S}^{red}$ are determined (line 20). Finally, the optimal redistribution rule \eqref{optimal_redistribution_rule} is applied to compute the probabilities of the reduced scenarios (lines 21-29).

\clearpage
\section{Detailed results for instances with 10, 15, 20 customers - \\ deterministic setting} \label{sec_appendixA_deterministic}

In the following, we report the results of the numerical experiments for the instances with 10, 15 and 20 customers, in the deterministic setting ($|\mathcal{S}| = 1$). Specifically, we provide a comparison between ILS-SP heuristic (Algorithm~\ref{alg:ILS-P}) and Gurobi solver, in terms of CPU time, number of routes and optimality of the solution. The acronym ``LB'' stands for Lower Bound provided by Gurobi and the asterisk indicates that the time limit of 10800 seconds has been reached.

\begin{table}[ht!]
  \centering
  \caption{10 customers - deterministic setting.}
    \resizebox{\textwidth}{!}{
  \begin{tabular}{|c|c|c|c|c|c|c|c|c|c|}
    \hline
    \multirow{3}{*}{Instance} & \multicolumn{2}{|c|}{Time (s)} & \multicolumn{2}{|c|}{Number of routes}  & \multicolumn{3}{|c|}{Objective} & \multicolumn{2}{|c|}{Gap} \\
    \cline{2-10}
    & \multirow{2}{*}{Algorithm}& \multirow{2}{*}{Model} & \multirow{2}{*}{Algorithm} & \multirow{2}{*}{Model} & \multirow{2}{*}{Algorithm} & \multicolumn{2}{|c|}{Model} & \multirow{2}{*}{Gurobi} & \multirow{2}{*}{Algorithm-Model}\\
    \cline{7-8}
    &  &  &  &  &  & Incumbent & LB &  & \\
    \hline
      tc0c10s2cf1 & 11.05 &	0.76 & 2 & 2 & 9.97 & 9.97 & 9.97 & 0\% & 0\% \\
      tc0c10s2ct1 & 10.78 & 9.51 & 2 & 2 & 9.84 & 9.84 & 9.84 & 0\%	& 0\% \\
      tc0c10s3cf1 & 11.71 & 2.54 & 2 & 2 & 9.97 & 9.97 & 9.97 & 0\% & 0\% \\
      tc0c10s3ct1 & 11.99 & 16.75 & 2 & 2 & 9.84 & 9.84 & 9.84 & 0\% & 0\% \\
      tc1c10s2cf2 & 11.04 & 10.76 & 1 & 1 & 7.08 & 7.08 & 7.08 & 0\% & 0\% \\
      tc1c10s2cf4 & 12.36 & 6.70 & 2 & 2 & 11.85 & 11.85 & 11.85 & 0\% & 0\% \\
      tc1c10s2ct2 & 10.18 & 12.72 & 1 & 1 & 7.25 & 7.25 & 7.25 & 0\% & 0\% \\
      tc1c10s2ct3 & 11.66 & 354.42 & 1 & 1 & 11.04 & 11.04 & 11.04 & 0\% & 0\% \\
      tc1c10s2ct4 & 11.80 & 7.30 & 1 & 1 & 11.36 & 11.36 & 11.36 & 0\% & 0\%\\
      tc1c10s3cf2 & 11.27 & 4.02 & 1 & 1 & 7.08 & 7.08 & 7.08 & 0\% & 0\% \\
      tc1c10s3cf4 & 13.26 & 11.96 & 2 & 2 & 11.85 & 11.85 & 11.85 & 0\% & 0\%\\
      tc1c10s3ct2 & 11.01 & 19.04 & 1 & 1 & 7.25 & 7.25 & 7.25 & 0\% & 0\%\\
      tc1c10s3ct3 & 13.68 & 634.33 & 1 & 1 & 9.83 & 9.83 & 9.83 & 0\% & 0\%\\
      tc1c10s3ct4 & 12.47 & 15.51 & 1 & 1 & 11.36 & 11.36 & 11.36 & 0\% & 0\%\\
      tc2c10s2cf0 & 8.67 & 2.41 & 2	& 2 & 8.21 & 8.21 & 8.21 & 0\% & 0\%\\
      tc2c10s2ct0 & 11.74 & 41.03 & 2 & 2 & 7.45 & 7.45 & 7.45 & 0\% & 0\% \\
      tc2c10s3cf0 & 8.85 & 2.71 & 2	& 2 & 8.21 & 8.21 & 8.21 & 0\% & 0\%\\
      tc2c10s3ct0 & 11.71 & 136.18 & 2 & 2 & 7.53 & 7.53 & 7.53 & 0\% & 0\% \\ 
    \hline
  \end{tabular}}
  \label{tab_sevrp-deterministic_10}
\end{table}

\begin{table}[ht!]
  \centering
    \caption{15 customers - deterministic setting.}
    \resizebox{\textwidth}{!}{
  \begin{tabular}{|c|c|c|c|c|c|c|c|c|c|}
    \hline
    \multirow{3}{*}{Instance} & \multicolumn{2}{|c|}{Time (s)} & \multicolumn{2}{|c|}{Number of routes}  & \multicolumn{3}{|c|}{Objective} & \multicolumn{2}{|c|}{Gap} \\
    \cline{2-10}
    & \multirow{2}{*}{Algorithm}& \multirow{2}{*}{Model} & \multirow{2}{*}{Algorithm} & \multirow{2}{*}{Model} & \multirow{2}{*}{Algorithm} & \multicolumn{2}{|c|}{Model} & \multirow{2}{*}{Gurobi} & \multirow{2}{*}{Algorithm-Model}\\
    \cline{7-8}
    &  &  &  &  &  & Incumbent & LB &  & \\
    \hline
      tc0c15s3cf2 & 45.18 & 268.10 & 2 & 2 & 12.32 & 12.32 & 12.32 & 0\% & 0\% \\
      tc0c15s3ct2 & 35.53 & 142.13 & 2 & 2 & 11.21 & 11.21 & 11.21 & 0\% & 0\% \\
      tc0c15s4cf2 & 48.01 & 2817.04 & 2 & 2 & 12.32 & 12.32 & 12.32 & 0\% & 0\% \\
      tc0c15s4ct2 & 40.79 & 387.32 & 2 & 2 & 11.21 & 11.21 & 11.21 & 0\% & 0\% \\
      tc1c15s3cf1 & 25.88 & 4064.50 & 1 & 1 & 7.45 & 7.45 & 7.45 & 0\% & 0\% \\
      tc1c15s3cf3 & 42.24 & 3938.97 & 2	& 2 & 10.71 & 10.71 & 10.71 & 0\% & 0\% \\
      tc1c15s3cf4 & 60.42 & 268.63 & 2 & 2 & 11.80 & 11.80 & 11.80 & 0\% & 0\% \\
      tc1c15s3ct1 & 24.30 & 4404.25 & 1	& 1 & 7.75 & 7.75 & 7.75 & 0\% & 0\% \\
      tc1c15s3ct3 & 33.11 & 883.37 & 1 & 1 & 9.62 & 9.62 & 9.62 & 0\% & 0\% \\
      tc1c15s3ct4 & 50.50 & 1680.79 & 1	& 1 & 11.50 & 11.50 & 11.50 & 0\% & 0\% \\
      tc1c15s4cf1 & 24.69 & 6332.66 & 1	& 1 & 7.45 & 7.45 &	7.45 & 0\% & 0\% \\
      tc1c15s4cf3 & 40.10 & 4477.57 & 2	& 2 & 10.71 & 10.71 & 10.71 & 0\% & 0\% \\
      tc1c15s4cf4 & 74.77 & 431.53 & 2 & 2 & 11.80 & 11.80 & 11.80 & 0\% & 0\%\\
      tc1c15s4ct1 & 23.92 & 6833.78 & 1 & 1 & 7.75 & 7.75 & 7.75 & 0\% & 0\% \\
      tc1c15s4ct3 & 38.16 & 10800* & 1	& 1& 10.60 & 10.60 & 8.92 & 15.80\% & 0\% \\
      tc1c15s4ct4 & 34.06 & 128.83 & 1 & 1 & 11.34 & 11.34 & 11.34 & 0\% & 0\% \\
      tc2c15s3cf0 & 37.00 & 10800* & 1 & 1 & 9.30 &	9.30 & 6.55 & 29.52\% & 0\% \\
      tc2c15s3ct0 & 32.37 & 10800* & 1 & 1 & 9.60 & 9.60 & 6.42 & 33.13\% & 0\% \\
      tc2c15s4cf0 & 37.57 & 10800* & 1 & 1 & 9.30 & 9.30 & 6.55 & 29.57\% & 0\% \\
      tc2c15s4ct0 & 32.65 & 10800* & 1 & 1 & 9.26 & 9.26 & 6.39 & 30.96\% & 0\%\\
      \hline
  \end{tabular}}
  \label{tab_sevrp-deterministic_15}
\end{table}

\begin{table}[ht!]
  \centering
    \caption{20 customers - deterministic setting.}
    \resizebox{\textwidth}{!}{
  \begin{tabular}{|c|c|c|c|c|c|c|c|c|c|}
    \hline
    \multirow{3}{*}{Instance} & \multicolumn{2}{|c|}{Time (s)} & \multicolumn{2}{|c|}{Number of routes}  & \multicolumn{3}{|c|}{Objective} & \multicolumn{2}{|c|}{Gap} \\
    \cline{2-10}
    & \multirow{2}{*}{Algorithm}& \multirow{2}{*}{Model} & \multirow{2}{*}{Algorithm} & \multirow{2}{*}{Model} & \multirow{2}{*}{Algorithm} & \multicolumn{2}{|c|}{Model} & \multirow{2}{*}{Gurobi} & \multirow{2}{*}{Algorithm-Model}\\
    \cline{7-8}
    &  &  &  &  &  & Incumbent & LB &  & \\
    \hline
 tc0c20s3cf2 & 98.32 & 10800* & 2 & 2 & 13.24 & 13.24 & 12.41 & 6.28\% & 0\% \\
    tc0c20s3ct2 & 71.52	& 8994.82 & 2 & 2 & 12.27 & 12.27 & 12.27 & 0\% & 0\% \\
    tc0c20s4cf2 & 104.51 & 10800* & 2 & 2 & 13.10 & 13.10 & 11.76 & 10.19\% & 0\%\\
    tc0c20s4ct2 & 75.69 & 10790.63 & 2 & 2 & 12.20 & 12.20 & 12.20 & 0\% & 0\% \\
    tc1c20s3cf1 & 117.35 & 10800* & 2 & 2 & 12.99 & 12.99 & 11.82 & 9.03\% & 0\% \\
    tc1c20s3cf3 & 68.04 & 10800* & 1 & 1 & 11.28 & 11.28 & 9.55 & 15.27\% & 0\% \\
    tc1c20s3cf4 & 133.06 & 10800* & 3 & 3 & 14.22 & 14.22 & 13.65 & 4.02\% & 0\% \\
    tc1c20s3ct1 & 120.25 & 10800* & 2 & 2 & 13.93 & 14.19 & 11.48 & 19.13\% & $-$1.86\%	\\
    tc1c20s3ct3 & 73.23 & 10800* & 1 & 1 & 10.11 & 10.11 & 9.21 & 8.92\% & 0\% \\
    tc1c20s3ct4 & 94.37	& 10800* & 2 & 2 & 14.12 & 14.12 & 13.63 & 3.44\% & 0\% \\
    tc1c20s4cf1 & 81.71 & 10800* & 1 & 1 & 12.15 & 12.29 & 10.18 & 17.12\% & $-$1.14\% \\	
    tc1c20s4cf3 & 69.05 & 10800* & 1 & 1 & 11.28 & 11.28 & 9.62 & 14.73\% & 0\% \\
    tc1c20s4cf4 & 151.06 & 10800* & 3 & 3 & 14.22 & 14.22 & 13.59 & 4.41\% & 0\% \\
    tc1c20s4ct1 & 86.74	& 10800* & 2 & 2 & 12.56 & 12.91 & 10.25 & 20.59\% & $-$2.70\% \\
    tc1c20s4ct3 & 63.89 & 10800* & 1 & 1 & 10.71 & 10.71 & 9.53 & 11.07\% & 0\% \\
    tc1c20s4ct4 & 101.93 & 10800* & 2 & 2 & 14.05 & 14.05 & 13.67 & 2.66\% & 0\% \\
    tc2c20s3cf0 & 82.92 & 10800* & 1 & 1 & 10.68 & 10.68 & 8.56 & 19.86\% & 0\% \\
    tc2c20s3ct0 & 71.89 & 10800* & 1 & 1 & 10.86 & 10.86 & 7.86 & 27.58\% & 0\% \\
    tc2c20s4cf0 & 88.40 & 10800* & 1 & 1 & 10.68 & 10.68 & 8.54 & 20.02\% & 0\% \\
    tc2c20s4ct0 & 75.89 & 10800* & 1 & 1 & 10.86 & 10.86 & 7.76 & 28.55\% & 0\% \\
      \hline
  \end{tabular}}
  \label{tab_sevrp-deterministic_20}
\end{table}

\clearpage
\section{Detailed results for instances with 10 customers -\\ stochastic setting} \label{sec_appendixB_stochastic}

In the following, we report the results of the numerical experiments for the instances with 10 customers in the stochastic setting. Specifically, we start by providing a comparison between ILS-SP heuristic (Algorithm~\ref{alg:ILS-P}) and Gurobi solver, in terms of CPU time, number of routes and optimality of the solution. We consider an increasing number of scenarios ($|\mathcal{S}| = \{5, 10, 20, 50\}$). The acronym ``LB'' stands for Lower Bound provided by Gurobi and the asterisk indicates that the time limit of 10800s has been reached.

Different probability distributions on the energy consumption are explored: uniform (see \ref{sec_appB_unif}), normal (see \ref{sec_appB_normal}), exponential (see \ref{sec_appB_expon}). Then, we report the detailed results of an in-sample stability analysis with respect to the size of the scenario tree (see \ref{sec_appB_insample}). Finally, we show the values of classic stochastic measures obtained in the case $|\mathcal{S}| = 20$ (see \ref{sec_appB_stochmeasure}).

\subsection{Uniform distribution} \label{sec_appB_unif}

\begin{table}[ht!]
  \centering
  \caption{10 customers - $|\mathcal{S}| = 5$ - uniform distribution.}
    \resizebox{\textwidth}{!}{
  \begin{tabular}{|c|c|c|c|c|c|c|c|c|c|}
    \hline
    \multirow{3}{*}{Instance} & \multicolumn{2}{|c|}{Time (s)} & \multicolumn{2}{|c|}{Number of routes}  & \multicolumn{3}{|c|}{Objective} & \multicolumn{2}{|c|}{Gap} \\
    \cline{2-10}
    & \multirow{2}{*}{Algorithm}& \multirow{2}{*}{Model} & \multirow{2}{*}{Algorithm} & \multirow{2}{*}{Model} & \multirow{2}{*}{Algorithm} & \multicolumn{2}{|c|}{Model} & \multirow{2}{*}{Gurobi} & \multirow{2}{*}{Algorithm-Model}\\
    \cline{7-8}
    &  &  &  &  &  & Incumbent & LB &  & \\
    \hline
tc0c10s2cf1	&	19.34	&	16.46	&	2	&	2	&	10.33	&	10.33	&	10.33	&	0\%	&	0\%	\\
tc0c10s2ct1	&	20.00	&	86.67	&	2	&	2	&	10.04	&	10.04	&	10.04	&	0\%	&	0\%	\\
tc0c10s3cf1	&	20.37	&	41.37	&	2	&	2	&	11.51	&	11.51	&	11.51	&	0\%	&	0\%	\\
tc0c10s3ct1	&	28.78	&	263.43	&	2	&	2	&	10.02	&	10.02	&	10.02	&	0\%	&	0\%	\\
tc1c10s2cf2	&	28.80	&	115.84	&	1	&	1	&	7.13	&	7.13	&	7.13	&	0\%	&	0\%	\\
tc1c10s2cf4	&	24.65	&	281.74	&	2	&	2	&	11.90	&	11.90	&	11.90	&	0\%	&	0\%	\\
tc1c10s2ct2	&	34.30	&	626.87	&	1	&	1	&	7.36	&	7.36	&	7.36	&	0\%	&	0\%	\\
tc1c10s2ct3	&	24.49	&	3913.02	&	1	&	1	&	10.24	&	10.24	&	10.24	&	0\%	&	0\%	\\
tc1c10s2ct4	&	29.63	&	383.82	&	2	&	2	&	11.95	&	11.95	&	11.95	&	0\%	&	0\%	\\
tc1c10s3cf2	&	37.12	&	283.98	&	2	&	2	&	7.43	&	7.43	&	7.43	&	0\%	&	0\%	\\
tc1c10s3cf4	&	30.08	&	1621.69	&	2	&	2	&	12.12	&	12.12	&	12.12	&	0\%	&	0\%	\\
tc1c10s3ct2	&	45.34	&	2103.30	&	2	&	2	&	7.55	&	7.55	&	7.55	&	0\%	&	0\%	\\
tc1c10s3ct3	&	33.99	&	10800*	&	1	&	1	&	9.73	&	9.73	&	8.66	&	10.95\%	&	0\%	\\
tc1c10s3ct4	&	37.95	&	10800*	&	2	&	2	&	12.13	&	12.13	&	11.97	&	1.34\%	&	0\%	\\
tc2c10s2cf0	&	19.73	&	9.31	&	2	&	2	&	7.75	&	7.75	&	7.75	&	0\%	&	0\%	\\
tc2c10s2ct0	&	37.98	&	1062.10	&	2	&	2	&	7.48	&	7.48	&	7.48	&	0\%	&	0\%	\\
tc2c10s3cf0	&	18.99	&	14.66	&	2	&	2	&	8.25	&	8.25	&	8.25	&	0\%	&	0\%	\\
tc2c10s3ct0	&	39.34	&	2436.22	&	2	&	2	&	7.57	&	7.57	&	7.57	&	0\%	&	0\%	\\
    \hline
  \end{tabular}}
  \label{tab_sevrp-stoch_c10_S5_unif}
\end{table}

\begin{table}[ht!]
  \centering
  \caption{10 customers - $|\mathcal{S}| = 10$ - uniform distribution.}
    \resizebox{\textwidth}{!}{
  \begin{tabular}{|c|c|c|c|c|c|c|c|c|c|}
    \hline
    \multirow{3}{*}{Instance} & \multicolumn{2}{|c|}{Time (s)} & \multicolumn{2}{|c|}{Number of routes}  & \multicolumn{3}{|c|}{Objective} & \multicolumn{2}{|c|}{Gap} \\
    \cline{2-10}
    & \multirow{2}{*}{Algorithm}& \multirow{2}{*}{Model} & \multirow{2}{*}{Algorithm} & \multirow{2}{*}{Model} & \multirow{2}{*}{Algorithm} & \multicolumn{2}{|c|}{Model} & \multirow{2}{*}{Gurobi} & \multirow{2}{*}{Algorithm-Model}\\
    \cline{7-8}
    &  &  &  &  &  & Incumbent & LB &  & \\
    \hline
tc0c10s2cf1	&	24.01	&	170.48	&	2	&	2	&	11.97	&	11.97	&	11.97	&	0\%	&	0\%	\\
tc0c10s2ct1	&	30.30	&	5747.90	&	2	&	2	&	9.99	&	9.99	&	9.99	&	0\%	&	0\%	\\
tc0c10s3cf1	&	27.15	&	309.56	&	2	&	2	&	12.27	&	12.27	&	12.27	&	0\%	&	0\%	\\
tc0c10s3ct1	&	48.58	&	10800	&	2	&	2	&	10.00	&	10.00	&	9.99	&	0.09\%	&	0\%	\\
tc1c10s2cf2	&	62.05	&	1603.15	&	2	&	2	&	7.44	&	7.44	&	7.44	&	0\%	&	0\%	\\
tc1c10s2cf4	&	42.33	&	10800*	&	2	&	2	&	12.14	&	12.14	&	11.88	&	2.19\%	&	0\%	\\
tc1c10s2ct2	&	71.35	&	10560.38	&	2	&	2	&	7.59	&	7.59	&	7.59	&	0\%	&	0\%	\\
tc1c10s2ct3	&	44.90	&	10800*	&	1	&	1	&	10.39	&	10.39	&	9.89	&	4.83\%	&	0\%	\\
tc1c10s2ct4	&	50.76	&	10800*	&	2	&	2	&	12.00	&	12.00	&	11.84	&	1.34\%	&	0\%	\\
tc1c10s3cf2	&	70.74	&	1600.98	&	2	&	2	&	7.51	&	7.51	&	7.51	&	0\%	&	0\%	\\
tc1c10s3cf4	&	50.97	&	10800*	&	2	&	2	&	12.00	&	12.00	&	11.41	&	4.91\%	&	0\%	\\
tc1c10s3ct2	&	84.34	&	10800*	&	2	&	2	&	7.71	&	7.73	&	7.13	&	7.74\%	&	$-0.19\%$		\\
tc1c10s3ct3	&	75.80	&	10800*	&	1	&	1	&	9.67	&	9.84	&	6.88	&	30.16\%	&	$-1.82\%$		\\
tc1c10s3ct4	&	64.80	&	10800*	&	2	&	2	&	12.10	&	12.10	&	11.10	&	8.29\%	&	0\%	\\
tc2c10s2cf0	&	29.15	&	658.23	&	2	&	2	&	8.02	&	8.02	&	8.02	&	0\%	&	0\%	\\
tc2c10s2ct0	&	73.36	&	10800*	&	2	&	2	&	7.51	&	7.51	&	7.34	&	2.27\%	&	0\%	\\
tc2c10s3cf0	&	28.25	&	384.49	&	2	&	2	&	8.15	&	8.15	&	8.15	&	0\%	&	0\%	\\
tc2c10s3ct0	&	65.38	&	10800*	&	2	&	2	&	7.55	&	7.63	&	6.69	&	12.34\%	&	$-1.04\%$		\\
    \hline
  \end{tabular}}
  \label{tab_sevrp-stoch_c10_S10_unif}
\end{table}

\begin{table}[ht!]
  \centering
  \caption{10 customers - $|\mathcal{S}| = 20$ - uniform distribution.}
    \resizebox{\textwidth}{!}{
  \begin{tabular}{|c|c|c|c|c|c|c|c|c|c|}
    \hline
    \multirow{3}{*}{Instance} & \multicolumn{2}{|c|}{Time (s)} & \multicolumn{2}{|c|}{Number of routes}  & \multicolumn{3}{|c|}{Objective} & \multicolumn{2}{|c|}{Gap} \\
    \cline{2-10}
    & \multirow{2}{*}{Algorithm}& \multirow{2}{*}{Model} & \multirow{2}{*}{Algorithm} & \multirow{2}{*}{Model} & \multirow{2}{*}{Algorithm} & \multicolumn{2}{|c|}{Model} & \multirow{2}{*}{Gurobi} & \multirow{2}{*}{Algorithm-Model}\\
    \cline{7-8}
    &  &  &  &  &  & Incumbent & LB &  & \\
    \hline
tc0c10s2cf1	&	39.84	&	10800*	&	2	&	2	&	12.42	&	12.42	&	12.18	&	1.90\%	&	0\%	\\
tc0c10s2ct1	&	50.44	&	10800*	&	2	&	2	&	9.98	&	9.98	&	9.94	&	0.41\%	&	0\%	\\
tc0c10s3cf1	&	38.86	&	10800*	&	2	&	2	&	12.30	&	12.30	&	11.77	&	4.31\%	&	0\%	\\
tc0c10s3ct1	&	77.55	&	10800*	&	2	&	2	&	10.00	&	10.00	&	9.95	&	0.46\%	&	0\%	\\
tc1c10s2cf2	&	108.03	&	10800*	&	2	&	2	&	7.52	&	7.52	&	7.33	&	2.49\%	&	0\%	\\
tc1c10s2cf4	&	74.11	&	10800*	&	2	&	2	&	12.21	&	12.21	&	11.56	&	5.32\%	&	0\%	\\
tc1c10s2ct2	&	135.72	&	10800*	&	2	&	2	&	7.73	&	7.73	&	6.89	&	10.84\%	&	0\%	\\
tc1c10s2ct3	&	68.96	&	10800*	&	1	&	1	&	10.44	&	10.44	&	7.35	&	29.63\%	&	0\%	\\
tc1c10s2ct4	&	81.00	&	10800*	&	2	&	2	&	12.01	&	12.01	&	11.80	&	1.73\%	&	0\%	\\
tc1c10s3cf2	&	123.01	&	10800*	&	2	&	2	&	7.64	&	7.64	&	7.49	&	2.03\%	&	0\%	\\
tc1c10s3cf4	&	89.40	&	10800*	&	2	&	2	&	12.08	&	12.08	&	10.48	&	13.29\%	&	0\%	\\
tc1c10s3ct2	&	159.62	&	10800*	&	2	&	2	&	7.69	&	7.74	&	6.68	&	13.68\%	&	$-0.65\%$		\\
tc1c10s3ct3	&	115.46	&	10800*	&	1	&	1	&	9.91	&	10.67	&	6.18	&	42.05\%	&	$-7.09\%$		\\
tc1c10s3ct4	&	103.33	&	10800*	&	2	&	2	&	12.05	&	12.05	&	10.21	&	15.29\%	&	0\%	\\
tc2c10s2cf0	&	47.07	&	10800*	&	2	&	2	&	8.09	&	8.09	&	7.41	&	8.34\%	&	0\%	\\
tc2c10s2ct0	&	132.98	&	10800*	&	2	&	2	&	7.52	&	7.52	&	5.52	&	26.55\%	&	0\%	\\
tc2c10s3cf0	&	47.50	&	10800*	&	3	&	3	&	11.06	&	11.06	&	10.44	&	5.59\%	&	0\%	\\
tc2c10s3ct0	&	134.02	&	10800*	&	2	&	2	&	7.60	&	7.82	&	5.06	&	35.29\%	&	$-2.89\%$		\\
    \hline
  \end{tabular}}
  \label{tab_sevrp-stoch_c10_S20_unif}
\end{table}

\begin{table}[ht!]
  \centering
  \caption{10 customers - $|\mathcal{S}| = 50$ - uniform distribution.}
    \resizebox{\textwidth}{!}{
  \begin{tabular}{|c|c|c|c|c|c|c|c|c|c|}
    \hline
    \multirow{3}{*}{Instance} & \multicolumn{2}{|c|}{Time (s)} & \multicolumn{2}{|c|}{Number of routes}  & \multicolumn{3}{|c|}{Objective} & \multicolumn{2}{|c|}{Gap} \\
    \cline{2-10}
    & \multirow{2}{*}{Algorithm}& \multirow{2}{*}{Model} & \multirow{2}{*}{Algorithm} & \multirow{2}{*}{Model} & \multirow{2}{*}{Algorithm} & \multicolumn{2}{|c|}{Model} & \multirow{2}{*}{Gurobi} & \multirow{2}{*}{Algorithm-Model}\\
    \cline{7-8}
    &  &  &  &  &  & Incumbent & LB &  & \\
    \hline
tc0c10s2cf1	&	84.67	&	10800*	&	2	&	2	&	12.40	&	12.40	&	11.66	&	5.97\%	&	0\%	\\
tc0c10s2ct1	&	115.79	&	10800*	&	2	&	2	&	10.02	&	10.02	&	9.86	&	1.56\%	&	0\%	\\
tc0c10s3cf1	&	79.99	&	10800*	&	2	&	2	&	12.24	&	12.24	&	11.79	&	3.65\%	&	0\%	\\
tc0c10s3ct1	&	162.68	&	10800*	&	2	&	2	&	10.20	&	12.39	&	8.60	&	30.56\%	&	$-17.63\%$		\\
tc1c10s2cf2	&	274.73	&	10800*	&	2	&	2	&	7.84	&	7.60	&	7.00	&	7.87\%	&	$3.16\%$		\\
tc1c10s2cf4	&	164.11	&	10800*	&	2	&	2	&	12.22	&	12.22	&	10.94	&	10.51\%	&	0\%	\\
tc1c10s2ct2	&	350.54	&	10800*	&	2	&	2	&	7.75	&	9.56	&	6.04	&	36.82\%	&	$-18.83\%$		\\
tc1c10s2ct3	&	159.61	&	10800*	&	1	&	2	&	10.47	&	12.95	&	6.13	&	52.68\%	&	$-19.13\%$		\\
tc1c10s2ct4	&	183.76	&	10800*	&	2	&	2	&	12.04	&	12.04	&	10.79	&	10.42\%	&	0\%	\\
tc1c10s3cf2	&	267.07	&	10800*	&	2	&	2	&	7.63	&	7.63	&	6.99	&	8.34\%	&	$-0.02\%$	\\
tc1c10s3cf4	&	183.45	&	10800*	&	2	&	2	&	12.11	&	12.79	&	9.92	&	22.39\%	&	$-5.30\%$	\\
tc1c10s3ct2	&	389.91	&	10800*	&	2	&	2	&	7.77	&	9.01	&	5.45	&	39.58\%	&	$-12.53\%$		\\
tc1c10s3ct3	&	275.79	&	10800*	&	1	&	1	&	10.08	&	11.85	&	5.23	&	55.88\%	&	$-14.95\%$		\\
tc1c10s3ct4	&	215.45	&	10800*	&	1	&	2	&	12.05	&	12.86	&	9.26	&	28.01\%	&	$-5.90\%$		\\
tc2c10s2cf0	&	104.34	&	10800*	&	3	&	3	&	11.14	&	11.14	&	10.07	&	9.59\%	&	0\%	\\
tc2c10s2ct0	&	310.73	&	10800*	&	2	&	2	&	7.65	&	8.17	&	4.33	&	47.00\%	&	$-6.09\%$		\\
tc2c10s3cf0	&	111.19	&	10800*	&	3	&	3	&	11.07	&	11.07	&	10.02	&	9.45\%	&	0\%	\\
tc2c10s3ct0	&	317.84	&	10800*	&	2	&	2	&	7.68	&	7.78	&	3.51	&	54.82\%	&	$-1.28\%$		\\
    \hline
  \end{tabular}}
  \label{tab_sevrp-stoch_c10_S50_unif}
\end{table}

\clearpage
\subsection{Normal distribution} \label{sec_appB_normal}

\begin{table}[ht!]
  \centering
  \caption{10 customers - $|\mathcal{S}| = 5$ - normal distribution.}
    \resizebox{\textwidth}{!}{
  \begin{tabular}{|c|c|c|c|c|c|c|c|c|c|}
    \hline
    \multirow{3}{*}{Instance} & \multicolumn{2}{|c|}{Time (s)} & \multicolumn{2}{|c|}{Number of routes}  & \multicolumn{3}{|c|}{Objective} & \multicolumn{2}{|c|}{Gap} \\
    \cline{2-10}
    & \multirow{2}{*}{Algorithm}& \multirow{2}{*}{Model} & \multirow{2}{*}{Algorithm} & \multirow{2}{*}{Model} & \multirow{2}{*}{Algorithm} & \multicolumn{2}{|c|}{Model} & \multirow{2}{*}{Gurobi} & \multirow{2}{*}{Algorithm-Model}\\
    \cline{7-8}
    &  &  &  &  &  & Incumbent & LB &  & \\
    \hline
tc0c10s2cf1	&	16.46	&	19.54	&	2	&	2	&	11.68	&	11.68	&	11.68	&	0\%	&	0\%	\\
tc0c10s2ct1	&	21.11	&	72.47	&	2	&	2	&	9.84	&	9.84	&	9.84	&	0\%	&	0\%	\\
tc0c10s3cf1	&	18.28	&	36.73	&	2	&	2	&	11.51	&	11.51	&	11.51	&	0\%	&	0\%	\\
tc0c10s3ct1	&	33.38	&	344.20	&	2	&	2	&	9.94	&	9.94	&	9.94	&	0\%	&	0\%	\\
tc1c10s2cf2	&	35.56	&	243.62	&	1	&	1	&	7.14	&	7.14	&	7.14	&	0\%	&	0\%	\\
tc1c10s2cf4	&	25.24	&	377.06	&	2	&	2	&	12.02	&	12.02	&	12.02	&	0\%	&	0\%	\\
tc1c10s2ct2	&	38.41	&	861.75	&	1	&	1	&	7.28	&	7.28	&	7.28	&	0\%	&	0\%	\\
tc1c10s2ct3	&	25.91	&	3789.47	&	1	&	1	&	10.83	&	10.83	&	10.83	&	0\%	&	0\%	\\
tc1c10s2ct4	&	29.44	&	302.51	&	2	&	2	&	12.07	&	12.07	&	12.07	&	0\%	&	0\%	\\
tc1c10s3cf2	&	36.73	&	250.19	&	1	&	1	&	7.19	&	7.19	&	7.19	&	0\%	&	0\%	\\
tc1c10s3cf4	&	31.10	&	2289.60	&	3	&	3	&	12.22	&	12.22	&	12.22	&	0\%	&	0\%	\\
tc1c10s3ct2	&	35.18	&	1233.05	&	1	&	1	&	7.39	&	7.39	&	7.39	&	0\%	&	0\%	\\
tc1c10s3ct3	&	37.96	&	10800* &	1	&	1	&	9.86	&	9.86	&	9.05	&	8.19\%	& 0\%	\\
tc1c10s3ct4	&	34.60	&	1946.13	&	2	&	2	&	11.91	&	11.91	&	11.91	&	0\%	&	0\%	\\
tc2c10s2cf0	&	19.05	&	11.82	&	2	&	2	&	8.25	&	8.25	&	8.25	&	0\%	&	0\%	\\
tc2c10s2ct0	&	35.28	&	916.66	&	2	&	2	&	7.55	&	7.55	&	7.55	&	0\%	&	0\%	\\
tc2c10s3cf0	&	19.19	&	10.06	&	2	&	2	&	8.22	&	8.22	&	8.22	&	0\%	&	0\%	\\
tc2c10s3ct0	&	40.45	&	2473.44	&	2	&	2	&	7.58	&	7.58	&	7.58	&	0\%	&	0\%	\\
    \hline
  \end{tabular}}
  \label{tab_sevrp-stoch_c10_S5_normal}
\end{table}

\begin{table}[ht!]
  \centering
  \caption{10 customers - $|\mathcal{S}| = 10$ - normal distribution.}
    \resizebox{\textwidth}{!}{
  \begin{tabular}{|c|c|c|c|c|c|c|c|c|c|}
    \hline
    \multirow{3}{*}{Instance} & \multicolumn{2}{|c|}{Time (s)} & \multicolumn{2}{|c|}{Number of routes}  & \multicolumn{3}{|c|}{Objective} & \multicolumn{2}{|c|}{Gap} \\
    \cline{2-10}
    & \multirow{2}{*}{Algorithm}& \multirow{2}{*}{Model} & \multirow{2}{*}{Algorithm} & \multirow{2}{*}{Model} & \multirow{2}{*}{Algorithm} & \multicolumn{2}{|c|}{Model} & \multirow{2}{*}{Gurobi} & \multirow{2}{*}{Algorithm-Model}\\
    \cline{7-8}
    &  &  &  &  &  & Incumbent & LB &  & \\
    \hline
tc0c10s2cf1	&	25.90	&	84.80	&	2	&	2	&	11.82	&	11.82	&	11.82	&	0\%	&	0\%	\\
tc0c10s2ct1	&	31.95	&	573.66	&	2	&	2	&	9.84	&	9.84	&	9.84	&	0\%	&	0\%	\\
tc0c10s3cf1	&	27.95	&	2472.73	&	2	&	2	&	11.82	&	11.82	&	11.82	&	0\%	&	0\%	\\
tc0c10s3ct1	&	57.71	&	10800	&	2	&	2	&	10.08	&	10.08	&	10.07	&	0.11\%	&	0\%	\\
tc1c10s2cf2	&	60.75	&	708.13	&	1	&	1	&	7.37	&	7.37	&	7.37	&	0\%	&	0\%	\\
tc1c10s2cf4	&	43.56	&	10800*	&	2	&	2	&	12.34	&	12.34	&	12.15	&	1.60\%	&	0\%	\\
tc1c10s2ct2	&	80.66	&	10800*	&	2	&	2	&	7.60	&	7.60	&	7.43	&	2.28\%	&	0\%	\\
tc1c10s2ct3	&	44.34	&	10800*	&	1	&	1	&	10.57	&	10.57	&	9.89	&	6.43\%	&	0\%	\\
tc1c10s2ct4	&	49.14	&	10800*	&	2	&	2	&	12.02	&	12.02	&	11.93	&	0.76\%	&	0\%	\\
tc1c10s3cf2	&	58.36	&	1153.97	&	1	&	1	&	7.17	&	7.17	&	7.17	&	0\%	&	0\%	\\
tc1c10s3cf4	&	50.25	&	10800*	&	2	&	2	&	12.24	&	12.24	&	11.11	&	9.21\%	&	0\%	\\
tc1c10s3ct2	&	62.68	&	10800*	&	1	&	1	&	7.37	&	7.37	&	7.33	&	0.56\%	&	0\%	\\
tc1c10s3ct3	&	76.17	&	10800*	&	1	&	1	&	10.06	&	10.33	&	6.91	&	33.13\%	&	$-2.57\%$		\\
tc1c10s3ct4	&	59.70	&	10800*	&	2	&	2	&	11.99	&	11.99	&	10.88	&	9.22\%	&	0\%	\\
tc2c10s2cf0	&	26.79	&	622.84	&	2	&	2	&	8.59	&	8.59	&	8.59	&	0\%	&	0\%	\\
tc2c10s2ct0	&	65.50	&	10800*	&	2	&	2	&	7.55	&	7.55	&	7.41	&	1.80\%	&	0\%	\\
tc2c10s3cf0	&	28.20	&	398.56	&	2	&	2	&	8.24	&	8.24	&	8.24	&	0\%	&	0\%	\\
tc2c10s3ct0	&	72.29	&	10800*	&	2	&	2	&	7.58	&	7.71	&	6.44	&	16.45\%	&	$-1.68\%$		\\
    \hline
  \end{tabular}}
  \label{tab_sevrp-stoch_c10_S10_normal}
\end{table}

\begin{table}[ht!]
  \centering
  \caption{10 customers - $|\mathcal{S}| = 20$ - normal distribution.}
    \resizebox{\textwidth}{!}{
  \begin{tabular}{|c|c|c|c|c|c|c|c|c|c|}
    \hline
    \multirow{3}{*}{Instance} & \multicolumn{2}{|c|}{Time (s)} & \multicolumn{2}{|c|}{Number of routes}  & \multicolumn{3}{|c|}{Objective} & \multicolumn{2}{|c|}{Gap} \\
    \cline{2-10}
    & \multirow{2}{*}{Algorithm}& \multirow{2}{*}{Model} & \multirow{2}{*}{Algorithm} & \multirow{2}{*}{Model} & \multirow{2}{*}{Algorithm} & \multicolumn{2}{|c|}{Model} & \multirow{2}{*}{Gurobi} & \multirow{2}{*}{Algorithm-Model}\\
    \cline{7-8}
    &  &  &  &  &  & Incumbent & LB &  & \\
    \hline
tc0c10s2cf1	&	38.78	&	10800*	&	2	&	2	&	11.99	&	11.99	&	11.97	&	0.13\%	&	0\%	\\
tc0c10s2ct1	&	53.79	&	10800*	&	2	&	2	&	9.91	&	9.91	&	9.89	&	0.18\%	&	0\%	\\
tc0c10s3cf1	&	36.44	&	10800*	&	2	&	2	&	12.13	&	11.96	&	11.88	&	0.73\%	&	1.35\%		\\
tc0c10s3ct1	&	99.69	&	10800*	&	2	&	2	&	10.01	&	10.01	&	9.97	&	0.40\%	&	0\%	\\
tc1c10s2cf2	&	122.20	&	10800*	&	2	&	2	&	7.59	&	7.59	&	7.28	&	4.01\%	&	0\%	\\
tc1c10s2cf4	&	74.79	&	10800*	&	2	&	2	&	12.18	&	12.18	&	11.45	&	6.03\%	&	0\%	\\
tc1c10s2ct2	&	170.25	&	10800*	&	2	&	2	&	7.67	&	7.71	&	6.80	&	11.81\%	&	$-0.54\%$		\\
tc1c10s2ct3	&	78.43	&	10800*	&	1	&	1	&	10.41	&	10.44	&	7.34	&	29.66\%	&	$-0.28\%$		\\
tc1c10s2ct4	&	86.48	&	10800*	&	2	&	2	&	12.05	&	12.05	&	11.73	&	2.67\%	&	0\%	\\
tc1c10s3cf2	&	122.50	&	10800*	&	2	&	2	&	7.56	&	7.56	&	7.39	&	2.27\%	&	0\%	\\
tc1c10s3cf4	&	92.91	&	10800*	&	2	&	2	&	12.09	&	12.09	&	11.20	&	7.33\%	&	0\%	\\
tc1c10s3ct2	&	174.34	&	10800*	&	2	&	2	&	7.77	&	7.86	&	6.47	&	17.67\%	&	$-1.10\%$		\\
tc1c10s3ct3	&	125.17	&	10800*	&	1	&	1	&	10.05	&	10.64	&	6.36	&	40.23\%	&	$-5.54\%$		\\
tc1c10s3ct4	&	98.57	&	10800*	&	2	&	2	&	12.10	&	12.10	&	10.38	&	14.21\%	&	0\%	\\
tc2c10s2cf0	&	45.46	&	10800*	&	3	&	3	&	11.05	&	11.05	&	10.57	&	4.32\%	&	0\%	\\
tc2c10s2ct0	&	138.97	&	10800*	&	2	&	2	&	7.55	&	7.61	&	5.47	&	28.10\%	&	$-0.78\%$		\\
tc2c10s3cf0	&	47.94	&	10800*	&	3	&	3	&	11.14	&	11.14	&	10.41	&	6.52\%	&	0\%	\\
tc2c10s3ct0	&	130.82	&	10800*	&	2	&	2	&	7.62	&	8.03	&	5.29	&	34.16\%	&	$-5.15\%$		\\
    \hline
  \end{tabular}}
  \label{tab_sevrp-stoch_c10_S20_normal}
\end{table}

\begin{table}[ht!]
  \centering
  \caption{10 customers - $|\mathcal{S}| = 50$ - normal distribution.}
    \resizebox{\textwidth}{!}{
  \begin{tabular}{|c|c|c|c|c|c|c|c|c|c|}
    \hline
    \multirow{3}{*}{Instance} & \multicolumn{2}{|c|}{Time (s)} & \multicolumn{2}{|c|}{Number of routes}  & \multicolumn{3}{|c|}{Objective} & \multicolumn{2}{|c|}{Gap} \\
    \cline{2-10}
    & \multirow{2}{*}{Algorithm}& \multirow{2}{*}{Model} & \multirow{2}{*}{Algorithm} & \multirow{2}{*}{Model} & \multirow{2}{*}{Algorithm} & \multicolumn{2}{|c|}{Model} & \multirow{2}{*}{Gurobi} & \multirow{2}{*}{Algorithm-Model}\\
    \cline{7-8}
    &  &  &  &  &  & Incumbent & LB &  & \\
    \hline
tc0c10s2cf1	&	83.88	&	10800*	&	2	&	2	&	12.11	&	12.11	&	11.40	&	5.89\%	&	0\%	\\
tc0c10s2ct1	&	109.10	&	10800*	&	2	&	2	&	10.11	&	10.11	&	9.92	&	1.89\%	&	0\%	\\
tc0c10s3cf1	&	79.29	&	10800*	&	2	&	2	&	12.12	&	12.12	&	11.87	&	2.03\%	&	0\%	\\
tc0c10s3ct1	&	201.46	&	10800*	&	2	&	1	&	10.17	&	11.88	&	8.54	&	28.09\%	&	$-14.44\%$		\\
tc1c10s2cf2	&	290.62	&	10800*	&	2	&	2	&	7.61	&	7.62	&	7.04	&	7.63\%	&	$-0.09\%$	\\
tc1c10s2cf4	&	169.22	&	10800*	&	2	&	2	&	12.06	&	12.06	&	10.96	&	9.19\%	&	0\%	\\
tc1c10s2ct2	&	365.59	&	10800*	&	2	&	2	&	7.76	&	8.50	&	5.94	&	30.18\%	&	$-8.34\%$		\\
tc1c10s2ct3	&	153.56	&	10800*	&	1	&	1	&	11.26	&	12.81	&	6.09	&	52.48\%	&	$-12.08\%$		\\
tc1c10s2ct4	&	176.25	&	10800*	&	2	&	2	&	12.08	&	12.08	&	11.14	&	7.82\%	&	0\%	\\
tc1c10s3cf2	&	277.53	&	10800*	&	2	&	2	&	7.61	&	7.61	&	7.14	&	6.26\%	&	0\%	\\
tc1c10s3cf4	&	212.29	&	10800*	&	2	&	3	&	12.16	&	12.50	&	10.19	&	18.46\%	&	$-2.73\%$		\\
tc1c10s3ct2	&	347.34	&	10800*	&	2	&	2	&	7.77	&	8.63	&	5.52	&	36.09\%	&	$-9.03\%$		\\
tc1c10s3ct3	&	272.49	&	10800*	&	1	&	2	&	10.07	&	15.95	&	5.24	&	67.16\%	&	$-34.81\%$		\\
tc1c10s3ct4	&	232.08	&	10800*	&	2	&	1	&	12.09	&	12.81	&	9.26	&	27.67\%	&	$-5.60\%$		\\
tc2c10s2cf0	&	100.81	&	10800*	&	3	&	3	&	11.07	&	11.07	&	9.99	&	9.76\%	&	0\%	\\
tc2c10s2ct0	&	313.85	&	10800*	&	2	&	2	&	7.62	&	7.80	&	4.69	&	39.86\%	&	$-2.29\%$		\\
tc2c10s3cf0	&	112.86	&	10800*	&	3	&	3	&	11.23	&	11.23	&	9.99	&	11.05\%	&	0\%	\\
tc2c10s3ct0	&	315.33	&	10800*	&	2	&	1	&	7.59	&	8.03	&	3.73	&	53.52\%	&	$-5.38\%$		\\
    \hline
  \end{tabular}}
  \label{tab_sevrp-stoch_c10_S50_normal}
\end{table}

\clearpage
\subsection{Exponential distribution} \label{sec_appB_expon}

\begin{table}[ht!]
  \centering
  \caption{10 customers - $|\mathcal{S}| = 5$ - exponential distribution.}
    \resizebox{\textwidth}{!}{
  \begin{tabular}{|c|c|c|c|c|c|c|c|c|c|}
    \hline
    \multirow{3}{*}{Instance} & \multicolumn{2}{|c|}{Time (s)} & \multicolumn{2}{|c|}{Number of routes}  & \multicolumn{3}{|c|}{Objective} & \multicolumn{2}{|c|}{Gap} \\
    \cline{2-10}
    & \multirow{2}{*}{Algorithm}& \multirow{2}{*}{Model} & \multirow{2}{*}{Algorithm} & \multirow{2}{*}{Model} & \multirow{2}{*}{Algorithm} & \multicolumn{2}{|c|}{Model} & \multirow{2}{*}{Gurobi} & \multirow{2}{*}{Algorithm-Model}\\
    \cline{7-8}
    &  &  &  &  &  & Incumbent & LB &  & \\
    \hline
tc0c10s2cf1	&	16.93	&	18.53	&	2	&	2	&	11.85	&	11.85	&	11.85	&	0\%	&	0\%	\\
tc0c10s2ct1	&	20.68	&	87.77	&	2	&	2	&	9.96	&	9.96	&	9.96	&	0\%	&	0\%	\\
tc0c10s3cf1	&	18.00	&	22.86	&	2	&	2	&	11.67	&	11.67	&	11.67	&	0\%	&	0\%	\\
tc0c10s3ct1	&	39.64	&	326.96	&	2	&	2	&	9.94	&	9.94	&	9.94	&	0\%	&	0\%	\\
tc1c10s2cf2	&	29.53	&	144.33	&	1	&	1	&	7.12	&	7.12	&	7.12	&	0\%	&	0\%	\\
tc1c10s2cf4	&	26.84	&	436.43	&	2	&	2	&	12.15	&	12.15	&	12.15	&	0\%	&	0\%	\\
tc1c10s2ct2	&	37.81	&	816.91	&	1	&	1	&	7.34	&	7.34	&	7.34	&	0\%	&	0\%	\\
tc1c10s2ct3	&	28.53	&	5602.28	&	1	&	1	&	10.76	&	10.76	&	10.76	&	0\%	&	0\%	\\
tc1c10s2ct4	&	26.13	&	133.95	&	2	&	2	&	12.08	&	12.08	&	12.08	&	0\%	&	0\%	\\
tc1c10s3cf2	&	31.46	&	140.90	&	1	&	1	&	7.12	&	7.12	&	7.12	&	0\%	&	0\%	\\
tc1c10s3cf4	&	33.29	&	1277.58	&	2	&	2	&	12.16	&	12.16	&	12.16	&	0\%	&	0\%	\\
tc1c10s3ct2	&	32.40	&	1739.31	&	1	&	1	&	7.37	&	7.37	&	7.37	&	0\%	&	0\%	\\
tc1c10s3ct3	&	40.71	&	10800*	&	1	&	1	&	9.88	&	9.88	&	7.78	&	21.25\%	&	0\%	\\
tc1c10s3ct4	&	34.68	&	3114.16	&	1	&	1	&	12.02	&	12.02	&	12.02	&	0\%	&	0\%	\\
tc2c10s2cf0	&	18.52	&	9.84	&	2	&	2	&	7.79	&	7.79	&	7.79	&	0\%	&	0\%	\\
tc2c10s2ct0	&	34.08	&	1228.23	&	2	&	2	&	7.50	&	7.50	&	7.50	&	0\%	&	0\%	\\
tc2c10s3cf0	&	19.12	&	14.33	&	2	&	2	&	7.97	&	7.97	&	7.97	&	0\%	&	0\%	\\
tc2c10s3ct0	&	35.27	&	2420.12	&	2	&	2	&	7.46	&	7.46	&	7.46	&	0\%	&	0\%	\\
    \hline
  \end{tabular}}
  \label{tab_sevrp-stoch_c10_S5_exponential}
\end{table}

\begin{table}[ht!]
  \centering
  \caption{10 customers - $|\mathcal{S}| = 10$ - exponential distribution.}
    \resizebox{\textwidth}{!}{
  \begin{tabular}{|c|c|c|c|c|c|c|c|c|c|}
    \hline
    \multirow{3}{*}{Instance} & \multicolumn{2}{|c|}{Time (s)} & \multicolumn{2}{|c|}{Number of routes}  & \multicolumn{3}{|c|}{Objective} & \multicolumn{2}{|c|}{Gap} \\
    \cline{2-10}
    & \multirow{2}{*}{Algorithm}& \multirow{2}{*}{Model} & \multirow{2}{*}{Algorithm} & \multirow{2}{*}{Model} & \multirow{2}{*}{Algorithm} & \multicolumn{2}{|c|}{Model} & \multirow{2}{*}{Gurobi} & \multirow{2}{*}{Algorithm-Model}\\
    \cline{7-8}
    &  &  &  &  &  & Incumbent & LB &  & \\
    \hline
tc0c10s2cf1	&	24.20	&	182.87	&	2	&	2	&	11.84	&	11.84	&	11.84	&	0\%	&	0\%	\\
tc0c10s2ct1	&	34.89	&	9119.73	&	2	&	2	&	9.94	&	9.94	&	9.94	&	0\%	&	0\%	\\
tc0c10s3cf1	&	26.88	&	239.33	&	2	&	2	&	11.84	&	11.84	&	11.84	&	0\%	&	0\%	\\
tc0c10s3ct1	&	50.49	&	10800	&	2	&	2	&	10.01	&	10.01	&	10.01	&	0.04\%	&	0\%	\\
tc1c10s2cf2	&	57.31	&	1572.35	&	1	&	1	&	7.17	&	7.17	&	7.17	&	0\%	&	0\%	\\
tc1c10s2cf4	&	41.89	&	10800*	&	2	&	2	&	12.01	&	12.01	&	11.75	&	2.21\%	&	0\%	\\
tc1c10s2ct2	&	74.09	&	4604.89	&	1	&	1	&	7.36	&	7.36	&	7.36	&	0\%	&	0\%	\\
tc1c10s2ct3	&	41.61	&	10800*	&	1	&	1	&	10.32	&	10.32	&	9.29	&	9.99\%	&	0\%	\\
tc1c10s2ct4	&	43.42	&	10800*	&	2	&	2	&	12.17	&	12.17	&	11.92	&	2.06\%	&	0\%	\\
tc1c10s3cf2	&	56.56	&	507.94	&	1	&	1	&	7.14	&	7.14	&	7.14	&	0\%	&	0\%	\\
tc1c10s3cf4	&	50.81	&	10800*	&	2	&	2	&	12.20	&	12.20	&	11.32	&	7.24\%	&	0\%	\\
tc1c10s3ct2	&	64.47	&	10800*	&	1	&	1	&	7.45	&	7.45	&	7.12	&	4.38\%	&	0\%	\\
tc1c10s3ct3	&	69.16	&	10800*	&	1	&	1	&	10.00	&	10.47	&	6.73	&	35.70\%	&	$-4.45\%$		\\
tc1c10s3ct4	&	66.45	&	10800*	&	2	&	2	&	12.02	&	12.02	&	11.23	&	6.57\%	&	0\%	\\
tc2c10s2cf0	&	28.11	&	139.37	&	2	&	2	&	8.16	&	8.16	&	8.16	&	0\%	&	0\%	\\
tc2c10s2ct0	&	62.22	&	10800*	&	2	&	2	&	7.54	&	7.54	&	7.21	&	4.32\%	&	0\%	\\
tc2c10s3cf0	&	29.53	&	575.31	&	2	&	2	&	8.01	&	8.01	&	8.01	&	0\%	&	0\%	\\
tc2c10s3ct0	&	69.14	&	10800*	&	2	&	2	&	7.47	&	7.47	&	6.90	&	7.64\%	&	0\%	\\
    \hline
  \end{tabular}}
  \label{tab_sevrp-stoch_c10_S10_exponential}
\end{table}

\begin{table}[ht!]
  \centering
  \caption{10 customers - $|\mathcal{S}| = 20$ - exponential distribution.}
    \resizebox{\textwidth}{!}{
  \begin{tabular}{|c|c|c|c|c|c|c|c|c|c|}
    \hline
    \multirow{3}{*}{Instance} & \multicolumn{2}{|c|}{Time (s)} & \multicolumn{2}{|c|}{Number of routes}  & \multicolumn{3}{|c|}{Objective} & \multicolumn{2}{|c|}{Gap} \\
    \cline{2-10}
    & \multirow{2}{*}{Algorithm}& \multirow{2}{*}{Model} & \multirow{2}{*}{Algorithm} & \multirow{2}{*}{Model} & \multirow{2}{*}{Algorithm} & \multicolumn{2}{|c|}{Model} & \multirow{2}{*}{Gurobi} & \multirow{2}{*}{Algorithm-Model}\\
    \cline{7-8}
    &  &  &  &  &  & Incumbent & LB &  & \\
    \hline
tc0c10s2cf1	&	38.54	&	10800*	&	2	&	2	&	11.86	&	11.86	&	11.85	&	0.15\%	&	0\%	\\
tc0c10s2ct1	&	51.73	&	10800*	&	2	&	2	&	9.97	&	9.97	&	9.93	&	0.40\%	&	0\%	\\
tc0c10s3cf1	&	40.92	&	10800*	&	2	&	2	&	12.08	&	12.08	&	11.70	&	3.10\%	&	0\%	\\
tc0c10s3ct1	&	83.79	&	10800*	&	2	&	2	&	9.96	&	9.96	&	9.89	&	0.69\%	&	0\%	\\
tc1c10s2cf2	&	107.00	&	10800*	&	1	&	1	&	7.18	&	7.18	&	6.93	&	3.50\%	&	0\%	\\
tc1c10s2cf4	&	72.84	&	10800*	&	2	&	2	&	12.08	&	12.08	&	11.32	&	6.25\%	&	0\%	\\
tc1c10s2ct2	&	137.79	&	10800*	&	1	&	1	&	7.41	&	7.41	&	6.91	&	6.82\%	&	0\%	\\
tc1c10s2ct3	&	69.36	&	10800*	&	1	&	1	&	10.78	&	10.78	&	6.92	&	35.76\%	&	0\%	\\
tc1c10s2ct4	&	81.84	&	10800*	&	2	&	2	&	12.08	&	12.08	&	11.38	&	5.82\%	&	0\%	\\
tc1c10s3cf2	&	113.55	&	10800*	&	2	&	2	&	7.61	&	7.61	&	7.47	&	1.92\%	&	0\%	\\
tc1c10s3cf4	&	84.10	&	10800*	&	2	&	2	&	12.14	&	12.14	&	11.06	&	8.91\%	&	0\%		\\
tc1c10s3ct2	&	158.51	&	10800*	&	2	&	2	&	7.81	&	8.30	&	6.22	&	25.10\%	&	$-5.91\%$		\\
tc1c10s3ct3	&	132.73	&	10800*	&	1	&	1	&	10.02	&	10.74	&	6.29	&	41.46\%	&	$-6.73\%$	\\
tc1c10s3ct4	&	101.84	&	10800*	&	2	&	2	&	11.97	&	12.16	&	10.46	&	13.99\%	&	$-1.56\%$		\\
tc2c10s2cf0	&	44.64	&	10800*	&	2	&	2	&	8.39	&	8.39	&	8.14	&	2.95\%	&	0\%	\\
tc2c10s2ct0	&	130.50	&	10800*	&	2	&	2	&	7.55	&	7.55	&	6.03	&	20.18\%	&	0\%	\\
tc2c10s3cf0	&	42.86	&	10800*	&	2	&	2	&	8.31	&	8.31	&	7.82	&	5.93\%	&	0\%	\\
tc2c10s3ct0	&	130.64	&	10800*	&	2	&	2	&	7.54	&	7.54	&	5.21	&	30.90\%	&	0\%	\\
    \hline
  \end{tabular}}
  \label{tab_sevrp-stoch_c10_S20_exponential}
\end{table}

\begin{table}[ht!]
  \centering
  \caption{10 customers - $|\mathcal{S}| = 50$ - exponential distribution.}
    \resizebox{\textwidth}{!}{
  \begin{tabular}{|c|c|c|c|c|c|c|c|c|c|}
    \hline
    \multirow{3}{*}{Instance} & \multicolumn{2}{|c|}{Time (s)} & \multicolumn{2}{|c|}{Number of routes}  & \multicolumn{3}{|c|}{Objective} & \multicolumn{2}{|c|}{Gap} \\
    \cline{2-10}
    & \multirow{2}{*}{Algorithm}& \multirow{2}{*}{Model} & \multirow{2}{*}{Algorithm} & \multirow{2}{*}{Model} & \multirow{2}{*}{Algorithm} & \multicolumn{2}{|c|}{Model} & \multirow{2}{*}{Gurobi} & \multirow{2}{*}{Algorithm-Model}\\
    \cline{7-8}
    &  &  &  &  &  & Incumbent & LB &  & \\
    \hline
tc0c10s2cf1	&	85.77	&	10800*	&	2	&	2	&	12.28	&	12.28	&	11.95	&	2.66\%	&	0\%	\\
tc0c10s2ct1	&	114.01	&	10800*	&	2	&	2	&	9.96	&	9.96	&	9.80	&	1.55\%	&	0\%	\\
tc0c10s3cf1	&	88.55	&	10800*	&	2	&	2	&	12.14	&	12.14	&	12.04	&	0.86\%	&	0\%	\\
tc0c10s3ct1	&	184.09	&	10800*	&	2	&	1	&	9.95	&	11.93	&	8.61	&	27.85\%	&	$-16.61\%$		\\
tc1c10s2cf2	&	272.67	&	10800*	&	2	&	2	&	7.64	&	7.64	&	6.98	&	8.59\%	&	$-0.05\%$		\\
tc1c10s2cf4	&	155.74	&	10800*	&	2	&	2	&	12.08	&	12.08	&	11.08	&	8.31\%	&	0\%	\\
tc1c10s2ct2	&	277.32	&	10800*	&	2	&	2	&	7.86	&	8.67	&	6.10	&	29.69\%	&	$-9.09\%$		\\
tc1c10s2ct3	&	155.11	&	10800*	&	2	&	2	&	11.62	&	12.65	&	6.26	&	50.51\%	&	$-8.13\%$		\\
tc1c10s2ct4	&	173.35	&	10800*	&	2	&	2	&	12.10	&	12.10	&	10.57	&	12.62\%	&	0\%	\\
tc1c10s3cf2	&	253.92	&	10800*	&	2	&	2	&	7.64	&	8.02	&	6.84	&	14.67\%	&	$-4.74\%$		\\
tc1c10s3cf4	&	183.11	&	10800*	&	2	&	2	&	12.08	&	13.14	&	9.95	&	24.25\%	&	$-5.54\%$		\\
tc1c10s3ct2	&	351.08	&	10800*	&	2	&	2	&	7.84	&	9.07	&	5.56	&	38.71\%	&	$-13.16\%$		\\
tc1c10s3ct3	&	245.48	&	10800*	&	1	&	1	&	10.08	&	13.77	&	5.33	&	61.31\%	&	$-26.77\%$		\\
tc1c10s3ct4	&	229.08	&	10800*	&	2	&	2	&	12.02	&	13.75	&	9.44	&	31.32\%	&	$-12.56\%$		\\
tc2c10s2cf0	&	92.67	&	10800*	&	3	&	3	&	10.87	&	10.87	&	10.43	&	4.04\%	&	0\%	\\
tc2c10s2ct0	&	292.45	&	10800*	&	2	&	2	&	7.57	&	8.01	&	4.52	&	43.62\%	&	$-5.60\%$		\\
tc2c10s3cf0	&	114.21	&	10800*	&	2	&	2	&	8.34	&	8.34	&	7.18	&	13.89\%	&	0\%	\\
tc2c10s3ct0	&	318.53	&	10800*	&	2	&	2	&	7.53	&	8.11	&	3.42	&	57.84\%	&	$-7.01\%$		\\
    \hline
  \end{tabular}}
  \label{tab_sevrp-stoch_c10_50_exponential}
\end{table}

\clearpage
\subsection{In-sample stability} \label{sec_appB_insample}
\begin{table}[ht!]
  \centering
  \caption{Approximate in-sample stability analysis based on the objective function value obtained from ILS-SP heuristic (Algorithm~\ref{alg:ILS-P}) for the 10-customer instances and $|\mathcal{S}| = \{5,10,20,50,100$\}. Three probability distributions for the energy consumption are considered.}
    \resizebox{\textwidth}{!}{
  \begin{tabular}{|c|c|c|c|c|c||c|c|c|c|c||c|c|c|c|c|c}
    \hline
    \multirow{2}{*}{Instance} &  \multicolumn{5}{|c||}{Uniform} & \multicolumn{5}{|c||}{Normal} & \multicolumn{5}{|c|}{Exponential}
    \\
    \cline{2-16}
     &  5 & 10 & 20 & 50 & 100  & 5 & 10 & 20 & 50 & 100 & 5 & 10 & 20 & 50 & 100
     \\
    \hline
tc0c10s2cf1 & 10.33 & 11.97 & 12.42 & 12.40 & 12.45 & 11.68	&11.82	&11.99 & 12.11 & 12.08 &	{	11.85	}	&	{	11.84	}	& {11.86} & {12.28} & {12.24}\\
tc0c10s2ct1 & 10.04 & 9.99 & 9.98 &	10.02 & 10.16  & 9.84	& 9.84 &	9.91 & 10.11 & 10.11&	{	9.96	}	&	{	9.94	}	& {9.97} & {9.96} & {10.15}\\
tc0c10s3cf1 & 11.51 & 12.27 & 12.30 & 12.24 & 12.22 & 11.51&	11.82&	12.13 & 12.12 & 12.19 &	{	11.67	}	&	{	11.84	}	& {12.08} & {12.14} & {12.18}\\
tc0c10s3ct1 & 10.02 & 10.00 & 10.00 & 10.20 & 10.20 & 9.94&	10.08	&10.01 & 10.17 & 10.18 &	{	9.94	}	&	{	10.01	} & {9.96} & {9.95} & {9.93}	\\
tc1c10s2cf2 & 7.13 & 7.44 &	7.52 &	7.84 &	8.79 & 7.14&	7.37&	7.59 & 7.61 & 7.62&	{	7.12	}	&	{	7.17	}	 & {7.18} & {7.64} & {7.66} \\
tc1c10s2cf4 & 11.90 & 12.14 & 12.21 & 12.22 &  12.23 & 12.02	&12.34	&12.18 & 12.06 & 12.09 &	{	12.15	}	&	{	12.01	}	& {12.08} & {12.08} & {12.05} \\
tc1c10s2ct2 & 7.36 & 7.59 &	7.73 &	7.75 &	7.76  & 7.28&	7.60&	7.67 & 7.76 & 7.78 &	{	7.34	}	&	{	7.36	}	&{7.41} & {7.86} & {7.86}\\
tc1c10s2ct3 & 10.24 & 10.39 & 10.44 & 10.47 & 11.59 & 10.83	&10.57&	10.41 & 11.26 & 11.60&	{	10.76	}	&	{	10.32	}	&{10.78}& {11.62} & {11.62}\\
tc1c10s2ct4 & 11.95 & 12.00 & 12.01 & 12.04 & 12.06 & 12.07	&12.02&	12.05 & 12.08 & 12.08 &	{	12.08	}	&	{	12.17	} &{12.08}& {12.10} & {12.07}	\\
tc1c10s3cf2 & 7.43	& 7.51	& 7.65 &  7.63 & 7.65 & 7.19	&7.17	&7.56 & 7.61 & 7.63 &	{	7.12	}	&	{	7.14	} & {7.61} & {7.64} & {8.85}	\\
tc1c10s3cf4 & 12.12 & 12.00 & 12.08 & 12.11 & 12.17 & 12.22 &	12.24 &	12.09 & 12.16 & 12.19 &	{	12.16	}	&	{	12.20	}	&{12.14} & {12.08} & {12.10}\\
tc1c10s3ct2 & 7.55 & 7.71 & 7.69 & 7.77 & 7.78 & 7.39	& 7.37&	7.77 & 7.77 & 7.72 &	{	7.37	}	&		7.45		&{7.81}& {7.84} & {7.82}\\
tc1c10s3ct3 & 9.73 & 9.67 & 9.91 & 10.08 & 10.13 & 9.86	&10.06&	10.05 & 10.07 & 10.04&	{	9.88	}	&	{	10.00	} &{10.02}& {10.08} & {10.11}	\\
tc1c10s3ct4 & 12.13 & 12.10 & 12.05 & 12.05 &12.05 & 11.91	&11.99	&12.10 & 12.09 & 12.04 &	{	12.02	}	&	{	12.02	}	&{11.97}&{12.02}& {12.04}\\
tc2c10s2cf0 & 7.75 & 8.02 & 8.09 & 11.14 & 11.09 & 8.25	&8.59&	11.05 & 11.07 & 11.14&	{	7.79	}	&	{	8.16	}	&{8.39}& {10.87}& {10.94}\\
tc2c10s2ct0 & 7.48 & 7.51 & 7.52 & 7.65 & 7.66 & 7.55	&7.55&	7.55 & 7.62 & 7.62&	{	7.50	}	&	{	7.54	}	&{7.55}& {7.57} & {7.63} \\
tc2c10s3cf0 & 8.25 & 8.15 & 11.06 & 11.07 & 11.17 & 8.22	&8.24&	11.14 & 11.23 & 11.19&	{	7.97	}	&	{	8.01	} & {8.31}& {8.34}& {11.10}	\\
tc2c10s3ct0 & 7.57 & 7.55 & 7.60 & 7.68 & 7.67 & 7.58	&7.58&	7.62 & 7.59 & 7.59&	{	7.46	}	&	{		7.47} & {7.54} & {7.53} & {7.52}	\\
\hline
  \end{tabular}}
  \label{tab_insample_stability_analysis}
\end{table}

\subsection{Stochastic measures} \label{sec_appB_stochmeasure}

\begin{table}[ht!]
  \centering
  \caption{Stochastic measures for the 10-customer instances and $|\mathcal{S}| = 20$. Three probability distributions for the energy consumption are considered.}
      \resizebox{\textwidth}{!}{
  \begin{tabular}{|c|c|c|c|c|c|c||c|c|c|c|c|c||c|c|c|c|c|c|}
    \hline
    \multirow{2}{*}{Instance} &  \multicolumn{6}{|c||}{Uniform} &  \multicolumn{6}{|c||}{Normal} &  \multicolumn{6}{|c|}{Exponential}
    \\
    \cline{2-19}
     &RP &  WS & \%EVPI & EVP & EEV & \%VSS &  RP &  WS & \%EVPI & EVP & EEV & \%VSS & RP &  WS & \%EVPI & EVP & EEV & \%VSS
     \\
    \hline
    tc0c10s2cf1	&12.42&	10.48	&15.58\%	&9.99	&inf	&inf	&11.99&	10.60&	11.61\%	&10.01	&inf	&inf	&11.86&	10.32&	13.03\%	&10.16&	inf&	inf\\
    tc0c10s2ct1	&9.98&	9.91	&0.66\%	&9.83&	9.98&	0	&9.91&	9.87&	0.32\%	&9.83&	9.91 &	0	&9.97&	9.92&	0.51\%	&9.83	&9.97&	0\\
    tc0c10s3cf1 &	12.30	&10.51	&14.55\%	&10.03&	inf	&inf	&12.13	&10.55	&13.02\%	&9.95	&inf	&inf&	12.08	&10.29&	14.80\%	&10.00	&inf	&inf\\
    tc0c10s3ct1	&10.00	&9.91&	0.88\%&	9.83	&10.00	&0	&10.01&	9.95&	0.59\%	&9.84&	10.01	&0	&9.96&	9.86	&1.07\%	 &9.84&	9.96&	0\\
tc1c10s2cf2	&7.52	&7.12	&5.31\%	&7.09&	inf	&inf	&7.59	&7.10&	6.38\%	&7.08&	inf &	inf	&7.18&	7.08&	1.40\%	&7.07&	inf&	inf\\
tc1c10s2cf4	&12.21&	11.79&	3.37\%	&11.85	&inf	&inf	&12.18	&11.89&	2.40\%	&11.84 &	inf	&inf	&12.08&	11.77&	2.58\%	&11.79&	inf	&inf\\
tc1c10s2ct2	&7.73&	7.27	&5.91\%	&7.23&	inf&	inf&	7.67	&7.27	&5.14\%	&7.24&	inf&	inf	&7.41	&7.28	&1.81\%	&7.29	&inf	&inf\\
tc1c10s2ct3	&10.44&	9.83&	5.79\%	&10.45&	inf&	inf&	10.41&	9.77	&6.12\%	&10.41&	inf	&inf	&10.78&	9.33	&13.42\%	&10.43&	inf&	inf\\
tc1c10s2ct4&	12.01&	11.45	&4.65\%&	11.36	&inf	&inf	&12.05&	11.52&	4.38\%&	11.30&	inf	&inf	&12.08&	11.48&	4.92\%	&11.26&	inf&	inf\\
tc1c10s3cf2 &	7.64	&7.11	&6.94\%	&7.08&	inf&	inf	&7.56	&7.10&	6.13\%	&7.08&	inf	&inf	&7.61	&7.09	&6.83\%	&7.07	&inf	&inf\\
tc1c10s3cf4	&12.08&	11.86&	1.87\%	&11.83&	inf	&inf	&12.09	&11.88&	1.68\%	&11.87	&inf	&inf	&12.14&	11.85&	2.35\%	&11.82&	inf&	inf\\
tc1c10s3ct2&	7.69	&7.30&	5.07\%	&7.26	&inf	&inf&	7.77	&7.26&	6.52\%&	7.24&	inf	&inf&	7.81	&7.27	&6.98\%	&7.25&	inf&	inf\\
tc1c10s3ct3&	9.91	&9.27	&6.47\%	&9.66	&10.42&	5.10\%	&10.05	&9.47	&5.79\%	&9.91	&10.06&	0.09\%	&10.02	&9.30&	7.17\%	&9.88	&10.23&	2.18\%\\
tc1c10s3ct4&	12.05&	11.43	&5.14\%	&11.36	&inf	&inf	&12.10	&11.41	&5.67\%	&11.37	&inf	&inf	&11.97&	11.35&	5.22\%	&11.33&	inf	&inf\\
tc2c10s2cf0&	8.09	&7.51	&7.12\%	&8.24	&8.09	&0.03\%	&11.05&	7.68&	30.47\%&	8.20&	inf&	inf	&8.39	&7.33	&12.60\%&	8.22	&inf	&inf\\
tc2c10s2ct0	&7.52	&7.35	&2.21\%&	7.49	&7.52	&0	&7.55&	7.41&	1.87\%	&7.44&	7.74	&2.53\%&	7.55&	7.33	&2.90\%	&7.47	&7.61&	0.81\%\\
tc2c10s3cf0	&11.06&	7.67&	30.63\%	&8.22&	inf&	inf&	11.14&	7.72	&30.66\%&	8.22&	inf&	inf	&8.31	&7.41	&10.82\%	&8.20	&inf	&inf\\
tc2c10s3ct0	&7.60&	7.43	&2.21\%	&7.52	&7.69&	1.16\%	&7.62&	7.46	&2.14\%&	7.58&	7.62 &	0	&7.54&	7.34&	2.63\%	&7.51&	7.54&	0\\
\hline
  \end{tabular}}
  \label{tab_stochastic_measures_10c_20scen}
\end{table}

\clearpage
\section{Detailed results for instances with 40, 80 customers} \label{sec_appendixC}
In the following, we report the results of the numerical experiments for the instances with 40 and 80 customers. Specifically, for the 40-customer instances, we consider an increasing number of scenarios ($|\mathcal{S}| = \{1, 10, 20, 50\}$) and exploring the case where an initial set of 50 scenarios is reduced to 20 through the FFS algorithm (Algorithm \ref{alg:FFS}). For the instances with 80 customers, we set $|\mathcal{S}| = 20$. All the results are obtained through ILS-SP heuristic (Algorithm~\ref{alg:ILS-P}) with a maximum number of iterations equal to 2000 and a runtime limit of 18000 seconds. The asterisk indicates that the time limit has been reached. Uniform distribution on the energy consumption is considered.

\begin{table}[ht!]
  \centering
  \caption{40 customers - $|\mathcal{S}|=\{1, 10\}$.}
  \resizebox{\textwidth}{!}{
  \begin{tabular}{|c|c|c|c|c||c|c|c|c|}
    \hline
    \multirow{2}{*}{Instance} & \multicolumn{4}{|c||}{$|S|=1$} & \multicolumn{4}{|c|}{$|S|=10$}\\
    \cline{2-9}
     & Time (s) & Number of iteration & Number of routes & Objective  & Time (s) & Number of iteration & Number of routes & Objective\\
    \hline
tc0c40s5cf0	&	661.19	&	2000	&	3	&	22.11	&	1417.58	&	2000	&	3	&	22.69\\
tc0c40s5cf4	&	686.12	&	2000	&	2	&	20.01	&	1550.06	&	2000	&	2	&	20.03\\
tc0c40s5ct0	&	669.10	&	2000	&	2	&	20.35	&	1484.50	&	2000	&	2	&	21.49\\
tc0c40s5ct4	&	692.58	&	2000	&	1	&	19.07	&	1510.30	&	2000	&	2	&	20.06\\
tc0c40s8cf0	&	593.09	&	2000	&	3	&	21.85	&	1959.07	&	2000	&	3	&	22.30\\
tc0c40s8cf4	&	1164.79	&	2000	&	3	&	17.96	&	8020.82	&	2000	&	4	&	18.56\\
tc0c40s8ct0	&	521.03	&	2000	&	1	&	19.51	&	2858.97	&	2000	&	2	&	19.84\\
tc0c40s8ct4	&	567.34	&	2000	&	1	&	18.64	&	2638.43	&	2000	&	1	&	18.90\\
tc1c40s5cf1	&	1261.19	&	2000	&	2	&	21.72	&	2165.60	&	2000	&	3	&	23.24\\
tc1c40s5ct1	&	1230.07	&	2000	&	2	&	22.53	&	2342.72	&	2000	&	2	&	22.49\\
tc1c40s8cf1	&	804.06	&	2000	&	1	&	20.52	&	3922.73	&	2000	&	2	&	20.94\\
tc1c40s8ct1	&	691.05	&	2000	&	1	&	19.41	&	2796.19	&	2000	&	2	&	19.92\\
tc2c40s5cf2	&	853.37	&	2000	&	1	&	15.95	&	2352.17	&	2000	&	1	&	16.28\\
tc2c40s5cf3	&	703.74	&	2000	&	2	&	13.40	&	2086.98	&	2000	&	2	&	13.56\\
tc2c40s5ct2	&	530.02	&	2000	&	2	&	14.52	&	2413.35	&	2000	&	1	&	15.21\\
tc2c40s5ct3	&	497.27	&	2000	&	2	&	13.91	&	1073.97	&	2000	&	2	&	14.14\\
tc2c40s8cf2	&	812.41	&	2000	&	1	&	15.64	&	3816.96	&	2000	&	1	&	15.53\\
tc2c40s8cf3	&	697.58	&	2000	&	3	&	13.32	&	3043.14	&	2000	&	3	&	13.41\\
tc2c40s8ct2	&	514.17	&	2000	&	2	&	14.50	&	2695.07	&	2000	&	1	&	14.25\\
tc2c40s8ct3	&	483.33	&	2000	&	3	&	12.98	&	1726.36	&	2000	&	3	&	13.04\\
\hline
  \end{tabular}}
    \label{tab_sevrp-stochastic_40c_s1s10}
\end{table}

\begin{table}[ht!]
  \centering
  \caption{40 customers - $|\mathcal{S}|=\{20, 50\}$.}
  \resizebox{\textwidth}{!}{
  \begin{tabular}{|c|c|c|c|c||c|c|c|c|}
    \hline
    \multirow{2}{*}{Instance} & \multicolumn{4}{|c||}{$|S|=20$} & \multicolumn{4}{|c|}{$|S|=50$}\\
    \cline{2-9}
     & Time (s) & Number of iteration & Number of routes & Objective  & Time (s) & Number of iteration & Number of routes & Objective\\
    \hline
tc0c40s5cf0	&	2174.05	&	2000	&	4	&	22.68	&	10800*	&	1392	&	4	&	22.79	\\
tc0c40s5cf4	&	2216.35	&	2000	&	2	&	20.35	&	10800*	&	1761	&	3	&	22.61	\\
tc0c40s5ct0	&	2201.56	&	2000	&	3	&	21.44	&	10800*	&	1875	&	3	&	21.86	\\
tc0c40s5ct4	&	2336.83	&	2000	&	2	&	21.00	&	10800*	&	1694	&	2	&	20.98	\\
tc0c40s8cf0	&	3048.14	&	2000	&	3	&	22.17	&	10800*	&	1355	&	4	&	22.61	\\
tc0c40s8cf4	&	10800*	&	1380	&	3	&	19.03	&	10800*	&	199	&	3	&	19.06	\\
tc0c40s8ct0	&	4429.88	&	2000	&	2	&	19.88	&	10800*	&	1095	&	2	&	20.25	\\
tc0c40s8ct4	&	3862.97	&	2000	&	1	&	20.06	&	10800*	&	1057	&	2	&	21.42	\\
tc1c40s5cf1	&	3340.73	&	2000	&	3	&	23.16	&	10800*	&	1291	&	4	&	28.70	\\
tc1c40s5ct1	&	3680.40	&	2000	&	3	&	23.74	&	10800*	&	1171	&	2	&	24.77	\\
tc1c40s8cf1	&	6042.46	&	2000	&	2	&	20.94	&	10800*	&	771	&	2	&	21.04	\\
tc1c40s8ct1	&	4351.24	&	2000	&	2	&	20.23	&	10800*	&	989	&	1	&	21.39	\\
tc2c40s5cf2	&	3022.32	&	2000	&	3	&	17.00	&	10800*	&	1521	&	2	&	17.25	\\
tc2c40s5cf3	&	2810.18	&	2000	&	2	&	13.60	&	10800*	&	1674	&	3	&	14.52	\\
tc2c40s5ct2	&	4049.67	&	2000	&	1	&	15.62	&	10800*	&	1184	&	1	&	16.95	\\
tc2c40s5ct3	&	1569.56	&	2000	&	2	&	14.10	&	7077.51	&	2000	&	2	&	14.26	\\
tc2c40s8cf2	&	7030.63	&	2000	&	1	&	15.62	&	10800*	&	580	&	1	&	15.75	\\
tc2c40s8cf3	&	4888.71	&	2000	&	3	&	13.43	&	10800*	&	962	&	3	&	13.41	\\
tc2c40s8ct2	&	5059.92	&	2000	&	1	&	14.34	&	10800*	&	713	&	1	&	14.42	\\
tc2c40s8ct3	&	2529.91	&	2000	&	3	&	13.09	&	10800*	&	1758	&	3	&	13.12	\\
\hline
  \end{tabular}}
    \label{tab_sevrp-stochastic_40c_s20s50}
\end{table}

\begin{table}[ht!]
  \centering
  \caption{40 customers - $|\mathcal{S}|=20$, reduced from an initial set of 50 scenarios.}
  \resizebox{0.7\textwidth}{!}{
  \begin{tabular}{|c|c|c|c|c|}
    \hline
    Instance & Time (s) & Number of iteration & Number of routes & Objective \\
    \hline
tc0c40s5cf0	&	2193.93 & 2000 & 4 & 22.38\\
tc0c40s5cf4	&	2440.75 & 2000 & 2 & 23.84\\
tc0c40s5ct0	&	2321.57 & 2000 & 2 & 21.15\\
tc0c40s5ct4	&	2303.74 & 2000 & 2 & 20.96\\
tc0c40s8cf0	&	3272.41 & 2000 & 3 & 22.41\\
tc0c40s8cf4	&   10800* &  1380 & 3	& 18.78\\
tc0c40s8ct0	&   4656.01 & 2000 & 2 & 20.26\\
tc0c40s8ct4	&	3796.65 & 2000 & 1 & 19.63\\
tc1c40s5cf1	&	3401.69 & 2000 & 4 & 25.78\\
tc1c40s5ct1	&	3532.22 & 2000 & 3 & 23.61\\
tc1c40s8cf1	&	6888.19 & 2000 & 2 & 20.66\\
tc1c40s8ct1	&	4069.98 & 2000 & 2 & 20.24\\
tc2c40s5cf2	&	3206.99 & 2000 & 3 & 16.99\\
tc2c40s5cf3	&	2963.66 & 2000 & 2 & 14.49\\
tc2c40s5ct2	&	3785.36 & 2000 & 1 & 15.24\\
tc2c40s5ct3	&	1459.03 & 2000 & 2 & 14.24\\
tc2c40s8cf2	&	7063.48 & 2000 & 1 & 15.71\\
tc2c40s8cf3	&	4916.70 & 2000 & 3 & 13.43\\
tc2c40s8ct2	&	4899.15 & 2000 & 1 & 14.42\\
tc2c40s8ct3	&	2616.29 & 2000 & 3 & 13.11
\\
\hline
  \end{tabular}}
    \label{tab_sevrp-stochastic_40c_50sr20}
\end{table}

\newpage
\clearpage 

\begin{table}[ht!]
  \centering
  \caption{80 customers - $|\mathcal{S}|=\{20, 50\}$.}
  \resizebox{\textwidth}{!}{
  \begin{tabular}{|c|c|c|c|c||c|c|c|c|}
    \hline
    \multirow{2}{*}{Instance} & \multicolumn{4}{|c||}{$|S|=20$} & \multicolumn{4}{|c|}{$|S|=50$}\\
    \cline{2-9}
     & Time (s) & Number of iteration & Number of routes & Objective  & Time (s) & Number of iteration & Number of routes & Objective\\
    \hline
tc0c80s8cf0	&	10800*	&	736	&	5	&	34.60	&	10800*	&	404	&	4	&	33.32	\\
tc0c80s8cf1	&	10800*	&	749	&	3	&	39.05	&	10800*	&	294	&	2	&	36.29	\\
tc0c80s8ct0	&	10800*	&	744	&	5	&	34.02	&	10800*	&	417	&	4	&	38.17	\\
tc0c80s8ct1	&	10800*	&	646	&	4	&	40.71	&	10800*	&	325	&	4	&	41.95	\\
tc0c80s12cf0	&	10800*	&	159	&	3	&	29.37	&	10800*	&	90	&	2	&	29.50	\\
tc0c80s12cf1	&	10800*	&	254	&	3	&	32.16	&	10800*	&	130	&	2	&	32.52	\\
tc0c80s12ct0	&	10800*	&	402	&	3	&	28.20	&	10800*	&	53	&	2	&	32.65	\\
tc0c80s12ct1	&	10800*	&	233	&	2	&	30.73	&	10800*	&	110	&	2	&	32.33	\\
tc1c80s8cf2	&	10800*	&	122	&	1	&	22.13	&	10800*	&	48	&	2	&	26.52	\\
tc1c80s8ct2	&	10800*	&	203	&	2	&	22.84	&	10800*	&	253	&	1	&	24.67	\\
tc1c80s12cf2	&	10800*	&	182	&	1	&	22.29	&	10800*	&	69	&	3	&	23.16	\\
tc1c80s12ct2	&	10800*	&	193	&	1	&	20.88	&	10800*	&	92	&	2	&	23.14	\\
tc2c80s8cf3	&	10800*	&	559	&	3	&	23.55	&	10800*	&	302	&	4	&	25.66	\\
tc2c80s8cf4	&	10800*	&	172	&	2	&	24.59	&		10800*	&	57	&	4	&	26.07	\\
tc2c80s8ct3	&	10800*	&	524	&	3	&	23.32	&		10800*	&	400	&	3	&	25.83	\\
tc2c80s8ct4	&	10800*	&	168	&	4	&	25.05	&		10800*	&	71	&	3	&	25.76	\\
tc2c80s12cf3	&	10800*	&	145	&	3	&	21.22	&		10800*	&	110	&	4	&	22.83	\\
tc2c80s12cf4	&	10800*	&	250	&	2	&	21.69	&		10800*	&	87	&	2	&	21.34	\\
tc2c80s12ct3	&	10800*	&	168	&	3	&	20.50	&		10800*	&	88	&	2	&	21.93	\\
tc2c80s12ct4	&	10800*	&	251	&	2	&	20.17	&		10800*	&	87	&	1	&	20.12	\\
\hline
  \end{tabular}}
    \label{tab_sevrp-stochastic_80c_s20s50}
\end{table}

\begin{table}[ht!]
  \centering
  \caption{80 customers - $|\mathcal{S}|=20$, reduced from an initial set of 50 scenarios.}
  \resizebox{0.7\textwidth}{!}{
  \begin{tabular}{|c|c|c|c|c|}
    \hline
     Instance & Time (s) & Number of iteration & Number of routes & Objective \\
    \hline
tc0c80s8cf0	&	10800*	&	842	&	3	&	34.58	\\
tc0c80s8cf1	&	10800*	&	637	&	3	&	33.01	\\
tc0c80s8ct0	&	10800*	&	657	&	6	&	36.35	\\
tc0c80s8ct1	&	10800*	&	638	&	2	&	39.62	\\
tc0c80s12cf0	&	10800*	&	178	&	3	&	31.00	\\
tc0c80s12cf1	&	10800*	&	265	&	3	&	31.58	\\
tc0c80s12ct0	&	10800*	&	141	&	3	&	30.37	\\
tc0c80s12ct1	&	10800*	&	281	&	3	&	31.39	\\
tc1c80s8cf2	&	10800*	&	144	&	3	&	25.08	\\
tc1c80s8ct2	&	10800*	&	356	&	3	&	25.44	\\
tc1c80s12cf2	&	10800*	&	159	&	1	&	21.67	\\
tc1c80s12ct2	&	10800*	&	254	&	1	&	20.98	\\
tc2c80s8cf3	&	10800*	&	554	&	4	&	24.32	\\
tc2c80s8cf4	&	10800*	&	158	&	2	&	23.72	\\
tc2c80s8ct3	&	10800*	&	711	&	3	&	25.56	\\
tc2c80s8ct4	&	10800*	&	159	&	2	&	23.79	\\
tc2c80s12cf3	&	10800*	&	396	&	4	&	22.01	\\
tc2c80s12cf4	&	10800*	&	239	&	2	&	21.69	\\
tc2c80s12ct3	&	10800*	&	246	&	2	&	20.46	\\
tc2c80s12ct4	&	10800*	&	202	&	2	&	20.60	\\
\hline
  \end{tabular}}
\label{tab_sevrp-stochastic_80c_50sr20}
\end{table}

\clearpage
\section{Detailed results for instances with 40 customers - \\ managerial insights} \label{sec_appendixD_managerial}
In the following, we report the results of the numerical experiments for the instances with 40 customers and different combinations of parameters. Specifically, we consider $Q^T=\alpha^TQ^{max}$, with $\alpha^T=\{20\%, 30\%, 40\%\}$, and $Q^G=\alpha^GQ^{max}$, with $\alpha^G=\{70\%, 80\%, 90\%\}$. The number of scenarios is set to 20, reduced from an initial set of 50 through the FFS algorithm (Algorithm \ref{alg:FFS}). All the results are obtained through ILS-SP heuristic (Algorithm~\ref{alg:ILS-P}) with a maximum number of iterations equal to 2000 and a runtime limit of 18000 seconds. Uniform distribution on the energy consumption is considered.

\begin{table}[ht!]
  \centering
  \caption{40 customers - analysis on $Q^T=\alpha^TQ^{max}$.}
  \resizebox{0.87\textwidth}{!}{
  \begin{tabular}{|c|c|c|c|c|c|c|}
    \hline
    \multirow{3}{*}{Instance} & \multicolumn{6}{|c|}{$\alpha^T$}\\
    & \multicolumn{2}{|c|}{$20\%$} & \multicolumn{2}{|c|}{$30\%$} & \multicolumn{2}{|c|}{$40\%$}\\
    \cline{2-7}
     & Number of routes & Objective  & Number of routes & Objective  & Number of routes & Objective \\
    \hline
tc0c40s5cf0	&	4	&	22.59	&	4	&	22.38	&	4	&	25.12	\\
tc0c40s5cf4	&	2	&	20.11	&	2	&	23.84	&	2	&	23.27	\\
tc0c40s5ct0	&	3	&	20.41	&	2	&	21.15	&	3	&	23.30	\\
tc0c40s5ct4	&	3	&	21.33	&	2	&	20.96	&	2	&	22.65	\\
tc0c40s8cf0	&	4	&	21.94	&	3	&	22.41	&	2	&	25.14	\\
tc0c40s8cf4	&	2	&	17.99	&	3	&	18.78	&	3	&	19.80	\\
tc0c40s8ct0	&	2	&	18.92	&	2	&	20.26	&	2	&	21.40	\\
tc0c40s8ct4	&	2	&	18.91	&	1	&	19.63	&	2	&	21.48	\\
tc1c40s5cf1	&	3	&	27.20	&	4	&	25.78	&	2	&	26.43	\\
tc1c40s5ct1	&	3	&	25.10	&	3	&	23.61	&	3	&	29.87	\\
tc1c40s8cf1	&	1	&	20.88	&	2	&	20.66	&	2	&	22.43	\\
tc1c40s8ct1	&	2	&	22.06	&	2	&	20.24	&	1	&	23.00	\\
tc2c40s5cf2	&	3	&	17.54	&	3	&	16.99	&	1	&	17.77	\\
tc2c40s5cf3	&	3	&	14.64	&	2	&	14.49	&	2	&	14.81	\\
tc2c40s5ct2	&	1	&	15.19	&	1	&	15.24	&	1	&	16.22	\\
tc2c40s5ct3	&	2	&	14.35	&	2	&	14.24	&	3	&	14.71	\\
tc2c40s8cf2	&	2	&	15.94	&	1	&	15.71	&	2	&	16.46	\\
tc2c40s8cf3	&	3	&	13.79	&	3	&	13.43	&	2	&	13.26	\\
tc2c40s8ct2	&	1	&	14.52	&	1	&	14.42	&	1	&	15.07	\\
tc2c40s8ct3	&	3	&	13.30	&	3	&	13.11	&	2	&	13.13	\\
\hline
  \end{tabular}}
    \label{tab_sevrp-stochastic_detailed_Qt}
\end{table}

\begin{table}[ht!]
  \centering
  \caption{40 customers - analysis on $Q^G=\alpha^G Q^{max}$.}
  \resizebox{0.87\textwidth}{!}{
  \begin{tabular}{|c|c|c|c|c|c|c|}
    \hline
    \multirow{3}{*}{Instance} & \multicolumn{6}{|c|}{$\alpha^G$}\\
    & \multicolumn{2}{|c|}{$70\%$} & \multicolumn{2}{|c|}{$80\%$} & \multicolumn{2}{|c|}{$90\%$}\\
    \cline{2-7}
     & Number of routes & Objective  & Number of routes & Objective  & Number of routes & Objective \\
    \hline
tc0c40s5cf0	&	4	&	22.55	&	4	&	22.38	&	6	&	24.76	\\
tc0c40s5cf4	&	2	&	21.20	&	2	&	23.84	&	4	&	23.67	\\
tc0c40s5ct0	&	2	&	21.21	&	2	&	21.15	&	4	&	22.87	\\
tc0c40s5ct4	&	2	&	21.02	&	2	&	20.96	&	3	&	22.42	\\
tc0c40s8cf0	&	2	&	22.21	&	3	&	22.41	&	5	&	24.99	\\
tc0c40s8cf4	&	2	&	18.46	&	3	&	18.78	&	3	&	19.03	\\
tc0c40s8ct0	&	2	&	19.64	&	2	&	20.26	&	3	&	21.46	\\
tc0c40s8ct4	&	2	&	20.21	&	1	&	19.63	&	2	&	22.68	\\
tc1c40s5cf1	&	3	&	23.36	&	4	&	25.78	&	$-$	&	inf	\\
tc1c40s5ct1	&	3	&	23.11	&	3	&	23.61	&	4	&	27.56	\\
tc1c40s8cf1	&	2	&	21.49	&	2	&	20.66	&	2	&	23.61	\\
tc1c40s8ct1	&	1	&	21.01	&	2	&	20.24	&	2	&	26.08	\\
tc2c40s5cf2	&	3	&	16.83	&	3	&	16.99	&	3	&	19.43	\\
tc2c40s5cf3	&	2	&	13.57	&	2	&	14.49	&	3	&	14.69	\\
tc2c40s5ct2	&	1	&	15.64	&	1	&	15.24	&	3	&	17.83	\\
tc2c40s5ct3	&	2	&	14.35	&	2	&	14.24	&	2	&	14.39	\\
tc2c40s8cf2	&	2	&	16.18	&	1	&	15.71	&	4	&	18.39	\\
tc2c40s8cf3	&	2	&	13.28	&	3	&	13.43	&	3	&	13.53	\\
tc2c40s8ct2	&	1	&	14.73	&	1	&	14.42	&	1	&	16.99	\\
tc2c40s8ct3	&	1	&	13.07	&	3	&	13.11	&	3	&	13.18	\\
\hline
  \end{tabular}}
    \label{tab_sevrp-stochastic_detailed_Qg}
\end{table}

\end{document}